%% file: Main.tex
\pdfoutput=1	
\documentclass[
	final%
	,12pt%
	,pagesize%
	,headings=normal%
	,paper=a4%
	,parskip=false%
	,headsepline=true%
	,abstract=true%
]{scrartcl}

\addtokomafont{sectioning}{\normalfont\bfseries}
\setkomafont{title}{\normalfont}
\setkomafont{subtitle}{\normalfont}
\setkomafont{author}{\normalfont}
\setkomafont{date}{\normalfont}
\addtokomafont{pageheadfoot}{\scshape\small}
\setkomafont{caption}{\footnotesize}

\setkomafont{captionlabel}{\usekomafont{caption}\bfseries}
\setcapindent{0pt} 
\usepackage{etoolbox}
\newtoggle{narrowpage}
\settoggle{narrowpage}{false}

\AtBeginEnvironment{abstract}{\footnotesize}

\usepackage{xpatch}
\makeatletter
\patchcmd{\@maketitle}{\huge}{\Large}{}{}
\makeatother

\usepackage[utf8]{inputenc}
\usepackage[T1]{fontenc}
\usepackage{csquotes}
\usepackage[english]{babel}
\usepackage{caption}
\usepackage{subcaption}
\usepackage[final]{graphicx}
\makeatletter
\def\input@path{{Pictures/}}
\makeatother
\graphicspath{%
{./}%
{./Pictures/}%
{./Pictures/FreedmanHeWang_1648E-Hs-ProjGrad/}%
{./Pictures/HelixKnot_1940E-Hs-ProjGrad/}%
}
\usepackage[usenames,dvipsnames]{xcolor}
\usepackage{tikz-cd}
\usepackage[draft]{fixme}
\usepackage[nointlimits]{amsmath}
\usepackage[varg]{txfonts}
\usepackage{amssymb,mathtools,mathrsfs}
\usepackage{upgreek}
\usepackage{braket}
\usepackage{cancel}
\usepackage{siunitx}
\usepackage{aliascnt} %
\usepackage[amsmath,thmmarks,hyperref,thref]{ntheorem}
\usepackage[expansion=true,protrusion=true]{microtype}
\usepackage{xkeyval}
\usepackage[normalem]{ulem}

\setkomafont{caption}{\footnotesize}
\setkomafont{captionlabel}{\usekomafont{caption}}

\usepackage[final,%
	pdftex,%
	bookmarks,%
	bookmarksdepth=3,%
	breaklinks=true,%
	colorlinks=true,%
	urlcolor=NavyBlue,%
	linkcolor=NavyBlue,%
	citecolor=ForestGreen,%
]{hyperref}%
\usepackage[all]{hypcap}

\addto\extrasenglish{%
}

\usepackage{xparse}
\usepackage{xspace}

\usepackage[nocompress]{cite}

\input{Macros}

\input{ntheoremenglish}

\title{Sobolev Gradients for the Möbius Energy}
\author{Philipp Reiter\thanks{Chemnitz University of Technology,
Faculty of Mathematics, 09107 Chemnitz, Germany, 
\href{mailto:reiter@math.tu-chemnitz.de}{reiter@math.tu-chemnitz.de}
} {} and Henrik Schumacher\thanks{(corresponding author) Institute for Mathematics, RWTH Aachen University, Templergraben 55, 52062 Aachen, Germany, \href{mailto:schumacher@instmath.rwth-aachen.de
}{schumacher@instmath.rwth-aachen.de}}}

\begin{document}

\maketitle

\begin{abstract}
\begin{small}
\input{Abstract}
\\

\noindent
\textbf{MSC-2020 classification:} 
  49Q10; %
  53A04; %
  58D10  %
\end{small}
\end{abstract}

\input{Introduction}

\input{Notation}

\input{Energy}

\input{Metric}

\input{Constraint}

\input{NumericalExperiments}

\appendix

\input{Appendix}

\input{Acknowledgments}

\bibliographystyle{abbrvhref}
\begin{small}
\setlength{\parskip}{0.5ex}
\bibliography{Literature}
\end{small}
\end{document}

%% file: Macros.tex
\newcommand{\Diff}[2]{D^{#1}_{#2}}
\newcommand{\diff}[2]{\delta^{#1}_{#2}}
\newcommand{\multop}{H}
\newcommand{\singularmeasure}{\mu}
\newcommand{\mass}{M}

\newcommand{\isometry}{f}
\newcommand{\Geodesic}{\eta}
\newcommand{\prx}{\pi_1}
\newcommand{\pry}{\pi_2}

\newcommand{\Energy}{\cE}
\newcommand{\Triangulation}{\cT}
\newcommand{\subT}{{\cT}}
\newcommand{\Polygon}{P}
\newcommand{\OtherPolygon}{Q}
\newcommand{\Edges}{E}
\newcommand{\Vertices}{V}
\newcommand{\Edge}{I}
\newcommand{\OtherEdge}{J}

\newcommand{\EdgeLengths}{\ell}

\newcommand{\MaxRadius}{h}

\newcommand{\conf}[1]{\bigparen{\sigma_{#1}- \tfrac{\DomDim}{p_{#1}}}}

\newcommand{\Circle}{{\mathbb{T}}}
\newcommand{\Torus}{{\Circle^2}}

\newcommand{\AmbDim}{m}
\newcommand{\AmbSpace}{{\R^{\AmbDim}}}

\newcommand{\DomDim}{n}

\newcommand{\lot}{\on{l.o.t.}}

\DeclareDocumentCommand{\LineElement}{ o o }{
	\IfValueTF{#1}{
	  \omega_{#1}\IfValueTF{#2}{(#2)}{}
	}{
	  \omega\IfValueTF{#2}{(#2)}{}
	}
}

\DeclareDocumentCommand{\LineElementC}{ o }{
	\IfValueTF{#1}{
	  \omega_\Curve(#1)
	}{
	  \omega_\Curve
	}
}

\newcommand{\Tangent}{\tau}
\newcommand{\Curve}{\gamma}

\newcommand{\OtherCurve}{\eta}
\newcommand{\SaddlePointMatrix}{\mathcal{A}}

\newcommand{\ConfSpace}{\mathcal{C}}
\newcommand{\TargetSpace}{\mathcal{N}}
\newcommand{\ConstraintMap}{\varPhi}
\newcommand{\ConstraintMfld}{\cM}

\newcommand{\Morphism}{F}

\newcommand{\hilberts}{{s}} 				%
\newcommand{\hilbertss}{{\sigma}}				%
\newcommand{\diffoffset}{{\nu}}

\newcommand{\strongs}{{\hilberts+\diffoffset}}		%
\newcommand{\strongss}{{\hilbertss+\diffoffset}}	%
\newcommand{\strongp}{p}								%

\newcommand{\weaks}{{\hilberts-\diffoffset}}		%
\newcommand{\weakss}{{\hilbertss-\diffoffset}}		%
\newcommand{\weakp}{q}								%

\newcommand{\Xnorm}[1]{\cX,#1}
\newcommand{\Hnorm}[1]{\cH,#1}
\newcommand{\Ynorm}[1]{\cY,#1}

\newcommand{\XC}{\cX}
\newcommand{\XCg}{\cX}
\newcommand{\XCgd}{\cX\dual}
\newcommand{\XCd}{\cX\dual}

\newcommand{\HC}{\cH}
\newcommand{\HCg}{\cH}
\newcommand{\HCgd}{\cH\dual}
\newcommand{\HCd}{\cH\dual}

\newcommand{\YC}{\cY}
\newcommand{\YCg}{\cY}
\newcommand{\YCgd}{\cY\dual}
\newcommand{\YCd}{\cY\dual}

\newcommand{\ZCg}{\cZ}

\newcommand{\XN}{\cX\TargetSpace}
\newcommand{\XNg}{\cX\TargetSpace}

\newcommand{\HN}{\cH\TargetSpace}
\newcommand{\HNg}{\cH\TargetSpace}
\newcommand{\HNgd}{\cH\dual\TargetSpace}
\newcommand{\HNd}{\cH\dual\TargetSpace}

\newcommand{\YN}{\cY\TargetSpace}
\newcommand{\YNg}{\cY\TargetSpace}
\newcommand{\YNgd}{\cY\dual\TargetSpace}
\newcommand{\YNd}{\cY\dual\TargetSpace}

\newcommand{\ZNg}{\cZ\TargetSpace}

\DeclareDocumentCommand{\RX}{ O{\,} O{\,} O{} }{\cX_{#2}#3#1}
\DeclareDocumentCommand{\RH}{ O{\,} O{\,} O{} }{\cH_{#2}#3#1}
\DeclareDocumentCommand{\RY}{ O{\,} O{\,} O{} }{\cY_{#2}#3#1}
\DeclareDocumentCommand{\RT}{ O{\,} O{\,} O{} }{T_{#2}#3#1}
\DeclareDocumentCommand{\RZ}{ O{\,} O{\,} O{} }{\cZ_{#2}#3#1}
\DeclareDocumentCommand{\T}{ O{\,} O{} }{T_{#1}#2}

\DeclareDocumentCommand{\Ri}{ O{} }{i_{#1}}
\DeclareDocumentCommand{\Rj}{ O{} }{j_{#1}}
\DeclareDocumentCommand{\RI}{ O{} }{\cI_{#1}}
\DeclareDocumentCommand{\RJ}{ O{} }{\cJ_{#1}}
\DeclareDocumentCommand{\RK}{ O{} }{\cK_{#1}}
\DeclareDocumentCommand{\Rg}{ O{} }{g_{#1}}

\newcommand{\SoboSlobo}{Sobolev--Slobodecki\u{\i}\xspace}

\DeclareDocumentCommand{\Sobo}{ O{} O{} o o}{
	\IfValueTF{#3}{
	  \IfValueTF{#4}{
	  	W^{#1}_{#2}(#3;#4)
	  }{
	  	W^{#1}_{#2}(#3)
	  }
	}{
	  W^{#1}_{#2}
	}
}
\DeclareDocumentCommand{\SoboC}{ O{} O{} }{\Sobo[#1][#2][\Circle][\AmbSpace]}

\DeclareDocumentCommand{\Holder}{ O{} O{} o o}{
	\IfValueTF{#3}{
	  \IfValueTF{#4}{
	  	C^{#1}_{#2}(#3;#4)
	  }{
	  	C^{#1}_{#2}(#3)
	  }
	}{
	  C^{#1}_{#2}
	}
}
\DeclareDocumentCommand{\HolderC}{ O{} O{} }{\Holder[#1][#2][\Circle][\AmbSpace]}

\DeclareDocumentCommand{\Lebesgue}{ O{} O{} o o}{
	\IfValueTF{#3}{
	  \IfValueTF{#4}{
	  	L^{#1}_{#2}(#3;#4)
	  }{
	  	L^{#1}_{#2}(#3)
	  }
	}{
	  L^{#1}_{#2}
	}
}
\DeclareDocumentCommand{\LebesgueC}{ O{} O{} }{\Lebesgue[#1][#2][\Circle][\AmbSpace]}

\newcommand{\nospaceperiod}{\makebox[0pt][l]{\,.}}

\newcommand{\qand}{\quad \text{and} \quad}

\DeclareMathOperator{\essinf}{ess\,inf}

\DeclareMathOperator{\LandO}{O}

\DeclareMathOperator{\argmin}{arg\,min}

\newcommand{\adj}{^{*\!}}
\newcommand{\dual}{'^{\!}}

\newcommand{\pinv}{^{\dagger\!}}

\newcommand{\cB}{{\mathcal{B}}}

\newcommand{\cD}{{\mathcal{D}}}
\newcommand{\cE}{{\mathcal{E}}}
\newcommand{\cF}{{\mathcal{F}}}

\newcommand{\cH}{{\mathcal{H}}}
\newcommand{\cI}{{\mathcal{I}}}
\newcommand{\cJ}{{\mathcal{J}}}
\newcommand{\cK}{{\mathcal{K}}}

\newcommand{\cM}{{\mathcal{M}}}

\newcommand{\cT}{{\mathcal{T}}}
\newcommand{\cU}{{\mathcal{U}}}

\newcommand{\cX}{{\mathcal{X}}}
\newcommand{\cY}{{\mathcal{Y}}}
\newcommand{\cZ}{{\mathcal{Z}}}

\newcommand{\dd}{{\on{d}}}

\newcommand{\at}{|}

\newcommand{\pd}{\partial}

\newcommand{\ceq}{\coloneqq}
\newcommand{\qec}{\eqqcolon}
\newcommand{\R}{{\mathbb{R}}}

\newcommand{\N}{\mathbb{N}}
\newcommand{\Z}{{\mathbb{Z}}}

\DeclareMathOperator{\id}{id}

\newcommand{\abs}[1]{\left\lvert#1\right\rvert} %
\newcommand{\nabs}[1]{\lvert{#1}\rvert} %
\newcommand{\bigabs}[1]{\big\lvert{#1}\big\rvert} %
\newcommand{\Bigabs}[1]{\Big\lvert{#1}\Big\rvert} %

\newcommand{\nnorm}[1]{\lVert{#1}\rVert}
\newcommand{\bignorm}[1]{\big\lVert{#1}\big\rVert}

\newcommand{\innerprod}[1]{\left\langle #1 \right\rangle}
\newcommand{\ninnerprod}[1]{\langle #1 \rangle}
\newcommand{\biginnerprod}[1]{\big\langle #1\big\rangle}
\newcommand{\Biginnerprod}[1]{\Big\langle #1 \Big\rangle}

\newcommand{\seminorm}[1]{\left[#1\right]}
\newcommand{\nseminorm}[1]{[{#1}]}

\newcommand{\paren}[1]{\left(#1\right)}
\newcommand{\nparen}[1]{(#1)}
\newcommand{\bigparen}[1]{\big(#1\big)}

\newcommand{\Bigparen}[1]{\Big(#1\Big)}

\newcommand{\bigbrackets}[1]{\big[#1\big]}

\newcommand{\Bigbrackets}[1]{\Big[#1\Big]}

\newcommand{\intervaloo}[1]{\left]#1\right[}
\newcommand{\intervalco}[1]{\left[#1\right[}
\newcommand{\intervalcc}[1]{\left[#1\right]}
\newcommand{\intervaloc}[1]{\left]#1\right]}

\newcommand{\bigintervalcc}[1]{{\big[#1\big]}}

\newcommand{\on}[1]{\operatorname{#1}}

\DeclareMathOperator{\ima}{im}

\DeclareMathOperator{\supp}{supp}   %
\DeclareMathOperator{\End}{End}    %

\DeclareMathOperator{\pr}{pr}
\DeclareMathOperator{\grad}{grad}

%% file: ntheoremenglish.tex
\newcommand{\mynewtheorem}[4] %
{
\newaliascnt{#1}{#2}
\newtheorem{#1}[#1]{#3}
\aliascntresetthe{#1}
\expandafter\def\csname #1autorefname\endcsname{%
#4%
}%
}

\newtheorem{theorem}{Theorem}[section]
\mynewtheorem{lemma}{theorem}{Lemma}{Lemma} 
\mynewtheorem{proposition}{theorem}{Proposition}{Proposition} 
\mynewtheorem{corollary}{theorem}{Corollary}{Corollary}
\mynewtheorem{quest}{theorem}{Question}{Question}

\mynewtheorem{problem}{theorem}{Problem}{Problem}

\theoremstyle{break}
\mynewtheorem{btheorem}{theorem}{Theorem}{Theorem} 
\mynewtheorem{blemma}{theorem}{Lemma} {Lemma}
\mynewtheorem{bcorollary}{theorem}{Corollary}{Corollary}
\mynewtheorem{bproblem}{theorem}{Problem}{Problem}

\theoremstyle{plain}
\theorembodyfont{\normalfont}
\mynewtheorem{definition}{theorem}{Definition}{Definition}
\mynewtheorem{example}{theorem}{Example}{Example}
\mynewtheorem{remark}{theorem}{Remark}{Remark}

\theoremstyle{break}
\mynewtheorem{bdefinition}{theorem}{Definition}{Definition}
\mynewtheorem{bexample}{theorem}{Example}{Example}
\mynewtheorem{bremark}{theorem}{Remark}{Remark}
\theoremheaderfont{\itshape}
\theorembodyfont{\upshape}
\theoremstyle{nonumberplain}
\theoremseparator{.}
\theoremsymbol{\ensuremath{\Box}}
\newtheorem{proof}{\textsc{Proof}}

%% file: Abstract.tex
Aiming at optimizing the shape of closed embedded curves within prescribed isotopy classes, we use a gradient-based approach to approximate stationary points of the Möbius energy. The gradients are computed with respect to Sobolev inner products similar to the $W^{3/2,2}$-inner product. This leads to optimization methods that are significantly more efficient and robust than standard techniques based on $L^2$-gradients.

%% file: Introduction.tex
\section{Introduction}\label{sec:Introduction}

Let $\Curve \colon \Circle \to \AmbSpace$ be a sufficiently smooth embedding\footnote{
In many cases we consider curves $\gamma$ being differentiable a.e.\ but not necessarily $C^{1}$.
Therefore we will always assume an \emph{embedding} to be a $C^{0}$-embedding.
Furthermore, $\gamma$ is \emph{immersed} (or \emph{regular})
if $\essinf\abs{\gamma'}>0$.} of the circle $\Circle$ into Euclidean space.
Its Möbius energy~\cite{MR1259363,MR1098918} is defined as
\begin{align}
	\Energy(\Curve) \ceq \int_\Circle \int_\Circle
	\paren{
		\frac{1}{\nabs{\Curve(x)-\Curve(y)}^2}	- \frac{1}{\varrho_\Curve^2(x,y)}
	}\, \nabs{\Curve'(x)}\, \nabs{\Curve'(y)} \, \dd x \, \dd y
	,
	\label{eq:MoebiusEnergy}
\end{align}
where $\varrho_\Curve(x,y)$ denotes the length of the shortest arc of $\Curve$ connecting $\Curve(x)$ and $\Curve(y)$.

The original motivation~\cite{MR928412} was to define an energy that measures
complexity or ``entangledness'' of a given curve.
One may expect that minimization will unravel the initial configuration to a state of less complexity.
Ideally, this should also preserve topological properties, in particular the isotopy class.
By definition, an isotopy class is a path component in the space of \emph{embedded} curves.
The Möbius energy was designed to erect infinite energy barriers that separate isotopy classes within the space of curves.
The term $\nabs{\Curve(x)-\Curve(y)}^{-2}$  blows up whenever a self-contact emerges,
lending itself as contact barrier for  modeling impermeability of curves and rods.
Moreover, this term promotes the spreading of the geometry, which indeed leads to the desired unfurling.
Subtracting the second term $\varrho_\Curve^{-2}(x,y)$ guarantees that the energy is finite for sufficiently smooth embeddings.
This way, any time-continuous descent method like, e.g., a gradient flow, will necessarily preserve the isotopy class.\footnote{Strictly speaking, this is not the full picture: Being scaling-invariant, the Möbius energy does not penalize pull-tight of small knotted arcs (see \cite[Thm.~3.1]{MR1195506}), which is in fact a change of topology.}
Another pleasant feature of the Möbius energy is that its critical points enjoy higher smoothness.

\begin{figure}[t]
\begin{center}
\newcommand{\filebasename}{FreedmanHeWang_1648E_NagasawaEnergy_}
\newcommand{\inc}[2]{\begin{tikzpicture}%
    \node[inner sep=0pt] (fig) at (0,0) {\includegraphics[	trim = 0 0 0 0, 
	clip = true,  
	angle = 0,
	width = 0.24\textwidth]{\filebasename#1}};
	\node[above right= -.3ex] at (fig.south west) {\begin{footnotesize}(#2)\end{footnotesize}};%
\end{tikzpicture}}
\capstart
\inc{000000}{0}%
\inc{000005}{5}%
\inc{000010}{10}%
\inc{000020}{20}%
\\	
\inc{000060}{60}%
\inc{000100}{100}%
\inc{000120}{120}%
\inc{000130}{130}%
\\
\inc{000160}{160}%
\inc{000175}{175}%
\inc{000185}{185}%
\inc{000200}{200}%
\end{center}
\caption[FreedmanHeWang]{
Discrete Sobolev gradient descent subject to edge length constraint and barycenter constraint.
The isotopy class is maintained along the iteration
which is the crucial feature of a knot energy.
As initial condition, we use a ``difficult'' configuration proposed
in~\protect\cite{MR1259363} (1648 edges; numbers in parentheses indicate the iteration steps).
The global minimizer (the round circle) is reached after about 200 iterations.
The curves have perceived constant thickness in the plots while a coordinate cross
serves as a reference for the respective scaling factor.
See also \protect\autoref{fig:Benchmark} for a comparison to further optimization methods; the present one is ``$W^{3/2,2}$ projected gradient, explicit''.
}
\label{fig:FreedmanHeWangKnot}
\end{figure}
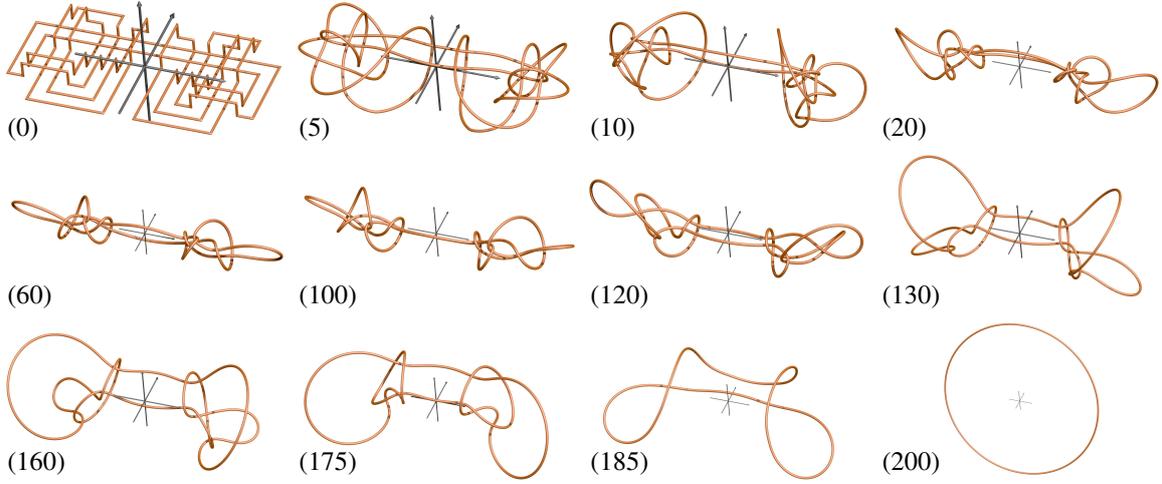

In this paper we propose a new concept of numerical optimization techniques
for the large family of self-repulsive energies
by discussing the prototypical case of the Möbius energy.
Due to the nonlocal point-point interactions (which manifest themselves in the occurrence of a double integral), any evaluation of the energy or its gradient is rather expensive;
this renders the numerical optimization a challenging task.
The key idea of our approach is to introduce a special geometric variant of the metric of the Sobolev space $\Sobo[3/2,2]$ that discourages movement of an embedded curve in regions of near self-contact.
Contrary to black-box approaches, our method allows us to minimize the Möbius energy of even quite complicated starting configuration within only a few hundred iterations (see \autoref{fig:FreedmanHeWangKnot} and \autoref{fig:Knot10}).
As illustrated in \autoref{fig:GradientPlots},
computing gradients with respect to this metric
allows for choosing significantly larger step sizes
compared to the $L^{2}$- or even the $\Sobo[3/2,2]$-metric.
This is in agreement with the interpretation of $\Sobo[3/2,2]$-gradient descent as a coarse discretization of an \emph{ordinary} differential equation.
In contrast to full discretization (i.e., in space and time) of a general (transient) partial differential equation, an ordinary differential equation does not require any mesh-dependent bound on the time step size for stability.
Consequently, our gradient descent scheme requires only few iteration steps, even for fine spatial resolution. This makes it, besides from being robust, particularly efficient.
This is demonstrated by the performance comparison in \autoref{fig:Benchmark}.

Potential applications for self-repulsive energies are manifold as they can be employed as barriers
for shape optimization problems and physical simulation with self-contact:
They arise, for instance, 
in mechanics~\cite{CAN,GRvdM,kroemer-valdman,SvdH,SvdH18,vdHNGT}
and
in molecular biology \cite{CS1,CS2,GPL2,GPL1,maddocks,MRM}.
The Möbius energy can also be considered as differentiable relaxation of \emph{curve thickness}.
For example, as reported in~\cite{ISI:A1996VT33600042}, the speed of migration
of knotted DNA molecules undergoing gel electrophoresis
seems to be proportional to the average crossing number of the corresponding maximizers
of curve thickness.
Software tools for the maximization of thickness or equivalently, for the minimization of ropelength, have been developed in \cite{MR1702021} (\emph{SONO}) and \cite{MR2802724} (\emph{ridgerunner}).
Further potential fields of applications for repulsive energies include
computer graphics~\cite{BWRAG,ST},
packing problems~\cite{MR2869515,MR2807140},
the modeling of
coiling and kinking of submarine communications cables~\cite{coyne,zajac},
and even solar coronal structures~\cite{MR3171936}.

\begin{figure}[t]
\begin{center}
\newcommand{\inc}[2]{\begin{tikzpicture}
    \node[inner sep=0pt] (fig) at (0,0) {\includegraphics[	
    trim = 140 40 320 140, 
	clip = true,  
	width = 0.32\textwidth
	]{#1}};
	\node[above right= -.3ex] at (fig.south west) {\begin{footnotesize}#2\end{footnotesize}};    
\end{tikzpicture}
}
\capstart
	\inc{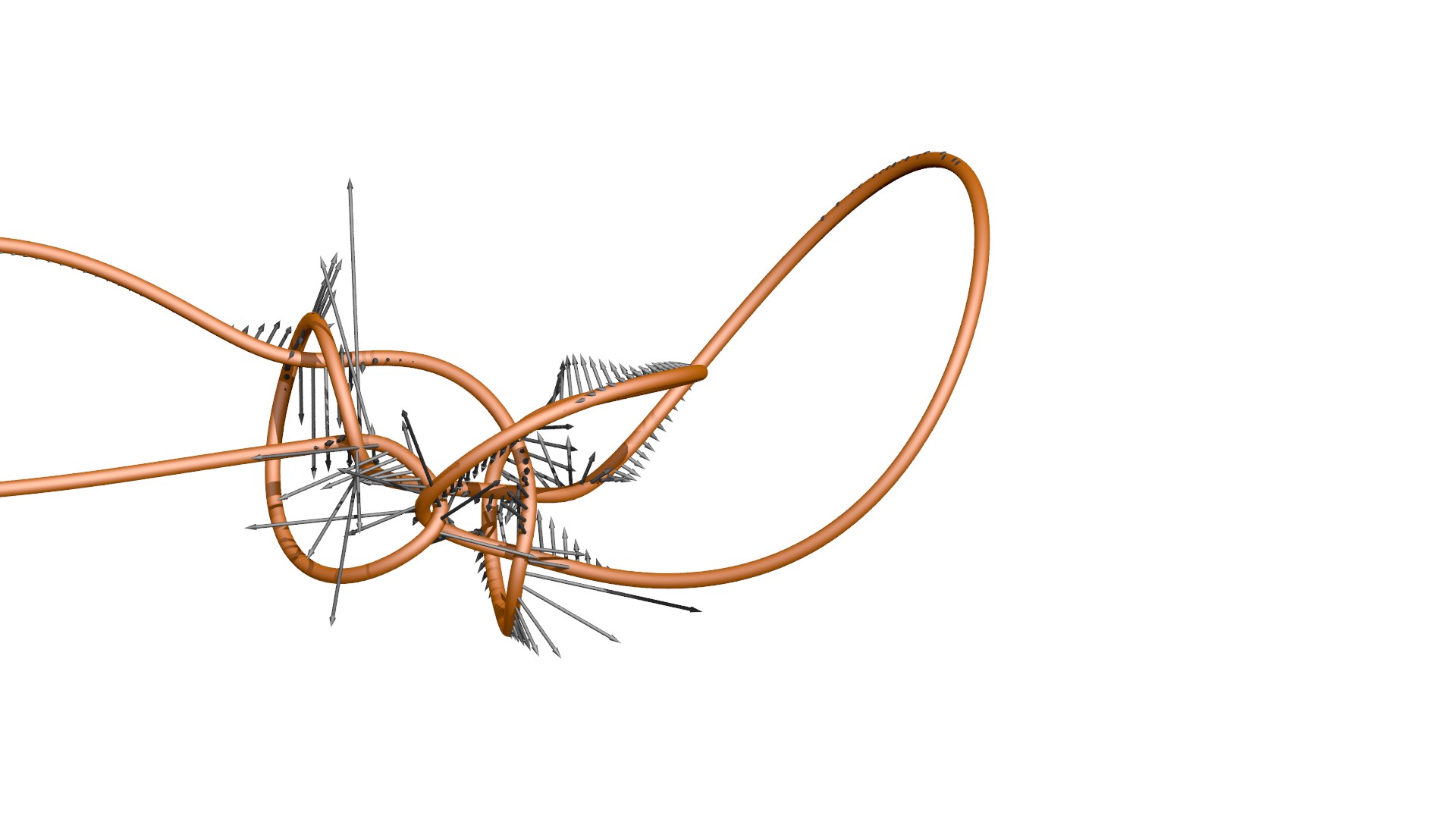}{$L^2$}%
	\inc{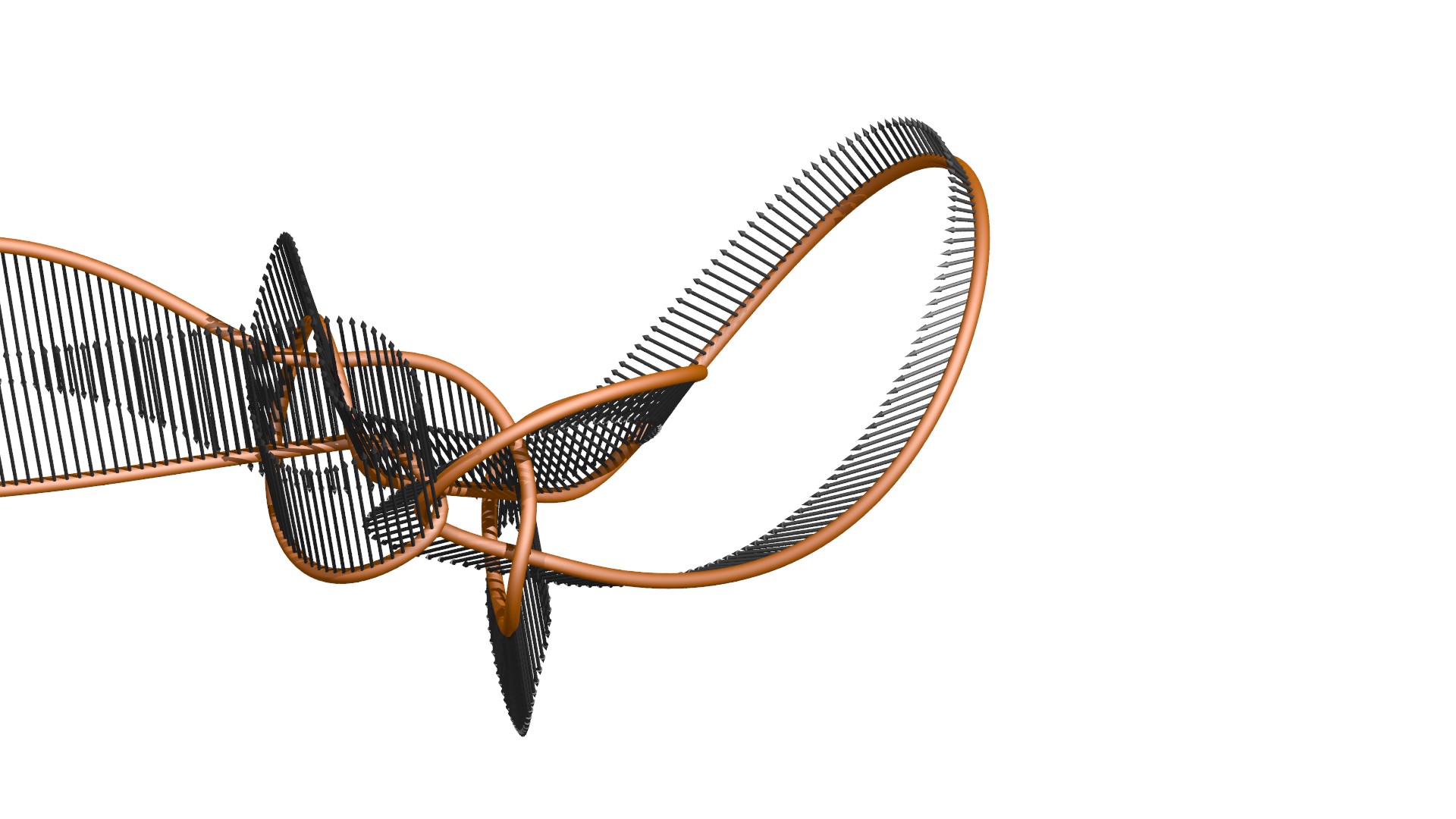}{pure $\Sobo[3/2,2]$}%
	\inc{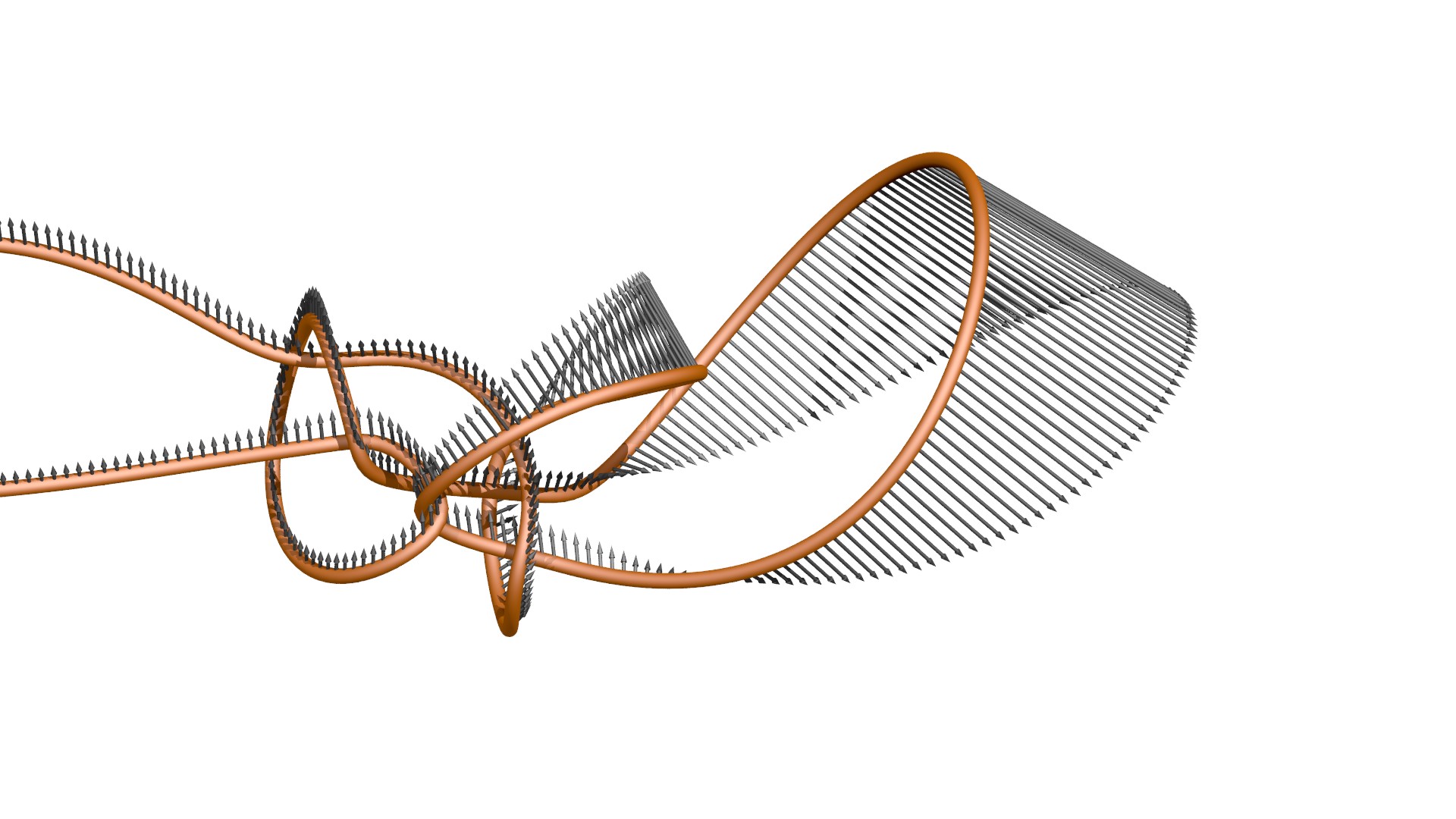}{geometric $\Sobo[3/2,2]$}
\end{center}
\caption{%
We visualize different gradients as vector fields along a given curve.
The $L^2$-gradient is pathologically concentrated on regions of near self-contact. Consequently, one has to pick tiny step sizes to prevent self-collision.
The pure $\Sobo[3/2,2]$-gradient behaves much better in the sense that
it is more uniformly distributed along the curve. 
However, this can still be improved considerably by adding a lower order term to the inner product that discourages movement in regions of near self-contact, cf.{} \protect\autoref{prop:MetricDefinition}.
}
\label{fig:GradientPlots}
\end{figure}
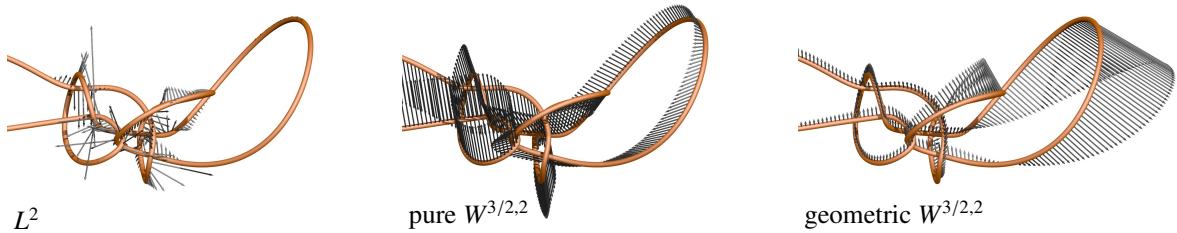

\subsection*{Previous work}

Since its invention by O'Hara~\cite{MR1098918,MR1195506,MR1986069}
and the very influential paper by Freedman, He, and Wang~\cite{MR1259363},
the Möbius energy has been studied by many authors.
Detailed investigations on its derivatives have been performed in \cite{MR3461038,MR1733697,MR3394390}.
Existence of minimizers in prime knot classes has been established in \cite{MR1259363}.
Invariance of the energy under conformal transformations of $\AmbSpace$ has been studied in \cite{MR1470730,MR1259363,MR1702037,MR3915937}.
Smoothness of minimizers has been established in \cite{MR1259363,MR1733697},
while smoothness and even analyticity of \emph{all} critical points has finally been shown in \cite{MR3461038} and \cite{Blatt2018}.
Except for the global minimizer~\cite{MR1259363}
and first results on critical points in nontrivial
prime knot classes~\cite{blatt2020,MR1246479},
almost nothing is known on the geometry of the energy space.
In light of the Smale conjecture (proven by Hatcher~\cite{MR701256}),
it would be of great interest to know
whether some gradient flow of the Möbius energy
actually defines a retract of the unknots to the round circles.
The $\Lebesgue[2]$-gradient flow of the Möbius energy has been studied in \cite{MR2875646,blatt2016gradient,MR1733697}.

\begin{figure}[t!]
\capstart
\begin{center}
\includegraphics[width=\textwidth]{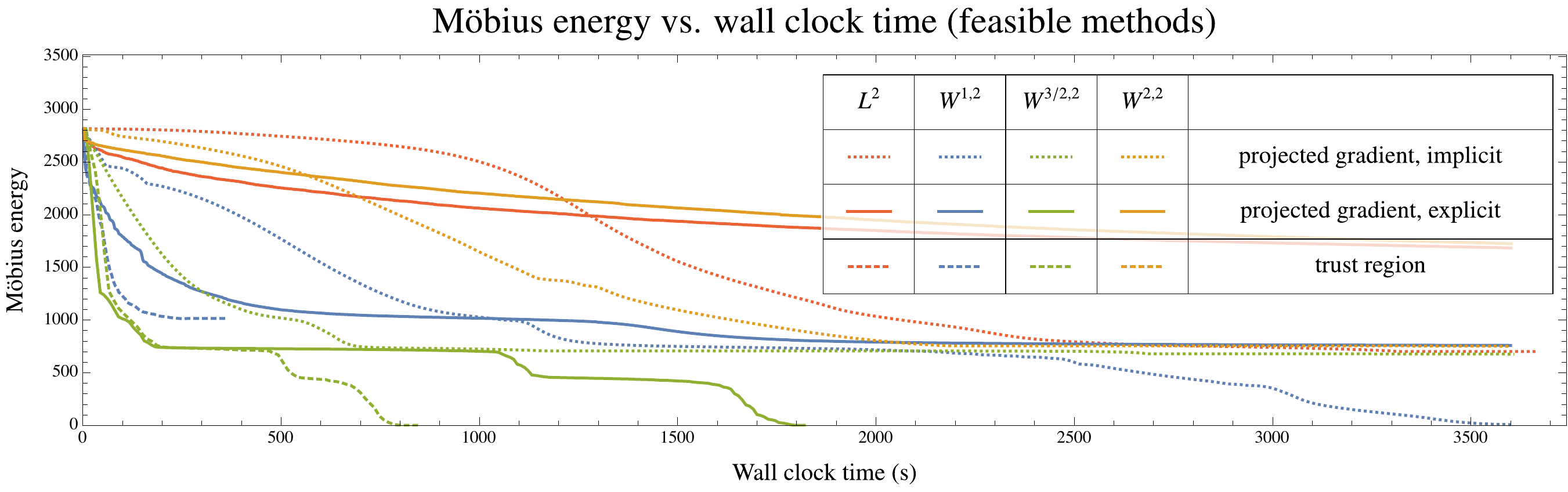}\\
\includegraphics[width=\textwidth]{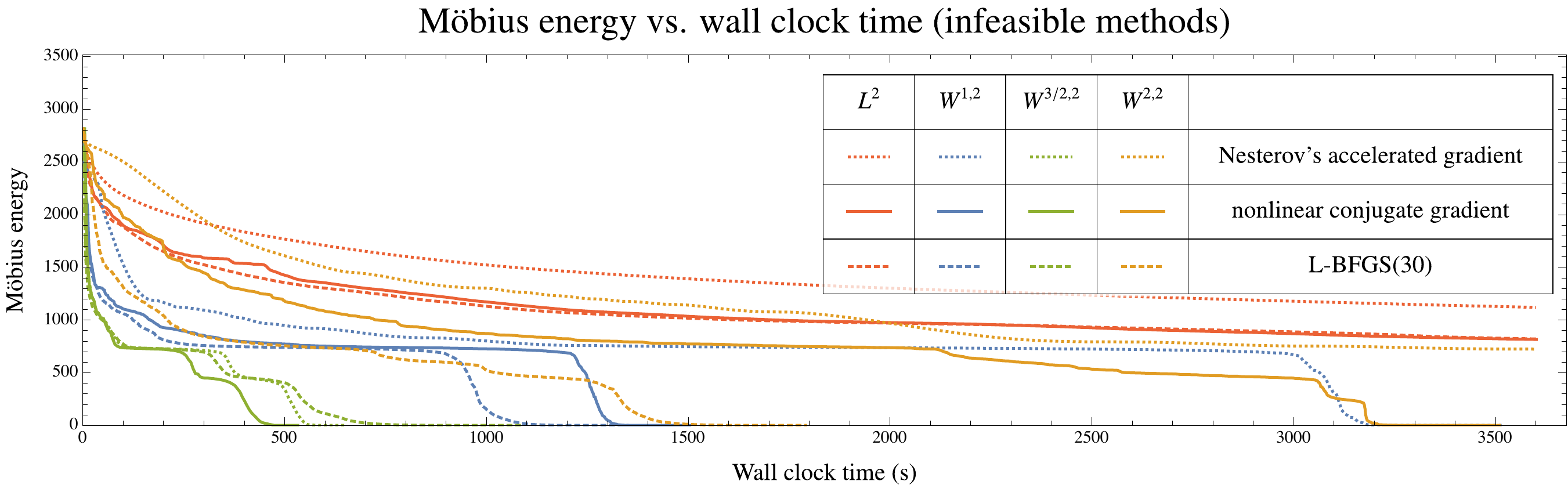}
\end{center}
\caption{
Exemplary performance comparison between several feasible (top) and infeasible (bottom) optimization methods and with respect to various Sobolev metrics, applied to the initial configuration from \protect\autoref{fig:FreedmanHeWangKnot} (1648 edges). ``Feasible'' means that the constraints were respected in each iteration step (up to a certain tolerance, of course). ``Infeasible'' means that a penalty formulation was used in place of hard constraints.
Each dataset corresponds to a combination of an optimization method (encoded by line dashing) and a Sobolev-metric (encoded by color; e.g., green corresponds to our Sobolev metric)
that have been employed to compute gradients.
We see that apart from the implicit projected gradient descent (which generally does not work well in this context), all optimization methods perform best in conjuction with our $\Sobo[3/2,2]$-metric.
All experiments were implemented in Mathematica\textsuperscript{\tiny\textregistered} and ran single-threaded for 60 minutes on an Intel\textsuperscript{\tiny\textregistered} Xeon\textsuperscript{\tiny\textregistered} E5-2690~v3. 
Further details will be provided in \protect\autoref{sec:OptimizationMethods}.}
\label{fig:Benchmark}
\end{figure}

Various numerical methods have been devised for discretizing and minimizing the Möbius energy
\cite{MR1246479,MR1470748,MR1702037,SimonMinimumDistanceEnergy},
partially with error analysis~\cite{MR2221528,MR2718624,MR3268981}.
A recently proposed scheme also preserves conformal invariance~\cite{2019arXiv191107024B,2018arXiv180907984B}.

The Möbius energy has also inspired the development of similar so-called \emph{knot energies}
\cite{MR1317077,MR1692638,MR2668877,MR2902275,MR3105400}
and higher-dimensional generalizations \cite{KvdM,MR3936491,MR1702037,OHara2020,MR2197957,MR2739778,MR3078345}.

\begin{figure}[t!]
\capstart
\begin{center}
\includegraphics[width=\textwidth]{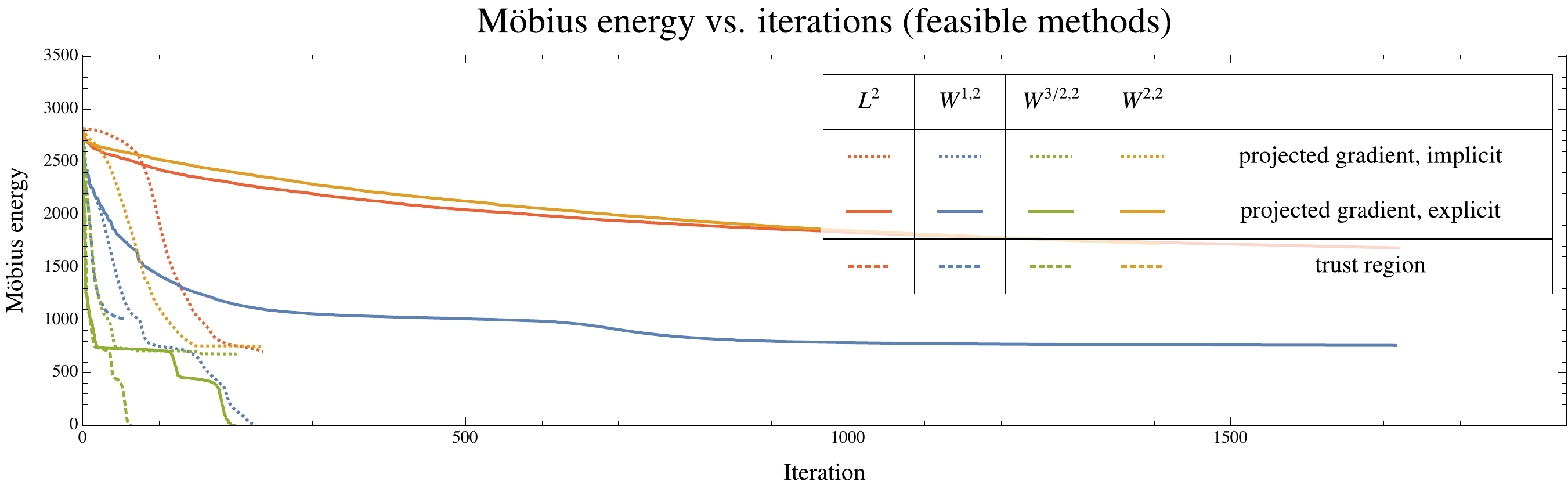}\\
\includegraphics[width=\textwidth]{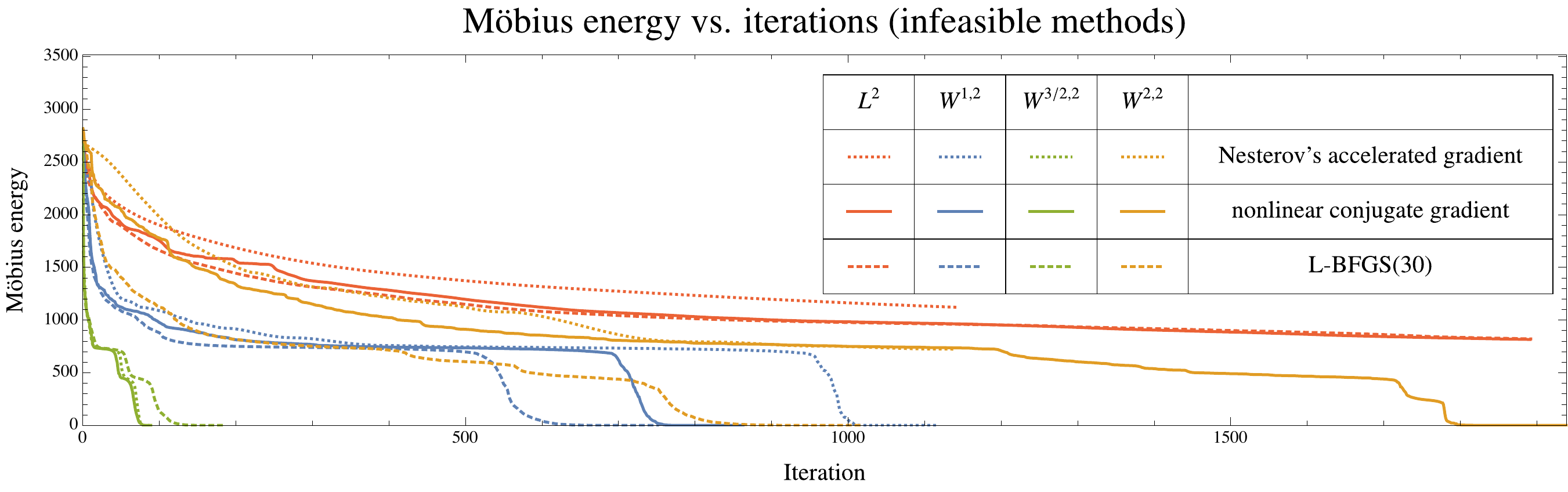}
\end{center}
\caption{As a matter of fact, the statistics in \protect\autoref{fig:Benchmark} highly rely on the hardware.
To provide a more independent comparison, we here plot the number of iteration steps versus the values of the Möbius energy attained after that time.
Of course, as the effort involved for performing a single iteration step differs among the methods discussed here,
it is debatable whether this is a meaningful unit after all.}
\label{fig:Benchmark2}
\end{figure}

Both theoretical and numerical results have been
obtained on linear combinations of the bending energy
and the Möbius energy~\cite{MR2729316,MR1098918,MR1657477}.
More generally, in order to find minimizers
of an elastic energy within an isotopy class,
each knot energy can be employed in two ways: either as regularizer as it was done, e.g., in \cite{2018arXiv180402206B,2019arXiv191107024B,MR3800032,GRvdM,gilsbach2021,HS1,MR3725391,MR3461787};
or by using it to encode a hard bound into the domain, which was done with the knot thickness in \cite{HS2,MR2029004,MR1674767}.

The applicability of self-avoiding energies is heavily limited by their immense cost:
Typical discretizations replace the double integrals by double sums which leads to a computational complexity of at least $\Omega((N \cdot \AmbDim)^2)$ for evaluating the discrete Möbius energy and its derivative, where $N \cdot \AmbDim$ is the number of degrees of freedom of the discretized geometry (e.g., the number of vertices of a polygonal line times the dimension of the ambient space).
This issue can be mended by sophisticated kernel compression techniques, see \cite{10.1145/3439429}.
In this article, however, we focus on another issue that is more related to mathematical optimization, namely the fact that, for $N \to \infty$, the discretized optimization problems become increasingly ill-conditioned.
It is well-known that the convergence rate of many gradient-based optimization methods (method of steepest descent, nonlinear conjugate gradient method, and also more sophisticated quasi-Newton methods like L-BFGS) is very sensitive to the condition number of the Hessian of the energy (at a minimum) on the one hand, and the inner product that is used to compute the gradients on the other hand.
The Hessian of the Möbius energy is deeply related to the fractional Laplacian $(-\Delta)^{3/2}$ which is a differential operator of order three, cf.~\cite{MR1733697}.
Thus the condition number of the discrete problem grows like $O(h^{-3})$ where $h$ denotes the typical length of an edge in the discretization.
In practice, this results in a rapid increase of the number of optimization iterations to ``reach the minimizer'' when the discretization is refined (i.e., for $h \to 0$).
Combined with the immense cost of evaluating $\Energy$ and $D\Energy$, this leads to a prohibitively high cost of minimizing $\Energy$ with black-box optimization routines (see \autoref{fig:Benchmark} and \autoref{fig:Benchmark2}).

In particular, this issue applies to the \emph{explicit Euler} time discretization scheme for the $L^2$-gradient flow of the Möbius energy.
Denoting the discretized energy by $\Energy_h$, the next time iterate $\Curve_{(t + \Delta t)}$ is computed from the current iterate $\Curve_t$ by solving
\begin{align*}
	\biginnerprod{\tfrac{\Curve_{(t + \Delta t)} - \Curve_t}{\Delta t }, \varphi}_{L^2_{\Curve_t}} + D \Energy_h(\Curve_{t}) \, \varphi = 0
	\quad
	\text{for all discrete vector fields $\varphi \colon \Circle \to \AmbSpace$}
\end{align*}
where $\biginnerprod{u,v}_{L^2_{\Curve_t}}=\int_{\Circle}\innerprod{u(x),v(x)}\abs{\Curve'(x)}\dd x$.
This can also be reinterpreted as method of steepest descent with respect to the (discretized) $L^2$-gradient and with step size $\Delta t>0$. Here the ill-conditioning manifests itself in the \emph{Courant--Friedrichs--Lewy condition}: 
As the $L^2$-gradient flow is a system of third order para\-bolic partial differential equations, the step size has to be truncated to $\Delta t = \LandO(h^3)$ in order to make this scheme stable.
This is also why a line search that enforces the \emph{Armijo condition} (also referred to as \emph{first Wolfe condition}), cf.~\cite[Chapter~3]{MR2244940},
will typically lead to tiny step sizes, rendering the method impractical for optimization (see \autoref{fig:Benchmark}).
It is well-known that the Courant--Friedrichs--Lewy condition can be circumvented by implicit time integration schemes.
For example, in the \emph{implicit Euler} or \emph{backward Euler} scheme, one determines the next iterate $\Curve_{(t + \Delta t)}$ by solving the equation
\begin{align*}
	\biginnerprod{\tfrac{\Curve_{(t + \Delta t)} - \Curve_t}{\Delta t}, \varphi}_{L^2_{\Curve_t}} + D \Energy_h(\Curve_{(t + \Delta t)}) \, \varphi = 0
	\quad
	\text{for all discrete vector fields $\varphi \colon \Circle \to \AmbSpace$.}
\end{align*}

Standard techniques for solving this nonlinear equation, e.g., Newton's method, require solving multiple 
linearizations of the above equation
and thus involve the Hessian $DD\Energy_h$ \emph{in each time iteration}. Moreover, the linearization has to be recomputed whenever the step size $\Delta t$ changes,
which makes it nontrivial to set up an adaptive time stepping scheme.
This explains why implicit time integrators turn out to be rather inefficient optimization schemes (see \autoref{fig:Benchmark}).
If one allows oneself to employ second derivatives of $\Energy_h$, applying Newton's method (and its damped or regularized derivates) for solving $D\Energy_h(\Curve_*) = 0$ in the first place would lend itself as a more efficient optimization algorithm.
However, it is well-known that Newton's method does not necessarily perform well when applied far away from critical points.

\subsection*{Sobolev gradients}

These problems can be overcome by optimization methods based on Sobolev gradients which are defined in terms of a Sobolev metric $G$ that is ``natural'' for the Möbius energy.
Blatt~\cite{MR2887901} characterized the energy space of the Möbius energy $\Energy$ as  $\Sobo[1,\infty][][\Circle][\AmbSpace] \cap \Sobo[3/2,2][][\Circle][\AmbSpace]$,
cf.\@ \autoref{theo:EnergySpace}.
Here and in the following, $\Sobo[s,p]$ denotes the \emph{\SoboSlobo space} of functions with ``$s$ fractional derivatives in $\Lebesgue[p]$'' if $s \not \in \Z$ and a conventional Sobolev space for $s \in \Z$.
This result points to the fact that $D \Energy$ is a nonlinear differential operator of order $2 \cdot \frac{3}{2} = 3$,
which has already been observed by He~\cite{MR1733697}.
So morally, a suitable inner product $G$ should be of the form
\begin{align*}
	G(u,w) 
	\textstyle
	\ceq 
	\int_\Circle \ninnerprod{ (- \Delta)^{3/4} \, u(x), (- \Delta)^{3/4} \, w(x)} \, \dd x
	.
\end{align*}
Then the $G$-gradient $\grad (\Energy) \at_\Curve$ at $\Curve$ can be defined by the following weak formulation:
\begin{align*}
	G(\grad (\Energy) \at_\Curve, w)
	\ceq
	D\Energy(\Curve)\,w
	\quad
	\text{for all $w \in \Holder[\infty][][\Circle][\AmbSpace]$}.	
\end{align*}
So, at least formally, the $G$-gradient satisfies the equation
\begin{align*}
		\grad (\Energy) \at_\Curve = (- \Delta)^{-3/2} \, D\Energy(\Curve).
\end{align*}
By a somewhat naive counting of fractional derivatives, the right hand side is a nonlinear differential operator of order zero. 
Hence there is a chance that $\grad (\Energy) \at_\Curve$ resides in the same Banach space as $\Curve$ so that $\grad(\Energy)$ would be a vector field.
Then the evolution equation
\begin{align}\label{eq:ode}
	\pd_t \Curve_t 
	=
	 - \grad (\Energy) \at_{\Curve_{t}}
\end{align}
would actually be an \emph{ordinary} differential equation.
Indeed, this turns out to be true and is part of our main result (see \autoref{theo:ProjectedGradients}).
This seems to imply that no Courant--Friedrichs--Lewy condition applies to the discretized problem, so that the number of gradient descent iterations ``to reach the minimum'' is quite insensitive to the mesh resolution. 
At least, this is what we observed in our experiments.

Since the inner product $G$ involves a choice of a Riemannian metric on the para\-metrization domain (line element and Laplacian), 
it is even more natural to define a $\Curve$-dependent family $\Curve \mapsto G_\Curve$ of inner products.
With the Riesz operator $\RI \at_\Curve \, u \ceq G_\Curve(u, \cdot)$, the $G$-gradient can then be expressed by
\begin{align}
	\grad (\Energy) \at_\Curve = (\RI \at_\Curve)^{-1} \; D\Energy(\Curve).
	\label{eq:DefinitionofGradient}
\end{align}
There are plenty of possible choices for $\RI$.
Most important is that $\RI \at_\Curve$ is an elliptic pseudo-differential operator of order three.
All compact perturbations of $\RI$ that are positive-definite will lead to the operator with the same qualitative properties.
In particular, we are not limited to the exact fractional Laplacian; this gives us the freedom to pick an $\RI \at_\Curve$ that is computationally more amenable.
Up to lower order terms, we design $G$ such that it resembles the $\Sobo[3/2,2]$-Gagliardo inner product, replacing intrinsic distances by (the easier computable) secant distances (see \autoref{prop:MetricDefinition}).
For a curve parametrized by arc length  (i.e., $\nabs{\Curve'} = 1$) and up to lower order terms, it reads
\begin{align}
	G_{\Curve}(u,w)
	&=
	\textstyle
	\int_\Circle	\int_\Circle
		\Biginnerprod{
			\frac{u'(x) - u'(y)}{ \nabs{\Curve(x)-\Curve(y)}^{1/2}}
			, 
			\frac{w'(x) - w'(y)}{ \nabs{\Curve(x)-\Curve(y)}^{1/2}}
		}	
		\, \frac{\dd x\, \dd y}{\nabs{\Curve(x) - \Curve(y)}}
	+
	\lot
	\label{eq:G}
\end{align}
where in case of a curve~$\Curve$ parameterized by arc length the lower-order terms are given by
\begin{align*}
	\MoveEqLeft
	\textstyle
	\lot 
	= \int_{\Circle} \int_{\Circle} 
 	\Biginnerprod{
 		\frac{u(x)-u(y)}{\abs{\Curve(x)-\Curve(y)}^{1/2}}
 		,
 		\frac{w(x)-w(y)}{\abs{\Curve(x)-\Curve(y)}^{1/2}}
 	}
 	\, 
 	\Bigparen{
 		\frac{1}{\abs{\Curve(x)-\Curve(y)}^{2}}
 		-
 		\frac{1}{\varrho_{\Curve}(x,y)^{2}}
 	}
 	\,
 	\frac{\dd x \, \dd y}{\abs{\Curve(x)-\Curve(y)}}
 	\\
 	&\qquad
	\textstyle
 	+ 
 	\Biginnerprod{ \int_\Circle u(x) \,\dd x, \int_\Circle w(y) \, \dd y}
\end{align*}
and $\varrho_{\Curve}$ denotes the geodesic distance introduced
in \eqref{eq:geoddist} below. Here the first summand is essentially the $\Sobo[1/2,2]$-Gagliardo inner product with the energy density as additional weight.

Indeed, even if $\Curve$ is not parametrized by arc length, a more detailed analysis reveals that $\RI \at_\Curve$ has (up to a constant) the same principal symbol as $(-\Delta_\Curve)^{3/2}$ where $\Delta_\Curve$ is the Laplace-Beltrami operator with respect to the Riemannian metric 
on $\Circle$ induced by the embedding $\Curve$ (see the proof of \autoref{prop:MetricDefinition}).

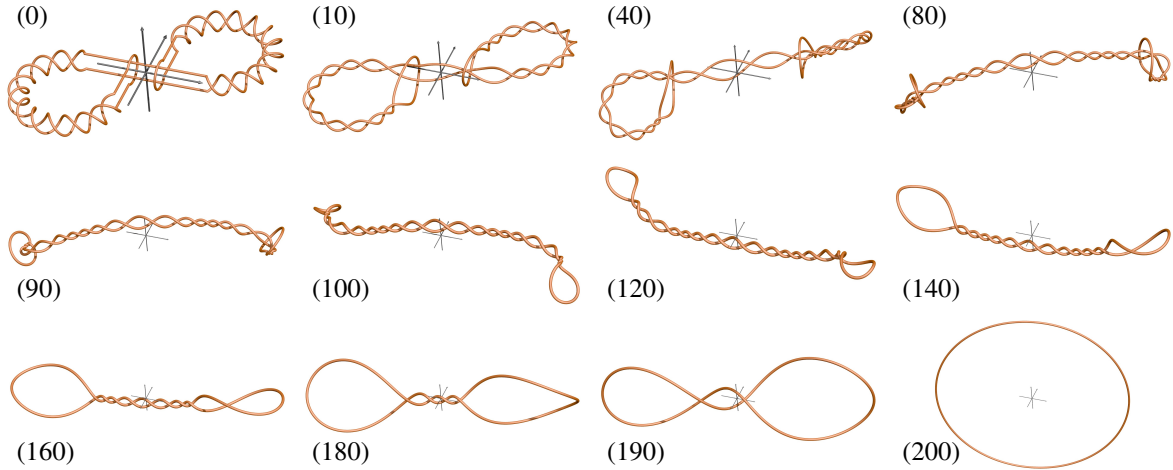
\begin{figure}[t]
\newcommand{\file}{HelixKnot_1940E_NagasawaEnergy_}
\newcommand{\inca}[2]{\begin{tikzpicture}
    \node[inner sep=0pt] (fig) at (0,0) {\includegraphics{\file#1}};
	\node[below right= .25ex] at (fig.north west) {\begin{footnotesize}(#2)\end{footnotesize}};    
\end{tikzpicture}
}
\newcommand{\incb}[2]{\begin{tikzpicture}
    \node[inner sep=0pt] (fig) at (0,0) {\includegraphics{\file#1}};
	\node[above right= .25ex] at (fig.south west) {\begin{footnotesize}(#2)\end{footnotesize}};    
\end{tikzpicture}
}
\begin{center}
\presetkeys{Gin}{
	trim = 0 0 0 0, 
	clip = true,  
	angle = 0,
	width = 0.24\textwidth
}{}
\capstart	
	\inca{000000}{0}%
	\inca{000010}{10}%
	\inca{000040}{40}%
	\inca{000080}{80}%
	\\	
	\incb{000090}{90}%
	\incb{000100}{100}%
	\incb{000120}{120}%
	\incb{000140}{140}%
	\\		
	\incb{000160}{160}%
	\incb{000180}{180}%
	\incb{000190}{190}%
	\incb{000200}{200}%
\end{center}
\caption{Discrete Sobolev gradient descent as in \protect\autoref{fig:FreedmanHeWangKnot} starting at another difficult configuration (1940 edges).}
\label{fig:HelixKnot}
\end{figure}

\subsection*{As Riemannian as you can get}

The overarching idea behind all this is to consider $(\ConfSpace,G)$ as a Riemannian manifold and $\Energy \colon \ConfSpace \to \R$ as a smooth function.
Here $\ConfSpace$ denotes a Banach manifold of immersed embedded curves which will be defined in \eqref{eq:confspace} below.
If $\grad (\Energy)$ is a well-behaved vector field on $\ConfSpace$, various optimization techniques that work on Riemannian manifolds can be utilized to minimize $\Energy$.
This is actually a long standing dream of differential geometers: 
to apply Riemannian geometry to an infinite-dimensional space of shapes.
Such Sobolev inner products and their geodesics have been studied from a geometrical point of view, e.g., in
\cite{MR2888014,MR4073208,MR2201275}.
It has observed that $W^{1,2}$-inner products work well in the numerical treatment of full dimensional elasticity and of membrane energies such as the area functional for surfaces or the length functional for curves \cite{Neuberger:1997:SobolevGradients,MR1246481,MR3915940}. Moreover, it is known that $W^{2,2}$-inner products provide good preconditioning for bending energies such as Bernoulli's elastic energy of curves, Kirchhoff's thin shell energy, the Willmore energy and Helfrich-type energies \cite{Eckstein:2007:GSF:1281991.1282017,doi:10.1111/cgf.12450,MR3915940,1703.06469}.
Various standard optimization schemes (e.g, nonlinear conjugate gradient, Nesterov's accelerated gradient, L-BFGS, trust region) can be sped up significantly by using the ``right'' notion of gradient (see \autoref{fig:Benchmark} and \autoref{fig:Benchmark2}). This is because these methods exploit that the gradient field is (locally) Lipschitz continuous with respect to the employed metric.

Alas, the story here is not \emph{that} simple,
because there is no Morrey embedding from the energy space $\Sobo[3/2,2][][\Circle][\AmbSpace]$ to $\Sobo[1,\infty][][\Circle][\AmbSpace]$
and any open $\Sobo[3/2,2]$-neigh\-bor\-hood of an
embedded arc-length parametrized $\Sobo[3/2,2]$-curve
may contain non-embedded curves
or curves with vanishing or infinite derivative.
Therefore, Fréchet differentiability
of the Möbius energy could only be established
with respect to the somewhat artificial
$\Sobo[3/2,2]\cap\Sobo[1,\infty]$-topology~\cite{MR3461038}.
This problem can be resolved by working in the slightly smaller Banach space $\RX{} \!\ceq \Sobo[3/2+\diffoffset,p][][\Circle][\AmbSpace]$ with suitable $\diffoffset>0$ and $p \geq 2 $.
Then $\RX$ embeds into $C^{1}$ and
the \emph{configuration space}
\begin{align}\label{eq:confspace}
	\ConfSpace \ceq \set{\Curve \in \SoboC[\strongs,\strongp] | \text{$\Curve$ is an immersed embedding}}
\end{align}
is an open subset of $C^{1}$.

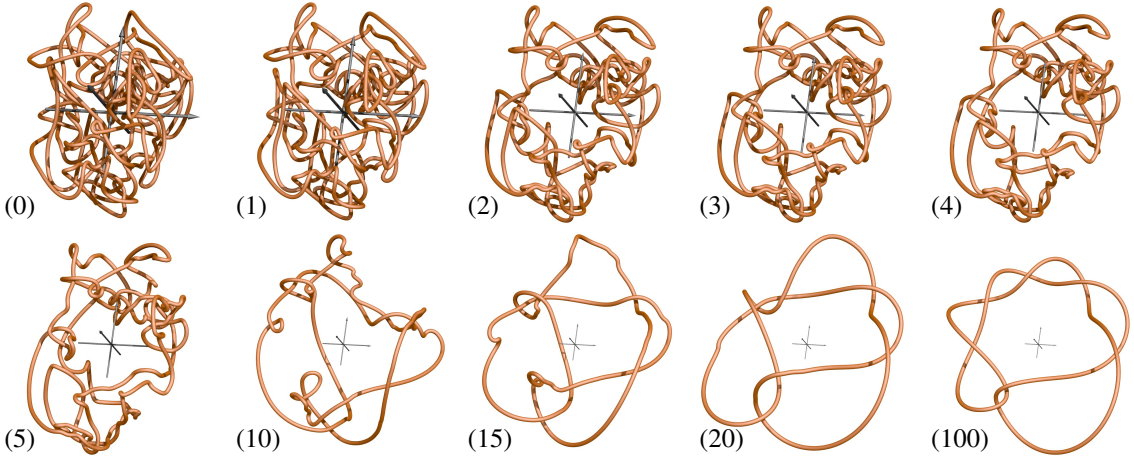
\begin{figure}[t]
\begin{center}
\newcommand{\file}{Knot_0010_E3000_NagasawaEnergy_}
\newcommand{\inc}[2]{\begin{tikzpicture}%
    \node[inner sep=0pt] (fig) at (0,0) {\includegraphics[	trim = 0 0 0 0, 
	clip = true,  
	angle = 0,
	width = 0.19\textwidth]{\file#1}};
	\node[above right= -.3ex] at (fig.south west) {\begin{footnotesize}(#2)\end{footnotesize}};%
\end{tikzpicture}}
\capstart	
\inc{000000}{0}%
\inc{000001}{1}%
\inc{000002}{2}%
\inc{000003}{3}%
\inc{000004}{4}%
\\
\inc{000005}{5}%
\inc{000010}{10}%
\inc{000015}{15}%
\inc{000020}{20}%
\inc{000100}{100}%
\end{center}
\caption[FreedmanHeWang]{
Discrete Sobolev gradient descent as in \protect\autoref{fig:FreedmanHeWangKnot} within the nontrivial knot class $7_2$, using the method ``$W^{3/2,2}$ projected gradient, explicit''. The initial configuration has 3000 edges and was randomly generated with \emph{KnotPlot}~\protect\cite{knotplot}.}
\label{fig:Knot10}
\end{figure}

We construct the Riesz isomorphism $\RI\at_\Curve $
as an elliptic pseudo-differential operator of order three, and we show in \autoref{prop:MetricDefinition} that it gives rise to a generalized Riesz isomorphism
$\RJ\at_\Curve \colon \RX[][] \to  \RY[][][\dual]$
where
$\RY \ceq \Sobo[3/2-\diffoffset,q][][\Circle][\AmbSpace]$,
with the Hölder conjugate $q \ceq (1-1/p)^{-1}$ of $p$.
Notice that $\RJ\at_\Curve$ does no longer identify $\RX$ with its dual space as 
$\RX \subsetneq \RY$, thus $\RY[][][\dual] \subsetneq  \RX[][][\dual]$.
So one of our major tasks (see \autoref{theo:DEnergy}) will be to establish that
$D\Energy(\Curve) \in \RY[][][\dual]$ whenever $\Curve \in \ConfSpace$.
Moreover, we show that $D\Energy$ is locally Lipschitz continuous as a mapping $\ConfSpace \to \RY[][][\dual]$, leading to our first main result:

\begin{theorem}\label{theo:Gradient}
The \emph{gradient} $\grad(\Energy)$ of $\Energy$ defined by \eqref{eq:DefinitionofGradient}
is a well-defined, locally Lipschitz continuous vector field on the configuration space $\ConfSpace$ (with respect to the norm on $\RX[][]$).
Moreover, it satisfies
$
	D\Energy(\Curve) \, \grad(\Energy) \at_\Curve \geq 0
$
with equality if and only if $D\Energy(\Curve) = 0$.
\end{theorem}

Combined with the Picard--Lindelöff theorem,
this statement guarantees the short-time existence of the gradient flow, both for the downward and the upward direction.

\bigskip

In \autoref{sec:Constraints},
we deal also with equality constraints, i.e.,
with Banach submanifolds of the form
$\ConstraintMfld \ceq \set{\Curve \in \ConfSpace | \ConstraintMap(\Curve) = 0}$,
where $\ConstraintMap \colon \ConfSpace \to \TargetSpace$ is a suitable submersion,
namely the constraint of constant speed and vanishing barycenter, cf.\ \eqref{eq:constraint},
into a further Banach space~$\TargetSpace {}=\Sobo[\strongss,\strongp][][\Circle][\R] \oplus \AmbSpace$.
We formulate a linear saddle point system for determining the projected gradient $\grad_{\ConstraintMfld} (\Energy|_\ConstraintMfld)\at_\Curve$ and analyze when the system is solvable. 
We perform the analysis for a concrete set of constraints (fixed barycenter and parametrization by arc length),
but we also try to outline which steps have to be taken for more general constraints.
Finally, \autoref{prop:SubmanifoldTheorem} will establish our second main result:

\begin{theorem}[Projected gradient]\label{theo:ProjectedGradients}
The \emph{projected gradient} $\grad_\ConstraintMfld(\Energy|_\ConstraintMfld)$ of $\Energy|_\ConstraintMfld$ defined by
\begin{align*}
	G_\Curve \bigparen{
		\grad_\ConstraintMfld(\Energy|_\ConstraintMfld)\at_\Curve,
		w
	} 
	\ceq 
	D(\Energy \at_\ConstraintMfld)(\Curve) \, w
	\quad
	\text{for all $w \in \Holder[\infty][][\Circle][\AmbSpace]$ with $D\ConstraintMap(\Curve) \, w = 0$}
\end{align*}
is a well-defined, locally Lipschitz continuous vector field on $\ConstraintMfld$.
The gradient satisfies
$
	D(\Energy|_\ConstraintMfld)(\Curve) \; \grad_\ConstraintMfld(\Energy|_\ConstraintMfld) \at_\Curve \geq 0
$
with equality if and only if $D(\Energy|_\ConstraintMfld)(\Curve) = 0$.
\end{theorem}

Invoking the Picard--Lindelöff theorem again, we conclude that both the downward and the upward gradient flows of $\Energy|_\ConstraintMfld$ exist for short times.

\medskip

The question of long-time existence is much more involved.
Following the way paved by Knappmann et al.~\cite{knappmann-etal2019}
for a subfamily of integral Menger curvature functionals,
one may derive this property in the case of subcritical Hilbert spaces.
These correspond to the functionals obtained by replacing the squares in \eqref{eq:MoebiusEnergy} by powers $\alpha\in(2,3)$.
Due to the fact that the general case where $p\ne2$ seems to be ``degenerate'' analogously to the $p$-Laplacian
it seems unclear whether long-time existence can be established
also for the setting discussed in this article.

\subsection*{Future directions}

The present study demonstrates the design
of a minimization scheme being both robust and efficient
which is based on a metric that is tailored to the structure of
a geometric nonlocal functional modeling self-avoidance.

The general strategy outlined in this paper
applies to a large range of functionals on curves and
surfaces of arbitrary dimension and codimension.
We stress the fact that the arguments given below mainly rely on analytical features of a functional defined on fractional
Sobolev spaces
rather than on geometric peculiarities, except for the metric itself which has to be chosen carefully depending on the respective problem.

Although the definition of the M\"obius energy has been motivated by the electrostatic energy~\cite{MR1195506}, it is admittedly not a physical quantity in the first place.
However, it seems to be an appropriate candidate to demonstrate the
general approach while avoiding too much technicalities as,
from an analyst's perspective, it is the most elementary smooth knot energy.

Even more importantly, one may find minimizers
of physical functionals such as, e.g., the bending energy or the Helfrich energy
within prescribed isotopy classes by
a regularization approach, cf.~\cite{GRvdM}.
In this context one may choose the regularizer to be a smooth repulsive functional
which approximates the (reciprocal) thickness
such as the tangent-point potential
which has been employed e.g.\ in~\cite{2018arXiv180402206B}.
In combination with the technique described
in the present paper, one may greatly improve
not only the performance but also the complexity
of the objects (i.e., isotopy types) that can be dealt with.

The higher-dimensional case as well as the adaption
of this technique to other functionals is work in progress \cite{YuSchumacherCrane}.

%% file: Notation.tex
\section{Preliminaries}\label{sec:Preliminaries}

\subsection*{General notation}

Throughout, we let $\Circle \ceq \set{x \in \R^2 | \nabs{x} = (2\,\uppi)^{-1}}$ be the round circle with a fixed orientation and normalized to have total length $\nabs{\Circle} = 1$.
We will make use of the identification $\Circle\cong\R/\Z$ whenever convenient.
Moreover, we write $\Torus = \Circle \times \Circle$ for the Cartesian product of the circle with itself
and denote by $\prx \colon \Torus \to \Circle$ and $\pry \colon \Torus \to \Circle$ the Cartesian projections onto the first and second factor, respectively.
We denote the canonical intrinsic distance function on $\Circle$ by
\begin{align*}
	d_\Circle(x,y) \ceq (2 \, \uppi)^{-1} \, \abs{\measuredangle(x,y)}
	=
	(2 \, \uppi)^{-1}\arccos\bigparen{(2 \, \uppi)^{2}\innerprod{x,y}} \in \bigintervalcc{0,\tfrac{1}{2}} 
	\qquad\text{for $x$, $y \in \Circle$}
\end{align*}
and the canonical line measure by $\dd x$ or $\dd y$.
Each sufficiently smooth immersed embedding $\Curve \colon \Circle \to \AmbSpace$ induces
a \emph{line element} $\LineElementC(x)$ and a \emph{unit tangent field} $\Tangent_\Curve$ via
\begin{align*}
	\LineElementC(x) \ceq \nabs{\Curve'(x)} \, \dd x
	\qand
	\Tangent_\Curve(x) \ceq \tfrac{\Curve'(x)}{\nabs{\Curve'(x)}}.
\end{align*}
Moreover $\Curve $ induces two further distance functions that we have to distinguish: The \emph{secant distance} $\nabs{\triangle \Curve}(x,y) \ceq \nabs{\Curve(x) - \Curve(y)}$ and the \emph{geodesic distance} $\varrho_\Curve$; more precisely,
\begin{align}\label{eq:geoddist}
	\textstyle
	\varrho_\Curve(x,y) 
	\ceq 
	\int_{I_\Curve(x,y)} \LineElement_\Curve
	,
	\quad
	I_\Curve(x,y)
	\ceq
	\argmin \set{
		\int_{J} \LineElement_\Curve
		|
		\text{$J \subset \Circle$ conn., $\partial J = \{x,y\}$}
		},
\end{align}
where $I_\Curve(x,y)$ denotes the shortest arc that connects $x$ and $y$.
Since $\Curve$ is immersed, $d_\Circle$ and $\varrho_\Curve$ are equivalent.
We point out that this equivalence extends to
$\nabs{\triangle \Curve}$
if the embedding $\Curve$ is sufficiently smooth, e.g., of class $\Holder[1,\alpha]$ with $\alpha \in \intervaloo{0,1}$
or $W^{1+\sigma,r}$ with  $\sigma - 1/r \ge 0$, cf.~\cite[Lemma~2.1]{MR2887901}.
In this case $\Curve$ is bi-Lipschitz continuous
and the measures $\dd x$ and $\LineElementC$ are equivalent as well, i.e., there are $c_1$, $c_2 >0$ such that $c_1 \, \dd x \leq \LineElementC(x) \leq c_2 \, \dd x$ holds for all $x \in \Circle$. This implies that also the Lebesgue norms 
\begin{align*}
	\nnorm{u}_{\Lebesgue[p]} \ceq \Bigparen{\textstyle \int_\Circle \nabs{u(x)}^p \, \dd x}^{1/p}
	\qand
	\nnorm{u}_{\Lebesgue[p][\Curve]} \ceq \Bigparen{\textstyle \int_\Circle \nabs{u(x)}^p \, \LineElementC (x)}^{1/p}
\end{align*}
for $1\leq p < \infty$ and any measurable function $u \colon \Circle \to \AmbSpace$ are equivalent.
We also employ this notation for bivariate measurable functions $U:\Circle^{2}\to\AmbSpace$, letting
\begin{align*}
	\nnorm{U}_{\Lebesgue[p]} \ceq \Bigparen{\textstyle \int_{\Circle^{2}} \nabs{U(x,y)}^p \, \dd x\, \dd y}^{1/p}
	\qand
	\nnorm{U}_{\Lebesgue[p][\Curve]} \ceq \Bigparen{\textstyle \int_{\Circle^{2}} \nabs{U(x,y)}^p \, \LineElementC (x)\, \LineElementC (y)}^{1/p}.
\end{align*}
Likewise, for $	0 < \sigma <1$ and $1\leq p < \infty$, the \SoboSlobo seminorms
\begin{align*}
	\nseminorm{u}_{\Sobo[\sigma,p]} 
	\ceq 
	\Bigparen{\textstyle 
		\int_{\Circle^{2}} \Bigabs{\frac{u(x) - u(y)}{d_\Circle(x,y)^\sigma}}^p  \frac{\dd x \, \dd y}{d_\Circle(x,y)}
	}^{1/p}
	\;\;\text{and}\;\;
	\nseminorm{u}_{\Sobo[\sigma,p][\Curve]} 
	\ceq 
	\Bigparen{\textstyle 
		\int_{\Circle^{2}} \bigabs{\frac{u(x) - u(y)}{\nabs{\triangle \Curve(x,y)}^\sigma}}^p \, \frac{\LineElementC(x)\,\LineElementC(y)}{\nabs{\triangle \Curve(x,y)}}
	}^{1/p}	
\end{align*}
and the induced norms $\nnorm{u}_{\Sobo[\sigma,p]}  \ceq \nseminorm{u}_{\Sobo[\sigma,p]} + \nnorm{u}_{\Lebesgue[p]}$ and $\nnorm{u}_{\Sobo[\sigma,p][\Curve]}  \ceq \nseminorm{u}_{\Sobo[\sigma,p][\Curve]} + \nnorm{u}_{\Lebesgue[p][\Curve]}$
are equivalent, respectively.
In all what follows, we will frequently make use of the following $\Curve$-dependent measures and operators:
\begin{align}
	\varOmega_\Curve(x,y) &\ceq \LineElementC(x)\,\LineElementC(y),
	&
	\singularmeasure_\Curve&\ceq \tfrac{\varOmega_\Curve}{\nabs{\triangle \Curve}},	
	\label{eq:Measures}
	\\
	\triangle u(x,y) &\ceq u(x) - u(y),
	&
	\diff{\sigma}{\Curve} u &\ceq \tfrac{\triangle u}{\nabs{\triangle \Curve}^\sigma}.
	\label{eq:Operators}
\end{align}
For example, the $\Curve$-dependent \SoboSlobo\ seminorm can be written much more economically as
$
	\seminorm{u}_{\Sobo[\sigma,p][\Curve]}
	=\nnorm{ \diff{\sigma+1/p}{\Curve} u}_{\Lebesgue[p][\Curve]}
	=
	\nnorm{ \diff{\sigma}{\Curve} u}_{\Lebesgue[p][\singularmeasure_\Curve]},
$
where $\Lebesgue[p][\singularmeasure_\Curve][\Torus][\AmbSpace]$ denotes the Lebesgue space with respect to $\singularmeasure_\Curve$ and $\nnorm{\cdot}_{\Lebesgue[p][\singularmeasure_\Curve]}$ its associated norm.

We define $\Sobo[s,p]$-seminorms for $1<s<2$ by concatenating the $\Sobo[s-1,p]$-seminorms with suitable differential operators of first order:
\begin{align*}
	\nseminorm{u}_{\Sobo[s,p]}
	\ceq 
	\nseminorm{u'}_{\Sobo[s-1,p]}
	\qand
	\nseminorm{u}_{\Sobo[s,p][\Curve]}
	\ceq 
	\nseminorm{\cD_\Curve u}_{\Sobo[s-1,p][\Curve]},
	\quad
	\text{where}
	\quad	
	\cD_\Curve u \ceq \tfrac{u'}{\nabs{\Curve'}}.
\end{align*}
Here, the differential operator $\cD_\Curve$ can be interpreted as \emph{derivative with respect to arc length}.
Provided that $\Curve$ is a sufficiently smooth immersed embedding,
$\nnorm{u}_{\Sobo[s,p]} 
\ceq 
\nseminorm{u}_{\Sobo[s,p]} 
+ 
\nnorm{u}_{\Lebesgue[p]}$
and
$\nnorm{u}_{\Sobo[s,p][\Curve]} \ceq 
\nseminorm{u}_{\Sobo[s,p][\Curve]} + 
\nnorm{u}_{\Lebesgue[p][\Curve]}$
are equivalent and both topologize the \SoboSlobo space
\begin{align*}
	\Sobo[s,p][][\Circle][\AmbSpace]
	\ceq \set{ u \in \Sobo[1,p][][\Circle][\AmbSpace] | \nseminorm{u}_{\Sobo[s,p]}< \infty}.
\end{align*}
More precisely, the norm $\nnorm{\cdot}_{\Sobo[s,p][\Curve]}$ is well-defined and equivalent to $\nnorm{\cdot}_{\Sobo[s,p]}$ if $\Curve$ is an immersed embedding of class $\Sobo[S,P][][\Circle][\AmbSpace]$ 
provided that one of the conditions for the ``product rule'' \autoref{lem:MultiplicationLemma} are met for $\sigma_1 = S-1$, $p_1 = P$, $\sigma_2 = s-1$, $p_2 = p$.

\subsection*{Spaces}\label{sect:spaces}

Our initial motivation to consider $\Sobo[3/2,2]$-inner products for optimization is the following characterization of the \emph{energy space} of the Möbius energy, i.e., of the smallest space that contains all finite-energy configurations:

\begin{theorem}[Blatt \protect\cite{MR2887901}]\label{theo:EnergySpace}
Let $\Curve \in \Sobo[1,\infty][][\Circle][\AmbSpace]$ be an embedded immersed curve parametrized by arc length, i.e., $\nabs{\Curve'(x)} =1$ for a.e.\@ $x$.
Then one has
$
 	\Energy(\Curve) < \infty
$
if and only if
$\Curve \in \Sobo[3/2,2][][\Circle][\AmbSpace]$.
\end{theorem}
Moreover, provided that $\Curve$ has a certain minimal regularity, the differential of $\Energy$ has been characterized as a nonlinear, nonlocal ``differential operator`` of order $3$ in the sense that $D\Energy(\Curve)$ is a distribution with three derivatives less than $\Curve$ (see \cite{MR1733697}).
We will see this also in \autoref{theo:DEnergy} below.
As indicated in the introduction,
instead of working with the energy space $\SoboC[3/2,2][] \cap \SoboC[1,\infty][]$,
we prefer spaces of curves with slightly higher regularity.
In the first place, we avoid some technicalities
effected by the critical scaling of $\Sobo[1/2,2][]$
(see~\cite{MR3799622})
related to discontinuous tangents,
in particular with respect to product rules.
Here and in the following, we fix parameters $\hilberts$, $\diffoffset$, and $\strongp$ satisfying
\begin{align}
	\hilberts >1,
	\quad	
	\diffoffset >0,
	\quad
	1 < \hilberts - \diffoffset < \hilberts + \diffoffset <2,
	\quad
	\strongp \in \intervalco{2,\infty},
	\qand
	\strongs - \tfrac{1}{\strongp} >1
	.
	\label{eq:Parameters}
\end{align}
In fact, we will soon focus on the case $\hilberts = \frac{3}{2}$ only.
Moreover, we think of $\diffoffset$ being close to $0$ and of $\strongp$ being close to $2$.
By the Morrey embedding theorem~\cite[Theorem~6.5]{MR2944369},
the space $\SoboC[\strongs,\strongp]$ embeds continuously into $\HolderC[1,\alpha]$
where $\alpha \ceq \strongs  - 1 - 1/\strongp \in \intervaloo{0,1}$.
Thus, the \emph{configuration space}
$\ConfSpace$ defined in~\eqref{eq:confspace}
is well-defined and an open subset of $\SoboC[\strongs,\strongp]$.
We consider the Banach spaces
\begin{align*}
	\XC \ceq \SoboC[\strongs,\strongp],
	\quad
	\HC \ceq \SoboC[\hilberts,2],
	\qand
	\YC \ceq \SoboC[\weaks,\weakp],	
\end{align*}
where $q \ceq (1- 1/p)^{-1}$ denotes the Hölder conjugate of $p$.
For $\Curve \in \ConfSpace$, we will equip these spaces with the norms
\begin{align}
	\nnorm{\cdot}_{\Xnorm{\Curve}} 
 	\ceq 
 	\nnorm{\cdot }_{\Sobo[\strongs,\strongp][\Curve]},
	\quad
	\nnorm{\cdot}_{\Hnorm{\Curve}} 
	\ceq 
	\nnorm{\cdot }_{\Sobo[\hilberts,2][\Curve]},
	\quad\text{and}\quad
	\nnorm{\cdot}_{\Ynorm{\Curve}} 
	\ceq 
	\nnorm{\cdot }_{\Sobo[\weaks,\weakp][\Curve]}.
	\label{eq:XHYNorms}
\end{align}
Their continuous dual spaces will be denoted by $\XCd$, $\HCd$, and $\YCd$.
Since $\ConfSpace \subset \XC$ is an open set, its tangent space $T_\Curve \ConfSpace$ is identical to $\XC$ for each $\Curve \in \ConfSpace$.
By the Sobolev embedding theorem, the canonical embeddings
\begin{align}\label{eq:canonical-embeddings}
	\Ri[\ConfSpace] \colon \XC \hookrightarrow \HC
	\qand
	\Rj[\ConfSpace] \colon \HC \hookrightarrow \YC
\end{align}
are well-defined and continuous with dense images.
We point out that $\HC$ is a Hilbert space; suitable scalar products on this space will play a pivotal role in defining the Sobolev gradients of the Möbius energy (see Section~\ref{sec:Metric}).

There are several reasons for picking the parameters $\diffoffset$ and $\strongp$ as in \eqref{eq:Parameters}:
So far, 
it is only clear that $\strongp \geq 2$ and $\diffoffset \geq 0$ are necessary for the existence of the continuous embeddings $\Ri[\ConfSpace]$ and $\Rj[\ConfSpace]$ while
$\strongs - 1/\strongp >1$ is necessary for the Morrey embedding $\ConfSpace \hookrightarrow \SoboC[1,\infty]$. 
In addition to that, we require $\diffoffset > 0$ in order to be able to use certain product rules for bilinear maps of the form
$B \colon \Sobo[\strongs,\strongp] \times \Sobo[\weaks,\weakp] \to \Sobo[\weaks,\weakp]$ 
and
$B \colon \Sobo[\strongs,\strongp] \times \Sobo[\hilberts,2] \to \Sobo[\hilberts,2]$ 
as discussed in \autoref{lem:MultiplicationLemma}.
Indeed, the requirements $\strongs - 1/\strongp >1$ and $\diffoffset >0$ allow us to treat all occurring nonlinearities in a satisfactory way.
The condition $\strongp < \infty$ guarantees that all involved Banach spaces are reflexive and separable.

%% file: Energy.tex
\section{Energy}\label{sec:Energy}

From now on, if not stated otherwise, we fix $\hilberts = \frac{3}{2}$ and suppose that $\diffoffset >0$ and $\strongp\geq 2$.
Our principal aim in this section is to investigate the Möbius energy 
\begin{align}
	\Energy \colon \ConfSpace \to \R,
	\qquad
	\Energy(\Curve) 
	\ceq 
	\textstyle
	\int_\Torus
	E(\Curve)\, \varOmega_\Curve
	\quad
	\text{where}
	\quad
	E(\Curve)
	\ceq	
	\frac{1}{\nabs{\triangle \Curve}^2}	- \frac{1}{\varrho_\Curve^2}
	\label{eq:EnergyDensity}
\end{align}
along with its first two derivatives.
The first two variations of the Möbius energy have been discussed under various regularity assumptions before, cf.~\cite{MR3461038,MR1733697,MR3394390}.
The first variation is typically given in terms of principal-value integrals.
Here, by keeping everything in weak (or variational) formulation, we can work with very low regularity assumptions and avoid principal-value integrals altogether.

\begin{theorem}\label{theo:DEnergy}
The following statements hold true:
\begin{enumerate}
	\item \label{item:OHaraisFrechetDiffable}
	The Möbius energy $\Energy \colon \ConfSpace \to \R$ is Fréchet differentiable.
	\item \label{item:DOHaraisExtendable}
	The linear functional $\RX[\Energy][\Curve] \ceq D\Energy (\Curve) \in \XCgd$ can be continuously extended to a functional $\RY[\Energy][\Curve] \in \YCgd$.
	In particular, this shows that $\RY[\Energy] \colon \ConfSpace \to \YCd$, $\Curve \mapsto \RY[\Energy][\Curve]$ is a (nonlinear) differential operator of order at most $(\strongs) + (\weaks) = 3$.
	\item \label{item:YEislocallyLipschitz}
	The mapping $\RY[\Energy]\colon \ConfSpace \to \YCd$ is locally Lipschitz continuous.
\end{enumerate}
\end{theorem}
\begin{proof}
We are going to show that the energy density
$E \colon \ConfSpace \to \Lebesgue[1][\Curve][\Torus][\R]$
is Fr\'{e}chet differentiable.
This will also imply that $\Energy$ is Fréchet differentiable with derivative identical to the linear form $\RX[\Energy][\Curve] \in \XCd$ defined by
\begin{align}
	\RX[\Energy][\Curve] \, u \ceq
	\textstyle
	\int_\Torus DE(\Curve)\, u \, \varOmega_\Curve
	+
	\int_\Torus
	E(\Curve)\,
	\bigparen{
		\ninnerprod{\cD_\Curve \Curve, \cD_\Curve u} \circ \prx
		+
		\ninnerprod{\cD_\Curve \Curve, \cD_\Curve u} \circ \pry
	} \, \varOmega_\Curve
	.
	\label{eq:DEnergy}
\end{align}
We do so by following a ``shoot first ask questions later'' approach.	
To this end, we first investigate pointwise derivatives of $E(\Curve)$.
For $k \in \N_0$ and $u_1, \dotsc,u_k \in \XCg$, we abbreviate
\begin{align*}
	F_k(\Curve;u_1,\dotsc,u_k)(x,y) &\ceq D^k \bigparen{\Curve \mapsto E(\Curve)(x,y)}(\Curve) \, (u_1,\dotsc,u_k)
	\qand
	\\
	G_k(\Curve;u_1,\dotsc,u_k)(x,y) 
	&\ceq 
	\textstyle
	\int_{I_\Curve(x,y)} D^k \bigparen{\Curve \mapsto \omega_\Curve}(\Curve) \, (u_1,\dotsc,u_k).
\end{align*}
Recall that the $\Sobo[s+\diffoffset,p]$-norm dominates the $\Holder[1]$-norm.
Thus, due to the definition of the geodesic distance in~\eqref{eq:geoddist}, for each point $(x,y)$ in the open set
\begin{align*}
	\varSigma \ceq \set{(x,y) \in \Torus | x \neq y \; \text{and} \; \varrho_\Curve(x,y) < \ell}
	\quad
	\text{where}
	\quad
	\textstyle
	\ell \ceq \frac{1}{2} \int_\Circle \LineElement_\Curve,
\end{align*}
there is an open neighborhood $\cU(x,y)$ of $\ConfSpace$ such that $\Curve \mapsto I_\Curve$ is constant on $\cU(x,y)$.
Consequently, sufficiently small perturbations of $\Curve$ do not affect the integration domain of $G_k(\Curve;\cdots)$.
Utilizing the formulas
\begin{align}
	D(\Curve \mapsto \LineElement_\Curve)(\Curve) \, u
	=
	\ninnerprod{\cD_\Curve \Curve, \cD_\Curve u} \LineElement[\Curve],
	\quad
	D(\Curve \mapsto \cD_\Curve v)(\Curve) \, u
	=
	- \ninnerprod{\cD_\Curve \Curve, \cD_\Curve u}  \cD_\Curve v,
	\label{eq:DerivativeomegaandD}
\end{align}
we obtain
\begin{align*}
	G_1(\Curve;u_1)
	&= 
	\textstyle	
	\int_{I_\Curve} \ninnerprod{\cD_\Curve \Curve , \cD_\Curve u_1} \, \LineElement_\Curve,
	\qand
	\\
	G_2(\Curve;u_1,u_2)
	&=
	\textstyle
	\int_{I_\Curve} 
		\bigparen{ \ninnerprod{\cD_\Curve u_1,\cD_\Curve u_2} - \ninnerprod{\cD_\Curve \Curve , \cD_\Curve u_1}\, \ninnerprod{\cD_\Curve \Curve , \cD_\Curve u_2}} 
	\, \LineElement_\Curve
	.
\end{align*}
By pointwise differentiation at $(x,y) \in \varSigma$ and by observing that $\varSigma$ has full measure, we are lead to the following identities which hold almost everywhere on $\Torus$:
\begin{align*}
	F_1(\Curve;u_1)
	&=
	2\, \Bigparen{
		\tfrac{1}{\varrho_\Curve^4} \, %
		{\varrho_\Curve \, G_1(\Curve;u_1)}
		-
		\tfrac{1}{\nabs{\triangle \Curve}^4} \, \ninnerprod{\triangle \Curve , \triangle u_1}
	}
	\qand
	\\
	F_2(\Curve;u_1,u_2)
	&=
	8\,
	\seminorm{
		\tfrac{1}{\nabs{\triangle \Curve}^6} \, \ninnerprod{\triangle \Curve , \triangle u_1}\, \ninnerprod{\triangle \Curve , \triangle u_2}
		- 
		\tfrac{1}{\varrho_\Curve^6} \, %
		{\varrho_\Curve \, G_1(\Curve;u_1)%
		\cdot\varrho_\Curve \, G_1(\Curve;u_2)}
	}
	\\
	&\quad
	-
	2\,
	\seminorm{
		\tfrac{1}{\nabs{\triangle \Curve}^4} \, \ninnerprod{\triangle u_1 , \triangle u_2}
		-
		\tfrac{1}{\varrho_\Curve^4} \, 
		\Bigparen{G_1(\Curve;u_1) \, G_1(\Curve;u_2) + \varrho_\Curve \, G_2(\Curve;u_1,u_2)}		
	}
	.
\end{align*}
Claim~1 below will imply that $F_1(\Curve;u_1)$ is indeed a candidate for $DE(\Curve) \, u_1$. 
Moreover, it guarantees that the right hand side of \eqref{eq:DEnergy} makes sense even if one replaces $u \in \XCg$ by $w \in \YCg$ so that $\RX[\Energy][\Curve] \in \XCgd$ has a unique continuous extension to an element $\RY[\Energy][\Curve] \in \YCgd$.

\textbf{Claim 1:} 
\emph{There exists a $\Curve$-dependent $C \geq 0$
such that $\nnorm{F_1(\Curve;u_1)}_{\Lebesgue[1][\Curve]} \leq C \, \nnorm{u_1}_{\Sobo[s-\diffoffset,q][\Curve]}$ holds for all $u_1 \in  \XCg$.}
\newline
We split $F_1$ as follows:
\begin{align*}
	F_1(\Curve;u_1)
	=
	2 \, 
	\tfrac{1}{\varrho_\Curve^4} \,
	\Bigparen{
		(\varrho_\Curve \, G_1(\Curve;u_1))  - \ninnerprod{\triangle \Curve , \triangle u_1}
	}		
	- 2 \, 
	\Bigparen{
		\tfrac{1}{\nabs{\triangle \Curve}^4}
		-
		\tfrac{1}{\varrho_\Curve^4}	
	}
	\, \ninnerprod{\triangle \Curve , \triangle u_1}
	.
\end{align*}
The desired bound for the first summand is derived in \autoref{lem:PrincipalPartofEnergyDerivative}.
The second summand can be treated with \autoref{lem:IntegrationbyParts} because it has the form
$2\, \cB^{\alpha,\beta}_\Curve(\Curve,u)$ with $\alpha = 0$ and $\beta = 2$.

We would like to use the $\Lebesgue[1]$-norm of $F_2$ to bound remainder terms of Taylor expansions. This will make use of the following claim.

\textbf{Claim~2:} 
\emph{There exists a number $\varXi(\Curve) {}>{} 0$, continuous in $\Curve$, such that\,
for all $u_1$, $u_2 \in \XCg$, we have
 $\nnorm{F_2(\Curve;u_1,u_2)}_{\Lebesgue[1][\Curve]} \leq \varXi(\Curve) \, \nnorm{u_1}_{\Xnorm{\Curve}} \,  \nnorm{u_2}_{\Ynorm{\Curve}}$.}
\newline
We may split $F_2(\Curve;u_1,u_2)$ into the following four summands:
\begin{align}
	&
	8 \, \Bigparen{
		\tfrac{1}{\nabs{\triangle \Curve}^6} - \tfrac{1}{\varrho_\Curve^6}
	} \, \ninnerprod{\triangle \Curve , \triangle u_1}\, \ninnerprod{\triangle \Curve , \triangle u_2}
	\label{eq:F2-1}
	\\
	&\qquad
	- 2 \, \Bigparen{
		\tfrac{1}{\nabs{\triangle \Curve}^4} - \tfrac{1}{\varrho_\Curve^4}
	} \, \ninnerprod{\triangle u_1 , \triangle u_2}	
	\label{eq:F2-2}
	\\
	&\qquad			
	+
	8 \, \tfrac{1}{\varrho_\Curve^6} \, \Bigparen{
		\ninnerprod{\triangle \Curve , \triangle u_1}\, \ninnerprod{\triangle \Curve , \triangle u_2}  - \varrho_\Curve \, G_1(\Curve;u_1) \, \varrho_\Curve \, G_1(\Curve;u_2)}
	\label{eq:F2-3}
	\\
	&\qquad				
	+ 2 \, \tfrac{1}{\varrho_\Curve^4} \, 
	\Bigparen{
		{G_1(\Curve;u_1) \, G_1(\Curve;u_2) + \varrho_\Curve \, G_2(\Curve;u_1,u_2)}
		-	
		\ninnerprod{\triangle u_1 , \triangle u_2} 
	}
	\label{eq:F2-4}	
	.
\end{align}
Here, \eqref{eq:F2-1} and \eqref{eq:F2-2} are again of the type discussed in \autoref{lem:IntegrationbyParts}, namely with $\alpha = 0$, $\beta = 4$ and $\alpha = 0$, $\beta = 2$, respectively. 
We may factorize~\eqref{eq:F2-3} as
\begin{align*}
	8\Bigparen{
		\Biginnerprod{\tfrac{\triangle \Curve}{\varrho_\Curve}, \tfrac{\triangle u_1}{\varrho_\Curve}}  + 				\tfrac{G_1(\Curve;u_1)}{\varrho_\Curve}
	}
	\cdot	
	\Bigparen{	
	\tfrac{1}{\varrho_\Curve^4} \, 	
	\bigparen{
		\ninnerprod{\triangle \Curve , \triangle u_2}  - \varrho_\Curve \, G_1(\Curve;u_2)
	}
	}
\end{align*}
to discover that we have discussed its second factor already.
Up to a $\gamma$-dependent constant, its first factor is bounded by $2 \, \nnorm{\cD_\Curve u_1}_{\Lebesgue[\infty]}$, thus dominated by $\nnorm{u_1}_{\Xnorm{\Curve}}$.
With $H \ceq \int_{I_\Curve}\!\ninnerprod{\cD_\Curve u_1 , \cD_\Curve u_2} \, \LineElement[\Curve]$, we can split \eqref{eq:F2-4} into the following two summands:
\begin{align}
	\tfrac{2}{\varrho_\Curve^4} \, 
	\bigbrackets{
		\varrho_\Curve \, H
		-
		\ninnerprod{\triangle u_1 , \triangle u_2}
	}
	+ \tfrac{2}{\varrho_\Curve^4} \, 
	\bigbrackets{
		G_1(\Curve;u_1) \, G_1(\Curve;u_2) 
		-
		\varrho_\Curve \, (H - G_2(\Curve;u_1,u_2))
	}
	.
	\label{eq:Energy1}
\end{align}
The first summand of \eqref{eq:Energy1} can be treated with \autoref{lem:PrincipalPartofEnergyDerivative}.
With $\varphi_i \ceq \ninnerprod{\cD_\Curve \Curve , \cD_\Curve u_i}$
and the identities
$\varrho_\Curve = \int_{I_\Curve} \LineElement[\Curve](s)$  
and
$H-G_2(\Curve;u_1,u_2) = \int_{I_\Curve} \varphi_1(t) \, \varphi_2(t)\, \LineElement[\Curve](t)$
the second summand of \eqref{eq:Energy1} simplifies to
\begin{align*}
	\MoveEqLeft
	\textstyle
	\tfrac{2}{\varrho_\Curve^4} \, 
	{
		{\int_{I_\Curve} \varphi_1(s) \, \LineElement[\Curve](s)}
		\,
		{\int_{I_\Curve} \varphi_2(t) \, \LineElement[\Curve](t)}
		-	
		{\int_{I_\Curve} \LineElement[\Curve](s)}
		\,
		{\int_{I_\Curve} \varphi_1(t) \, \varphi_2(t) \, \LineElement[\Curve](t)}
	}
	\\
	&=
	\textstyle	
	\tfrac{1}{\varrho_\Curve^4} \, 
	\int_{I_\Curve^2}
	 	\nparen{\varphi_1(s)- \varphi_1(t)} \, \varphi_2(t)
	\,\varOmega_\Curve(s,t)
	+
\tfrac{1}{\varrho_\Curve^4} \, 
	\int_{I_\Curve^2}
	 	\nparen{\varphi_1(t)- \varphi_1(s)} \, \varphi_2(s)
	\,\varOmega_\Curve(t,s)	
	\\
	&=
	\textstyle		
	-	
	\tfrac{1}{\varrho_\Curve^4} \, 
	\int_{I_\Curve^2}
	 	(\triangle \varphi_1) \, (\triangle \varphi_2)
	\,\varOmega_\Curve.
\end{align*}
Now the same techniques as in \autoref{lem:PrincipalPartofEnergyDerivative} and the product rule \autoref{lem:MultiplicationLemma} imply
\begin{align*}
	\textstyle
	\int_{\Torus}
	\frac{1}{\varrho_\Curve(x,y)^{4}}
	\,
	\bigabs{\textstyle
	\int_{I_\Curve^2(x,y)}
		\triangle \varphi_1  \, \triangle \varphi_2
	\, \varOmega_\Curve
	}
	\,
	\varOmega_\Curve	(x,y)
	\leq
	C \, \nseminorm{u_1}_{\Sobo[s+\diffoffset,p][\Curve]} \, \nseminorm{u_2}_{\Sobo[s-\diffoffset,q][\Curve]},
\end{align*}
which proves the claim.

\textbf{Claim~3:} \emph{$E$ is Fréchet differentiable with $DE(\Curve) \,u = F_1(\Curve;u)$.}
\newline
It suffices to show that there is a $C \geq 0$ such that
$
	\nnorm{
		E(\Curve + u) - E(\Curve) - F_1(\Curve;u)
	}_{\Lebesgue[1][\Curve]}
	\leq C \, \nnorm{u}_{\Xnorm{\OtherCurve}}^2
$
holds for all sufficiently small $u \in \XCg$.
Because $\ConfSpace \subset \XCg$ is open and $\varXi$ is continuous, we may find an $\varepsilon >0$ such that 
for all $\nnorm{u}_{\Xnorm{\Curve}} < \varepsilon$
we have
$\OtherCurve \ceq \Curve+u \in \ConfSpace$ and $\varXi(\OtherCurve) \leq 	2 \, \varXi(\Curve)$. 
By shrinking $\varepsilon$ if necessary, we may achieve that all the  densities $\varOmega_{\OtherCurve}$ and the norms $\nnorm{\cdot}_{\Xnorm{\OtherCurve}}$ are equivalent for all such $u$, i.e., there are $c>0$ and $\varLambda>1$ such that
\begin{gather}
	(1 -  c\, \nnorm{\cD_\Curve u}_{\Lebesgue[\infty]}) \, \LineElement[\Curve]
	\leq 
	\LineElement[\OtherCurve]
	\leq
	(1 + c\, \nnorm{\cD_\Curve u}_{\Lebesgue[\infty]})  \, \LineElement[\Curve]
	\label{eq:DensityEquivalence}	
	\\
	\varLambda^{-1} \, \varOmega_{\Curve}
	\leq 
	\varOmega_{\OtherCurve}
	\leq
	\varLambda \, \varOmega_{\Curve},
	\quad 
	\text{and}
	\quad
	\varLambda^{-1} \, \nnorm{\cdot}_{\Xnorm{\OtherCurve}}
	\leq 
	\nnorm{\cdot}_{\Xnorm{\Curve}}
	\leq 
	\varLambda \, \nnorm{\cdot}_{\Xnorm{\OtherCurve}}
	.
	\label{eq:DEnormequivalence}
\end{gather}
For the remainder of the proof, we let $u \in \XCg$ be of length $\nnorm{u}_{\Xnorm{\Curve}} < \varepsilon$ and abbreviate $\OtherCurve \ceq \Curve + u$ and $\Curve_t \ceq \Curve + t\, u$.

\begin{figure}[t!]
\capstart
\begin{center}
\begin{footnotesize}
\begin{subfigure}{0.235\textwidth}%
	\def\svgwidth{\textwidth}
	\input{BadSet1_pdf.tex}%
	\subcaption{}
	\label{fig:BadSet1}%
\end{subfigure}%
\begin{subfigure}{0.235\textwidth}%
	\def\svgwidth{\textwidth}
	\input{BadSet2_pdf.tex}%
	\subcaption{}
	\label{fig:BadSet2}%
\end{subfigure}%
\begin{subfigure}{0.235\textwidth}%
	\def\svgwidth{\textwidth}
	\input{BadSet3_pdf.tex}%
	\subcaption{}
	\label{fig:BadSet3}%
\end{subfigure}%
\begin{subfigure}{0.295\textwidth}%
	\def\svgwidth{\textwidth}
	\input{BadSet_pdf.tex}	
	\subcaption{}
	\label{fig:BadSet0}%
\end{subfigure}%
\end{footnotesize}
\end{center}
\caption{Illustration why handling $\Curve \mapsto \varrho_\Curve$ correctly is so difficult. 
\subref{fig:BadSet1} Two points $\Curve(x)$ and $\Curve(y)$ on a circular curve and the geodesic arc $\Curve(I_\Curve(x,y))$ connecting them.
\subref{fig:BadSet2} A displacement vector field $u$ along $\Curve$.
\subref{fig:BadSet3} The geodesic arc $\OtherCurve(I_\OtherCurve(x,y))$ connecting $\OtherCurve(x)$ and $\OtherCurve(y)$ on the displaced curve $\OtherCurve = \Curve + u$. Observe that $I_\Curve(x,y) \neq I_\OtherCurve(x,y)$, so $(x,y)$ belongs to the ``bad'' set $V(u)$.
\subref{fig:BadSet0} The points $(x,y)$, $(y,x)$ and the ``bad'' set $V(u)$, plotted over the relief of $\varrho_\Curve$ in the parameterization domain~$\Torus$.}
\label{fig:BadSet}
\end{figure}
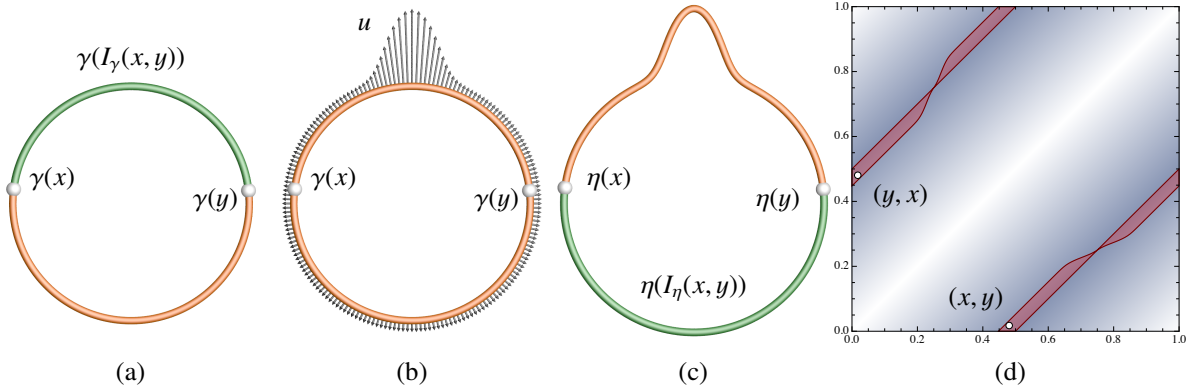

Now we split the integration domain 
$\Torus = U(u) \cup V(u)\cup \set{(x,x) | x \in \Circle} $ into a ``good'' part $U(u)$,  a ``bad'' part $V(u)$, and an ``ugly`` part, the diagonal of $\Torus$ (cf.~\autoref{fig:BadSet}). Since the latter is a null set, it may be neglected.
The other two parts are defined as follows:
\begin{align*}
	U(u)
	&\ceq
	\set{ 
		(x,y) \in \Torus | 
		\text{
			$x \neq y$
			and 	
			for all $t \in \intervalcc{0,1}$:
			$I_{\Curve_t}(x,y) = I_{\Curve}(x,y)$
		}
	}
	\qand
	\\
	V(u) 
	&\ceq 
	\set{ 
		(x,y) \in \Torus | 
		\text{
			there is a $t \in \intervalcc{0,1}$: $I_{\Curve_t}(x,y) \neq I_{\Curve}(x,y)$
		}
	}.	
\end{align*}
On the ``good'' part, we may apply Taylor's theorem along with  Claim~2 and \eqref{eq:DEnormequivalence} to obtain:
\begin{align*}
	\MoveEqLeft
	\textstyle
	\int_{U(u)} 
		\nabs{
			E(\OtherCurve) - E(\Curve) - F_1(\Curve; u)
		}
	\, \varOmega_{\Curve }
	\leq \textstyle 
		\int_{U(u)} 
		\int_0^1 \nabs{F_2(\Curve_t ; u,u)} \, \dd t
		\, \varOmega_{\Curve}
	\\
	&\leq \textstyle 
	\varLambda^2 
	\int_0^1 
		\int_{U(u)} \nabs{F_2(\Curve_t ; u,u)} \, \varOmega_{\Curve_t}
	\, \dd t
	\leq \textstyle 
	\varLambda^2 
	\int_0^1 \varXi(\Curve_t) \, \nnorm{u}_{\Xnorm{\Curve_t}}^2 \, \dd t
	\leq 2 \, \varXi(\Curve) \, \varLambda^4 \, \nnorm{u}_{\Xnorm{\Curve}}^2.
\end{align*}
Here the first estimate is only admissible on the ``good'' set $U(u)$ as the intrinsic distance~\eqref{eq:geoddist} is differentiable. However,
we cannot argue this way on the ``bad'' set $V(u)$.
Instead, we observe that $V(u)$ has positive distance from the diagonal.
Thus, by Claim~4 below, $\Curve \mapsto E(\Curve)|_{V(u)}$ is Lipschitz continuous as a map into $\Lebesgue[\infty][][V(u)][\R]$.
Together with Claim~5 below, which states that $V(u)$ is a small set, 
the triangle inequality 
$
\nabs{E(\OtherCurve) - E(\Curve) - F_1(\Curve; u)} 
	\leq  
	\nabs{E(\OtherCurve) - E(\Curve)} 
	+
	\nabs{F_1(\Curve; u)}
$ leads us to
\begin{align*}
	\MoveEqLeft
	\textstyle
	\int_{V(u)} 
		\nabs{
			E(\OtherCurve) - E(\Curve) - F_1(\Curve; u)
		}
	\, \varOmega_\Curve
	\leq C \, \nnorm{\cD_\Curve u}_{\Lebesgue[\infty]}^2,
\end{align*}
which proves the claim.

\textbf{Claim~4:}
\emph{$\Curve \mapsto \varrho_\Curve$ is locally Lipschitz continuous as a mapping into $\Lebesgue[\infty]$.}

As $\Torus\setminus\varSigma$ is a set of measure
zero, we may restrict our attention to $(x,y) \in \varSigma$.
For $\OtherCurve \ceq \Curve + u$ and
$(x,y) \in \varSigma$ there are two cases: The first case is $I_{\OtherCurve}(x,y) = I_{\Curve}(x,y)$. Then the bound $\nabs{\varrho_{\OtherCurve}(x,y)-\varrho_\Curve(x,y)} \leq C \,\nnorm{\cD_\Curve  u}_{\Lebesgue[\infty]}$ follows from the differentiability of $\LineElementC$.
The second case $I_{\OtherCurve}(x,y) \neq I_{\Curve}(x,y)$ is a bit more elaborate.
We abbreviate $I \ceq I_\Curve(x,y)$ and denote its complement by $J \ceq \Circle \setminus I$.
With \eqref{eq:DensityEquivalence}, we obtain
\begin{align*}
	\textstyle
	\int_{J} \LineElement_{\Curve} - c \, \nnorm{\cD_\Curve u}_{\Lebesgue[\infty]} \int_{J} \LineElement_{\Curve}
	\leq 
	\int_{J} \LineElement_{\OtherCurve}
	\leq
	\int_{I} \LineElement_{\OtherCurve}	
	\leq
	\int_{I} \LineElement_{\Curve} + c \, \nnorm{\cD_\Curve u}_{\Lebesgue[\infty]} \int_{I} \LineElement_{\Curve}.
\end{align*}
Here we used \eqref{eq:DensityEquivalence} for the first and the third inequality.
This shows
$
	\nabs{
		\int_{J} \LineElement_{\Curve}
		\!
		- 
		\!		
		\int_{I} \LineElementC
	}
	=
		\int_{J} \LineElement_{\Curve}	
		-
		\int_{I} \LineElement_{\Curve}
	\leq
	2 \, c \, \ell \, \nnorm{\cD_\Curve u}_{\Lebesgue[\infty]}
$.
By \eqref{eq:DensityEquivalence}, we obtain
$
	\nabs{
		\int_{J} \LineElement_{\OtherCurve}
		\!
		-
		\! 
		\int_{J} \LineElement_{\Curve}
	}
	\leq
	2
	\, c \,  \ell \, \nnorm{\cD_\Curve u}_{\Lebesgue[\infty]}
$.
Combining these latter two inequalities proves Claim~4:
\begin{align*}
	\MoveEqLeft
	\textstyle
	\nabs{
		\varrho_{\OtherCurve}(x,y) 
		\!
		- 
		\!
		\varrho_{\Curve}(x,y)
	}
	=
	\textstyle	
	\nabs{
		\int_{J} \LineElement[\OtherCurve]
		\!		
		-
		\!
		\int_{I} \LineElement[\Curve]
	}
	\leq
	\textstyle	
	\nabs{
		\int_{J} \LineElement[\OtherCurve]
		\! 	
		-
		\!		
		\int_{J} \LineElement[\Curve]
	}
	+
	\nabs{
		\int_{J} \LineElement[\Curve]
		\!		
		-
		\!		
		\int_{I} \LineElement[\Curve]
	}		
	\leq C \,\nnorm{\cD_\Curve u}_{\Lebesgue[\infty]}
	.
\end{align*}

\textbf{Claim~5:} \emph{$\int_{V(u)} \, \varOmega_\Curve \leq C \, \nnorm{\cD_\Curve u}_{\Lebesgue[\infty]}$ for $\nnorm{u}_{\Xnorm{\Curve}} < \varepsilon$.}
\newline
A Taylor expansion of the integrand leads to
\begin{align*}
	\nabs{
		\textstyle
		\int_{J} \LineElement[\Curve + t \, u]
		-
		\int_{J} \LineElement[\Curve] - \int_{J} \ninnerprod{\cD_\Curve \Curve ,\cD_\Curve (t\, u)} \,  \LineElement[\Curve]
	}
	\leq C \, \nnorm{\cD_\Curve (t\, u)}_{\Lebesgue[\infty]}^2.
\end{align*}
Now let $(x,y) \in V(u)$. Then there is a $t \in \intervalcc{0,1}$ such that
$I_{\Curve + t \, u}(x,y) = J$ and we have
\begin{align*}
	\varrho_{\Curve+ t \, u}(x,y)
	=
	\textstyle	
	\int_{J} \LineElement[\Curve + t \, u]
	&\geq
	\textstyle	
	\int_{J} \LineElement[\Curve] + \int_{J} \ninnerprod{\cD_\Curve \Curve ,\cD_\Curve (t \, u)} \,  \LineElement[\Curve]
	- 2 \, C \, \ell \, \nnorm{\cD_\Curve (t \,  u)}_{\Lebesgue[\infty]}^2
	\\
	&\geq
	\ell - t \, C \,\nnorm{\cD_\Curve  u}_{\Lebesgue[\infty]}
	\geq
	\ell - C \,\nnorm{\cD_\Curve  u}_{\Lebesgue[\infty]}
	.
\end{align*}
Together with Claim~4 this implies that $V(u)$ is contained in a narrow band around the set $\set{(x,y) | \varrho_\Curve(x,y) = \ell }$ 
whose area is proportional to $\nnorm{\cD_\Curve  u}_{\Lebesgue[\infty]}$ (cf.~\autoref{fig:BadSet0}).
	
\textbf{Claim~6:} \emph{There is a 
$C \geq 0$  such that
$\nabs{
		(\RY[\Energy][\Curve + u]) \, w
		-
		(\RY[\Energy][\Curve]) \, w
	}
	\leq
	C \, \nnorm{u}_{\Xnorm{\Curve}} \, \nnorm{w}_{\Ynorm{\Curve}}
	$
holds for all $w \in \YCg$ and all $u \in \XCg$ with $\nnorm{u}_{\Xnorm{\Curve}}< \varepsilon$.}
\newline
With the operator $\varPsi_\Curve w \ceq \bigparen{
		\ninnerprod{\cD_\Curve \Curve , \cD_\Curve w} \circ \prx 
		+ 
		\ninnerprod{\cD_\Curve \Curve , \cD_\Curve w} \circ \pry
	}$,
we may write
$
	\RY[\Energy][\Curve] \, w
	= 
	\textstyle
	\int_\Torus \bigparen{ F_1(\Curve;w) + E(\Curve) \, (\varPsi_\Curve w)} \, \varOmega_\Curve
$. 
By the triangle inequality and the fundamental theorem of calculus, we may bound $\nabs{
		(\RY[\Energy][\OtherCurve]) \, w
		-
		(\RY[\Energy][\Curve]) \, w
	}$ from above by
\begin{align}
	&
	\abs{
	\textstyle
	\int_0^1
	\frac{\dd}{\dd t}
	\,
	\Bigbrackets{		
	\int_{U(u)}
		\Bigparen{
			F_1(\Curve_t;w) 
			+ 
			E(\Curve_t) \, (\varPsi_{\Curve_t} w)
		}
	\, \varOmega_{\Curve_t}
	}
	\, \dd t
	}
	\label{eq:YELipschitz1}
	\\ 
	&\qquad
	\textstyle
	+
	\int_{V(u)} 
	\nabs{ 
		F_1(\Curve;w) 
		+ 
		E(\Curve) \, (\varPsi_\Curve w)
	} \, \varOmega_\Curve
	+
	\int_{V(u)}
	\nabs{ 
		F_1(\OtherCurve;w)
		+ 
		E(\OtherCurve) \, (\varPsi_{\OtherCurve} w)
	} 
	\, \varOmega_{\OtherCurve}
	\label{eq:YELipschitz4}
	.
\end{align}
Recall that on the ``good'' set $U(u)$ we may interchange differentiation and integration. Hence, we have
\begin{align*}
	\MoveEqLeft
	\textstyle
	\frac{\dd}{\dd t}
	\int_{U(u)}
		\Bigparen{
			F_1(\Curve_t;w) 
			+ 
			E(\Curve_t) \, (\varPsi_{\Curve_t} w)
		}
	\, \varOmega_{\Curve_t}
	\\
	&=
	\textstyle	
	\int_{U(u)}
	\Bigparen{
		F_2(\Curve_t; u, w) 
		+ 
		F_1(\Curve_t; u)  \, (\varPsi_{\Curve_t} w)
		+
		E(\Curve_t) \, ( \tfrac{\dd}{\dd t} \varPsi_{\Curve_t} w)
	}
	\, \varOmega_{\Curve_t}
	\\
	&\qquad
	\textstyle
	+
	\int_{U(u)}
	\Bigparen{
		F_1(\Curve_t;w)  
		+
		E(\Curve_t) \, (\varPsi_{\Curve_t} w)
	}
	\, (\varPsi_{\Curve_t} u)		 \, \varOmega_{\Curve_t}
	.
\end{align*}
Now it follows from the other claims above that \eqref{eq:YELipschitz1} is bounded by a multiple of
$\nnorm{u}_{\Xnorm{\Curve}} \, \nnorm{w}_{\Ynorm{\Curve}}$.
For \eqref{eq:YELipschitz4}, we exploit that $V(u)$ has finite, positive distance to the diagonal of $\Torus$: This implies that the quantities $\nabs{\triangle \Curve(x,y)}$ and $\varrho_\Curve(x,y)$ are uniformly bounded away from zero and that this remains true for sufficiently small perturbations $\eta= \Curve + u$ of $\Curve$. This is why we can 
express $\bigbrackets{\bigparen{ F_1(\Curve;w) + E(\Curve) \, (\varPsi_\Curve w)} \, \varOmega_\Curve}(x,y)$
by
\begin{align*}
	\MoveEqLeft
	\biginnerprod{
		Z_1\bigparen{\Curve'(x),\Curve'(y),\nabs{\triangle \Curve(x,y)},\varrho_\Curve(x,y)},  
		w'(x)
	}
	\, \varOmega_\Curve(x,y)
	\\
	&\qquad	
	+
	\biginnerprod{
		Z_2\bigparen{\Curve'(x),\Curve'(y),\nabs{\triangle \Curve(x,y)},\varrho_\Curve(x,y)}, 
		w'(y)
	} 
	\, \varOmega_\Curve(x,y)
	\quad
	\text{for $(x,y) \in V(u)$}
\end{align*}
with Lipschitz continuous functions $Z_1$ and $Z_2$.
So the same applies to~$\eta$, and Claim~5 shows that \eqref{eq:YELipschitz4} is bounded by
\begin{align*}
	C \, \nnorm{1_{V(u)}}_{L^p_\Curve} \, \nnorm{\cD_\Curve u}_{\Lebesgue[\infty]} \, \nnorm{w}_{\Sobo[1,q][\Curve]}
	\leq C \, \nnorm{\cD_\Curve u}_{\Lebesgue[\infty]}^{1+1/p} \, \nnorm{w}_{\Sobo[1,q][\Curve]},
\end{align*}
which finally proves the claim.
\end{proof}

\begin{remark}
In fact, a bit more is true:
The mapping $\RY[\Energy] \colon \ConfSpace \to \YCd$ is even continuously Fréchet differentiable and what we have shown in Claim~6 above is that $D (\RY[\Energy])(\Curve) \colon \XCg \times \YCg \to \R$ is a continuous bilinear form.
Now we may conclude that its second derivative must satisfy
$D^2\Energy(\Curve)(u_1,u_2) = D(\RY[\Energy])(\Curve)(u_1,\Rj[\ConfSpace] \; \Ri[\ConfSpace] \; u_2)$
for $u_1$, $u_2 \in \XCg$
where $\Ri[\ConfSpace]$ and $\Rj[\ConfSpace]$ denote the canonical embeddings defined in~\eqref{eq:canonical-embeddings}.
Although this might be relevant for optimization methods based on Newton's method and also for the implicit integration of the $L^2$-gradient flow, we do not dive into details here.
\end{remark}

\subsection*{Details}

Here we state and prove the lemmas used in the proof of \autoref{theo:DEnergy} above.
The following is our main tool for dealing with the lower order terms that occur in $\RY[\Energy][\Curve]$.

\begin{lemma}\label{lem:IntegrationbyParts}
Suppose $s = \frac{3}{2}$ and \eqref{eq:Parameters}
with $\frac{1}{p}+\frac{1}{q}=1$.
Fix $k \in \N \cup \set{0}$, $\alpha$, $\beta \in \R$.
Let $\Curve \in \ConfSpace$
be an embedded curve and
let $b \colon \nparen{\prod_{i=1}^k \R^{m_i}} \times\R^d \to \R$ be a $(k+1)$-multilinear form. For operators $L_i$, 
$K \in \set{ 
	\triangle / \varrho_\Curve, 
	\varphi \mapsto \cD_\Curve \varphi \circ \prx
	,
	\varphi \mapsto \cD_\Curve \varphi \circ \pry
}$, $i \in {1,\dotsc,k}$ consider the following multilinear form
$
	\cB^{\alpha,\beta}_\Curve
	\colon 
	(\prod_{i=1}^k \XCd) \times \YCd \to \R
$:
\begin{align*}
	\textstyle
	\cB^{\alpha,\beta}_\Curve(\varphi_1,\dotsc,\varphi_k,\psi)
	\ceq 
	\int_\Torus
	b( L_1 \varphi_1,\dotsc,L_k \varphi_k, K \psi) 
	\, 
	\frac{\varrho_\Curve^{\alpha+\beta}}{\nabs{\triangle \Curve}^\alpha}
	\,
	\Bigparen{
		\frac{1}{\nabs{\triangle \Curve}^{2+\beta}}
		-
		\frac{1}{\varrho_\Curve^{2+\beta}}
	}
	\, \varOmega_\Curve.
\end{align*}
Then $\cB^{\alpha,\beta}_\Curve$ is well-defined and
there is a continuous function $\varXi \colon \ConfSpace \to \intervalco{0,\infty}$ such that
\begin{align*}
	\textstyle
	\nabs{\cB^{\alpha,\beta}_\Curve(\varphi_1,\dotsc,\varphi_k,\psi)}
	\leq
	\nnorm{b} \,
	\varXi(\Curve) \,
	\bigparen{\prod_{i=1}^k 
	\nnorm{\cD_\Curve \varphi_i}_{\Lebesgue[\infty]} }\,
	\nnorm{\cD_\Curve \psi}_{\Sobo[s-1-\diffoffset,q][\Curve]}
	.
\end{align*}
\end{lemma}

Here the expression $\nnorm{b}$ denotes the operator norm of the multilinear form~$b$. If $k=0$ we use the convention
$\nparen{\prod_{i=1}^0 \R^{m_i}} \times\R^d=\R^{d}$
and
$(\prod_{i=1}^0 \XCd) \times \YCd =  \YCd$.
In this case, $\cB^{\alpha,\beta}_\Curve$
only depends on~$\psi$.

\begin{proof}
We heavily rely on the techniques developed in the proof of Theorem~1.1 in \cite{MR2887901}.
With the function $\zeta \colon \intervaloo{0,\infty} \to \R$,
$
	\zeta(r) \ceq r^{2+\alpha} \, \frac{r^{2+\beta} - 1}{r^{2} - 1}
$, 
we write
\begin{align*}
	\textstyle
	\frac{\varrho_\Curve^{\alpha+\beta}}{\nabs{\triangle \Curve}^\alpha}
	\Bigparen{
		\frac{1}{\nabs{\triangle \Curve}^{2+\beta}}
		-
		\frac{1}{\varrho_\Curve^{2+\beta}}
	}
	=
	\tfrac{1}{2}
	\zeta \Bigparen{\tfrac{\varrho_\Curve}{\nabs{\triangle \Curve}}}
	\cdot
	\tfrac{1}{\varrho_\Curve^4}
	\,
	\Bigparen{
		2 \, \varrho_\Curve^2
		-
		2 \, \ninnerprod{\triangle \Curve, \triangle \Curve}
	}.
\end{align*}	
Denoting the shorter arc between $x$, $y \in \Circle$ by $I_\Curve$, we observe
\begin{align*}
	2 \, \varrho_\Curve^2 - 2 \, \ninnerprod{\triangle \Curve, \triangle \Curve}
	&=
	\textstyle
	\int_{I_\Curve^2} \bigparen{
		\nabs{\Tangent_\Curve(s)}^2
		+
		\nabs{\Tangent_\Curve(t)}^2
	} \, \varOmega_\Curve(s,t)
	-
	2 \, \int_{I_\Curve^2}
		\ninnerprod{\Tangent_\Curve(s), \Tangent_\Curve(t)}
	\, \varOmega_\Curve(s,t)
	\\
	&=
	\textstyle	
	\int_{I_\Curve^2}
		\nabs{\Tangent_\Curve(s) - \Tangent_\Curve(t)}^2
	\, \varOmega_\Curve(s,t).
\end{align*}
Since $\zeta$ is continuous and $\Curve$ is bi-Lipschitz, the factor $\zeta(\tfrac{\varrho_\Curve}{\nabs{\triangle \Curve}})$ is bounded.
Moreover, the functions $L_i \varphi_i$ are uniformly bounded.
So it suffices to bound the $\Lebesgue[1][\Curve]$-norm of 
\begin{align}
	\textstyle
	(K \psi) \cdot
	\tfrac{1}{\varrho_\Curve^4}
	\,
	\int_{I_\Curve^2}
		\nabs{\Tangent_\Curve(s) - \Tangent_\Curve(t)}^2
	\, \varOmega_\Curve(s,t)
	.
	\label{eq:IntegrationByParts2}
\end{align}
We abbreviate half the length of $\Curve$ by $\ell$ and denote by $\Geodesic_x \ceq \exp_x^\Curve \colon \intervaloo{-\ell,\ell} \to \Circle$ the Riemannian exponential map induced by $\varrho_\Curve$.
With $X$ such that $y = \Geodesic_x(X)$, we may write
$\varrho_\Curve(x,y) = \nabs{X}$
and
$K \,\psi (x,y) = \int_0^1 \cD_\Curve \psi(\xi_{x}(\theta_1 , X))\, \dd \theta_1$, where
$\xi_{x}(\theta_1 , X)$ is either $\Geodesic_{x}(\theta_1 \, X)$, $x$, or $\Geodesic_{x}(X) = y$.
Thus, we may rewrite \eqref{eq:IntegrationByParts2} as follows:
\begin{align*}
	\textstyle
	\int_{\intervalcc{0,1}^{3}}
	\cD_\Curve \psi(\xi_{x}(\theta_1 , X))
	\,
	\nabs{
		\Tangent_\Curve(\Geodesic_{x}(\theta_2 X))-\Tangent_\Curve(\Geodesic_{x}(\theta_3 X))
	}^2
	\, \tfrac{1}{\nabs{X}^2} \, \dd\theta^{3} \, \dd\theta^{2} \, \dd\theta^{1}.	
\end{align*}
By Fubini's theorem, the $\Lebesgue[1][\Curve]$-norm of \eqref{eq:IntegrationByParts2} is bounded by
\begin{align*}
	\textstyle
	\int_{\intervalcc{0,1}^{3}}
	\!
	\int_{-\ell}^{\ell}	
	\int_\Circle
	\nabs{\cD_\Curve \psi(\xi_{x}(\theta_1 , X))}
	\,
	\nabs{\Tangent_\Curve(\Geodesic_{x}(\theta_2 X))-\Tangent_\Curve(\Geodesic_{x}(\theta_3 X))}^{2}
	\, \LineElementC(x) 
	\, \tfrac{\dd X}{\nabs{X}^2}\dd\theta^{3} 
	\, \dd\theta^{2} 
	\, \dd\theta^{1}.	
\end{align*}
We employ the Hölder inequality to obtain an upper bound for this integral:
For $s=\frac{3}{2}$,s $\tilde q \ceq 2\,  p \, (p -2 + 2 \, p \, \diffoffset)^{-1} \geq 1$, we have a Sobolev embedding
$\Sobo[\hilberts-1 - \diffoffset,q]
	\hookrightarrow
	\Lebesgue[\tilde q]$ due to~\eqref{eq:Parameters}, see \autoref{lem:vectorWembedding}, thus $\nnorm{\cD_\Curve \psi}_{\Lebesgue[\tilde q][\Curve]} \leq C(\Curve) \, \nnorm{\psi}_{\Sobo[s-\diffoffset,q][\Curve]}$
with $C(\Curve)$ depending continuously on $\Curve$.
The Hölder conjugate of $\tilde q$ is $\tilde p \ceq 2 \, p \, (p + 2 - 2 \, p \, \diffoffset)^{-1}$.
Thus, we  obtain the following upper bound:
\begin{align*}
	\MoveEqLeft
	\textstyle	
	\nnorm{\cD_\Curve \psi}_{\Lebesgue[\tilde q][\Curve][\Circle]}	
	\int_{\intervalcc{0,1}^{2}}
	\!
	\int_{-\ell}^{\ell}
		\bignorm{
			\Tangent_\Curve(\Geodesic_{(\cdot)}(\theta_2 \, X))-\Tangent_\Curve(\Geodesic_{(\cdot)}(\theta_3 \, X))
		}_{\Lebesgue[2\tilde p][\Curve][\Circle]}^2
	\, \frac{\dd X}{\nabs{X}^{2}}
	\, \dd\theta_3
	\, \dd\theta_2		
	.
\end{align*}
Exploiting that $\Geodesic_{(\cdot)}(\theta_i \, X) \colon \Circle \to \Circle$ are isometries with respect to $\varrho_\Curve$, and utilizing the substitution $Y \ceq (\theta_2-\theta_3) \, X$, we may compute as follows:
\begin{align}
	\MoveEqLeft
	\textstyle
	\int_{\intervalcc{0,1}^{2}}
	\!
	\int_{-\ell}^{\ell}
		\bignorm{
			\Tangent_\Curve(\Geodesic_{(\cdot)}(\theta_2 \, X))
			-
			\Tangent_\Curve(\Geodesic_{(\cdot)}(\theta_3 \, X))
		}_{\Lebesgue[2\tilde p][\Curve][\Circle]}^2
	\, \frac{\dd X}{\nabs{X}^{2}}
	\, \dd\theta_3
	\, \dd\theta_2
	\notag
	\\
	&=
	\textstyle	
	\int_{\intervalcc{0,1}^{2}}
	\!
	\int_{-\ell}^{\ell}
		\bignorm{
			\Tangent_\Curve(\Geodesic_{(\cdot)}((\theta_2-\theta_3) \, X))-\Tangent_\Curve(\cdot)
		}_{\Lebesgue[2\tilde p][\Curve][\Circle]}^2
	\, \frac{\dd X}{\nabs{X}^{2}}
	\, \dd\theta_3
	\, \dd\theta_2	
	\notag	
	\\
	&=
	\textstyle	
	\int_{\intervalcc{0,1}^{2}}
	\!
	\int_{-\nabs{\theta_2-\theta_3} \ell}^{\nabs{\theta_2-\theta_3} \ell}
		\bignorm{
			\Tangent_\Curve(\Geodesic_{(\cdot)}(Y))-\Tangent_\Curve(\cdot)
		}_{\Lebesgue[2\tilde p][\Curve][\Circle]}^2
	\, \nabs{\theta_2-\theta_3} \, \frac{\dd Y}{\nabs{Y}^{2}}
	\, \dd\theta_3
	\, \dd\theta_2	
	\notag			
	\\
	&\leq
	\textstyle	
	\int_{-\ell}^{\ell}
		\bignorm{
			\Tangent_\Curve(\Geodesic_{(\cdot)}(Y))-\Tangent_\Curve(\cdot)
		}_{\Lebesgue[2\tilde p][\Curve][\Circle]}^2
	\, \frac{\dd Y}{\nabs{Y}^{2}}			
	\qec
	\nnorm{\Tangent_\Curve}_{B,\Curve}^2
	.
	\label{eq:SubstitutionTrick}
\end{align}
Here $\nnorm{\cdot}_{B,\Curve}$ is a natural, $\Curve$-dependent norm on the Besov space $B^{s-1}_{2 \tilde p,2}(\Circle;\AmbSpace)$.
This shows that $\varXi(\Curve) = \frac{1}{2} \, \nnorm{\zeta(\varrho_\Curve/\nabs{\triangle \Curve})}_{\Lebesgue[\infty]} \, C(\Curve) \, \nnorm{\Tangent_\Curve}_{B,\Curve}^2$.
Because of $s - 1 +\diffoffset - 1/p > s - 1 - 1/ (2 \tilde p)$
with $s=\tfrac{3}{2}$ and~\eqref{eq:Parameters}, we have a continuous Sobolev embedding $\SoboC[s - 1 +\diffoffset,p] \hookrightarrow B^{s-1}_{2 \tilde p,2}(\Circle;\AmbSpace)$ (see \cite{MR781540},~Theorem~3.3.1 or
\cite{MR1419319},~Theorem~2.4.4/1), showing that $\varXi$ is continuous.
\end{proof}

The next statement allows us to handle the principal order terms of $\RY[\Energy][\Curve]$.

\begin{lemma}\label{lem:PrincipalPartofEnergyDerivative}
Suppose $s = \frac{3}{2}$ and \eqref{eq:Parameters} with $\frac{1}{p}+\frac{1}{q}=1$.
Let $I_\Curve(x,y) \subset \Circle$ denote a shortest arc with respect to $\varrho_\Curve$ that connects $x$ and $y$.
Then the bilinear form $\cB_\Curve \colon \XCd \times \YCd\to \R$ given by
\begin{align*}
	\cB_\Curve(u,w)
	\ceq
	\textstyle
	\int_\Torus
	\Bigparen{
		\textstyle
		\tfrac{2 }{\varrho_\Curve(x,y)} \int_{I_\Curve(x,y)} \ninnerprod{\cD_\Curve u, \cD_\Curve w}\, \LineElementC
		-
		2\,
		\biginnerprod{ 
			\tfrac{\triangle u}{\varrho_\Curve} , 
			\tfrac{\triangle w}{\varrho_\Curve}
		}(x,y)
		}
	\,
	\tfrac{\varOmega_\Curve(x,y)}{\varrho_\Curve(x,y)^2}	
\end{align*}
is well-defined and bounded. More precisely, we have
$
	\nabs{\cB_\Curve(u,w)} 
	\leq 	
	\nseminorm{u}_{\Xnorm{\Curve}}\,\nseminorm{w}_{\Ynorm{\Curve}}
$.
\end{lemma}
\begin{proof}
Using the following two identities
\begin{align*}
	\textstyle
	\tfrac{2}{\varrho_\Curve} \int_{I_\Curve} \ninnerprod{\cD_\Curve u, \cD_\Curve w}\, \LineElementC
	&=
	\textstyle
	\tfrac{1}{\varrho_\Curve^2} 	\int_{I_\Curve^2}
		\Bigparen{
			\ninnerprod{\cD_\Curve u(s), \cD_\Curve w(s)}
			+
			\ninnerprod{\cD_\Curve u(t), \cD_\Curve w(t)}	
		}\, \varOmega_\Curve(s,t),
	\\
	\textstyle
	2 \, \Biginnerprod{ 
		\tfrac{\triangle u}{\varrho_\Curve} , 
		\tfrac{\triangle w}{\varrho_\Curve}
	}
	&=
	\textstyle
	\tfrac{1}{\varrho_\Curve^2} 	\int_{I_\Curve^2}
		\Bigparen{ 
			\ninnerprod{\cD_\Curve u(s), \cD_\Curve w(t)}
			+
			\ninnerprod{\cD_\Curve u(t), \cD_\Curve w(s)}			
		}\, \varOmega_\Curve(s,t),
\end{align*}
we obtain
\begin{align}
	\textstyle
	\tfrac{2}{\varrho_\Curve} \int_{I_\Curve} \ninnerprod{\cD_\Curve u, \cD_\Curve w}\, \LineElementC
	-
	2\,\Biginnerprod{ 
		\tfrac{\triangle u}{\varrho_\Curve} , 
		\tfrac{\triangle w}{\varrho_\Curve}
	}
	&=
	\textstyle
	\int_{I_\Curve^2}
	\Biginnerprod{
		\tfrac{\triangle \cD_\Curve u(s,t)}{\varrho_\Curve},
		\tfrac{\triangle \cD_\Curve w(s,t)}{\varrho_\Curve}
	}
	\, \varOmega_\Curve(s,t).
	\label{eq:StrangeBilinearExpression}
\end{align}
Now we apply the same technique as in \eqref{eq:SubstitutionTrick}:
Utilizing the notation of \autoref{lem:IntegrationbyParts} and the substitutions $y = \Geodesic_x(X)$, $s = \Geodesic_x(\theta_1 \, X)$, $t = \Geodesic_x(\theta_2 \, X)$,
and
$Y = (\theta_1-\theta_2) \, X$, we arrive at:
\begin{align*}
	\cB_\Curve(u,w)
	&=
	\textstyle	
	\int_0^1 \!\! \int_0^1 \!\!
	\int_{\Circle}
	\!
	\int_{-\nabs{\theta_1-\theta_2}\ell}^{\nabs{\theta_1-\theta_2} \ell}
		\Biginnerprod{
			\tfrac{\triangle \cD_\Curve u(x,\Geodesic_x(Y))}{\nabs{Y}},
			\tfrac{\triangle \cD_\Curve w(x,\Geodesic_x(Y))}{\nabs{Y}}
		}
	\,
	\nabs{\theta_1-\theta_2} \, \dd Y \, \LineElementC[x]
	\, \dd \theta_1 \, \dd \theta_2.
\end{align*}
Thus, we can bound $\nabs{\cB_\Curve(u,w)}$ from above by
\begin{align*}
	\textstyle	
	\int_{\Circle}
	\!
	\int_{-\ell}^{\ell}
		\bigabs{
			\tfrac{\triangle \cD_\Curve u(x,\Geodesic_x(Y))}{\nabs{Y}}
		}
		\,
		\bigabs{
			\tfrac{\triangle \cD_\Curve w(x,\Geodesic_x(Y))}{\nabs{Y}}
		}
	\, \dd Y \, \LineElementC[x]	
	\leq 
	\nseminorm{u}_{\Sobo[\strongs,\strongp][\Curve]}\,\nseminorm{w}_{\Sobo[\weaks,\weakp][\Curve]}
	=
	\nseminorm{u}_{\Xnorm{\Curve}}\,\nseminorm{w}_{\Ynorm{\Curve}}.
\end{align*}
\end{proof}

%% file: BadSet1_pdf.tex
\begingroup%
  \makeatletter%
  \providecommand\color[2][]{%
    \errmessage{(Inkscape) Color is used for the text in Inkscape, but the package 'color.sty' is not loaded}%
    \renewcommand\color[2][]{}%
  }%
  \providecommand\transparent[1]{%
    \errmessage{(Inkscape) Transparency is used (non-zero) for the text in Inkscape, but the package 'transparent.sty' is not loaded}%
    \renewcommand\transparent[1]{}%
  }%
  \providecommand\rotatebox[2]{#2}%
  \ifx\svgwidth\undefined%
    \setlength{\unitlength}{360bp}%
    \ifx\svgscale\undefined%
      \relax%
    \else%
      \setlength{\unitlength}{\unitlength * \real{\svgscale}}%
    \fi%
  \else%
    \setlength{\unitlength}{\svgwidth}%
  \fi%
  \global\let\svgwidth\undefined%
  \global\let\svgscale\undefined%
  \makeatother%
  \begin{picture}(1,1.255814)%
    \put(0,0){\includegraphics[width=\unitlength]{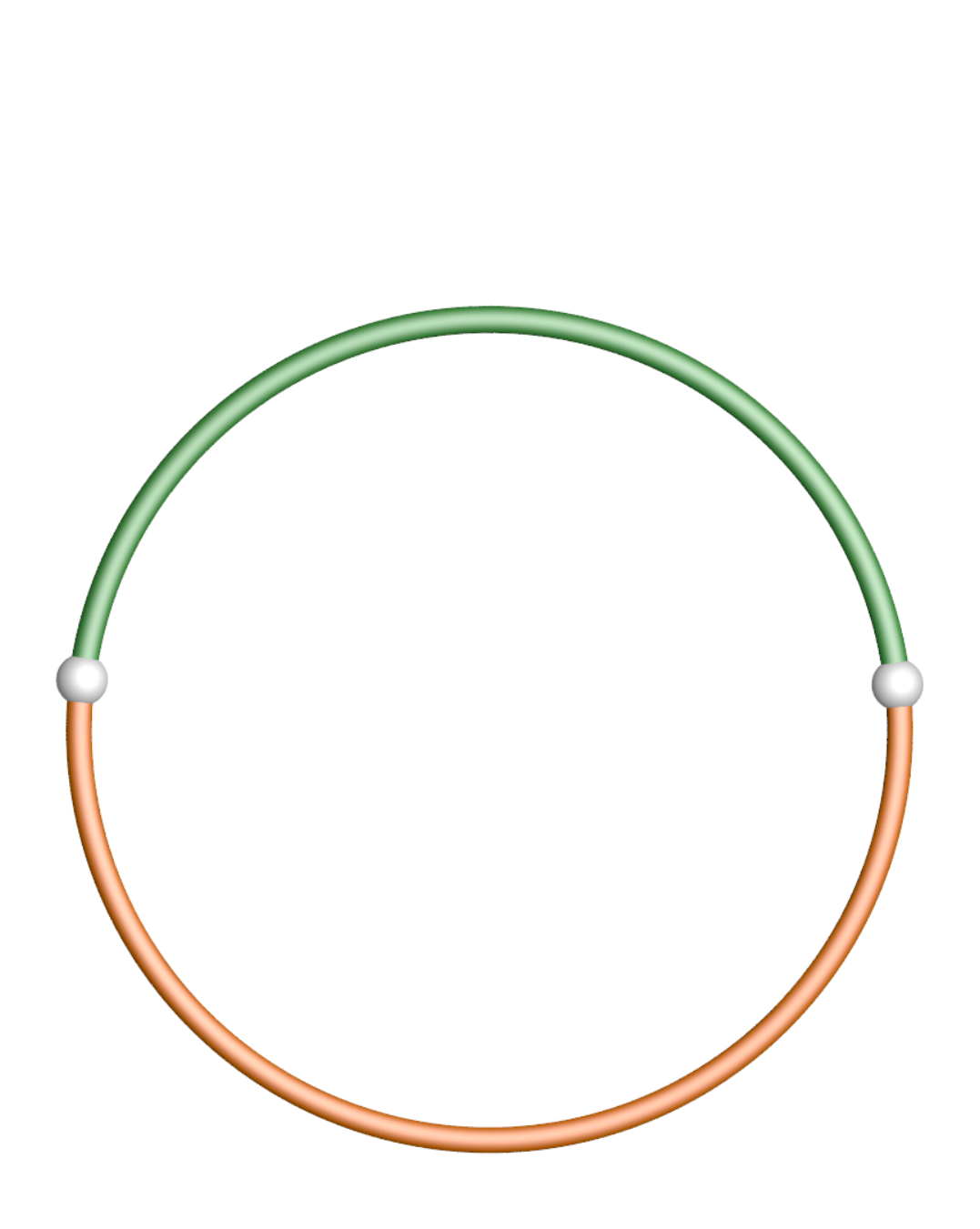}}%
    \put(0.22,0.57767442){\makebox(0,0)[cb]{\smash{$\Curve(x)$}}}%
    \put(0.8,0.50232558){\makebox(0,0)[cb]{\smash{$\Curve(y)$}}}%
    \put(0.5,1.0046512){\makebox(0,0)[cb]{\smash{$\Curve( I_\Curve(x,y))$}}}%
  \end{picture}%
\endgroup%

%% file: BadSet2_pdf.tex
\begingroup%
  \makeatletter%
  \providecommand\color[2][]{%
    \errmessage{(Inkscape) Color is used for the text in Inkscape, but the package 'color.sty' is not loaded}%
    \renewcommand\color[2][]{}%
  }%
  \providecommand\transparent[1]{%
    \errmessage{(Inkscape) Transparency is used (non-zero) for the text in Inkscape, but the package 'transparent.sty' is not loaded}%
    \renewcommand\transparent[1]{}%
  }%
  \providecommand\rotatebox[2]{#2}%
  \ifx\svgwidth\undefined%
    \setlength{\unitlength}{360bp}%
    \ifx\svgscale\undefined%
      \relax%
    \else%
      \setlength{\unitlength}{\unitlength * \real{\svgscale}}%
    \fi%
  \else%
    \setlength{\unitlength}{\svgwidth}%
  \fi%
  \global\let\svgwidth\undefined%
  \global\let\svgscale\undefined%
  \makeatother%
  \begin{picture}(1,1.255814)%
    \put(0,0){\includegraphics[width=\unitlength]{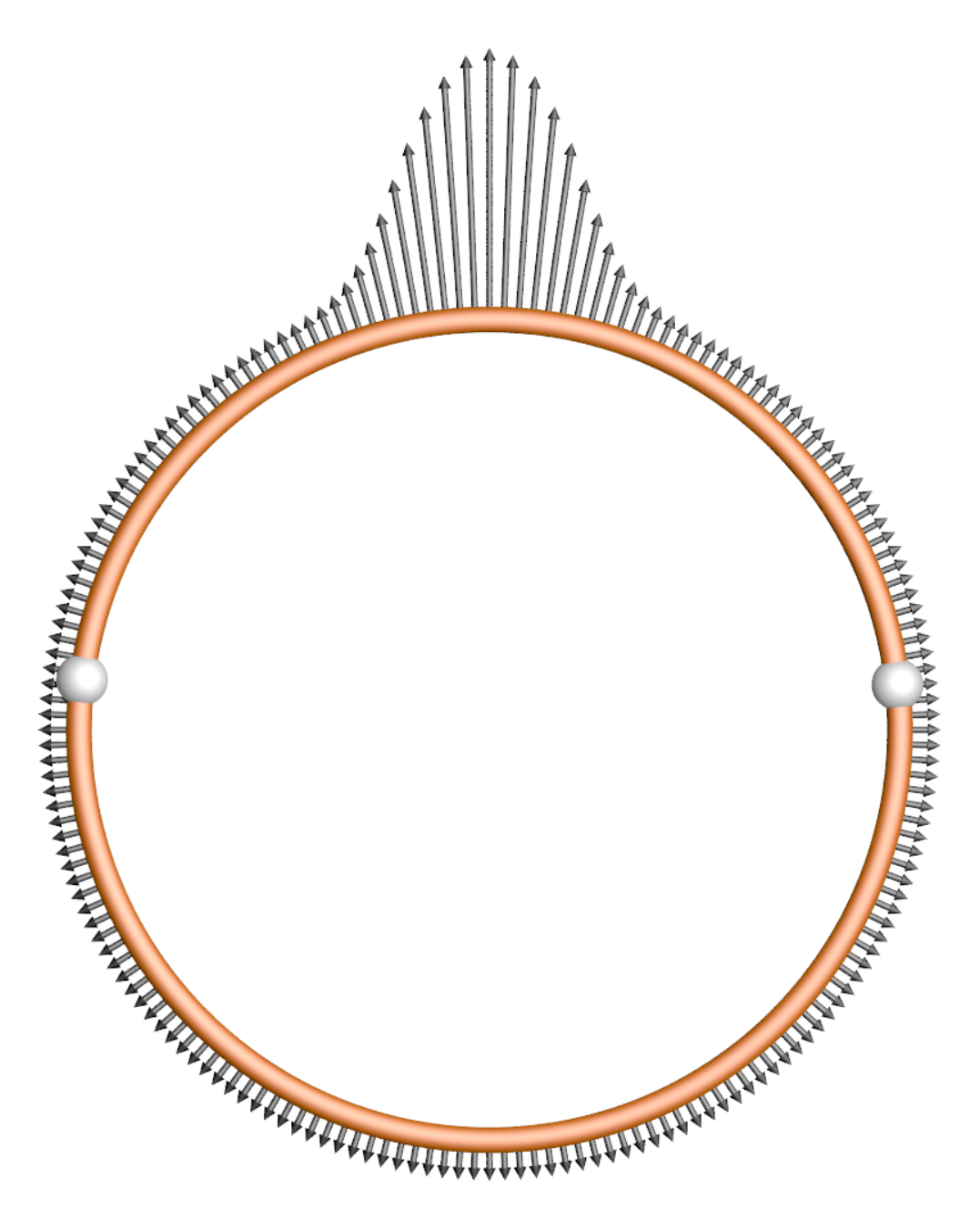}}%
    \put(0.22,0.57767442){\makebox(0,0)[cb]{\smash{$\Curve(x)$}}}%
    \put(0.8,0.50232558){\makebox(0,0)[cb]{\smash{$\Curve(y)$}}}%
    \put(0.33,1.1302326){\makebox(0,0)[cb]{\smash{$u$}}}%
  \end{picture}%
\endgroup%

%% file: BadSet3_pdf.tex
\begingroup%
  \makeatletter%
  \providecommand\color[2][]{%
    \errmessage{(Inkscape) Color is used for the text in Inkscape, but the package 'color.sty' is not loaded}%
    \renewcommand\color[2][]{}%
  }%
  \providecommand\transparent[1]{%
    \errmessage{(Inkscape) Transparency is used (non-zero) for the text in Inkscape, but the package 'transparent.sty' is not loaded}%
    \renewcommand\transparent[1]{}%
  }%
  \providecommand\rotatebox[2]{#2}%
  \ifx\svgwidth\undefined%
    \setlength{\unitlength}{360bp}%
    \ifx\svgscale\undefined%
      \relax%
    \else%
      \setlength{\unitlength}{\unitlength * \real{\svgscale}}%
    \fi%
  \else%
    \setlength{\unitlength}{\svgwidth}%
  \fi%
  \global\let\svgwidth\undefined%
  \global\let\svgscale\undefined%
  \makeatother%
  \begin{picture}(1,1.255814)%
    \put(0,0){\includegraphics[width=\unitlength]{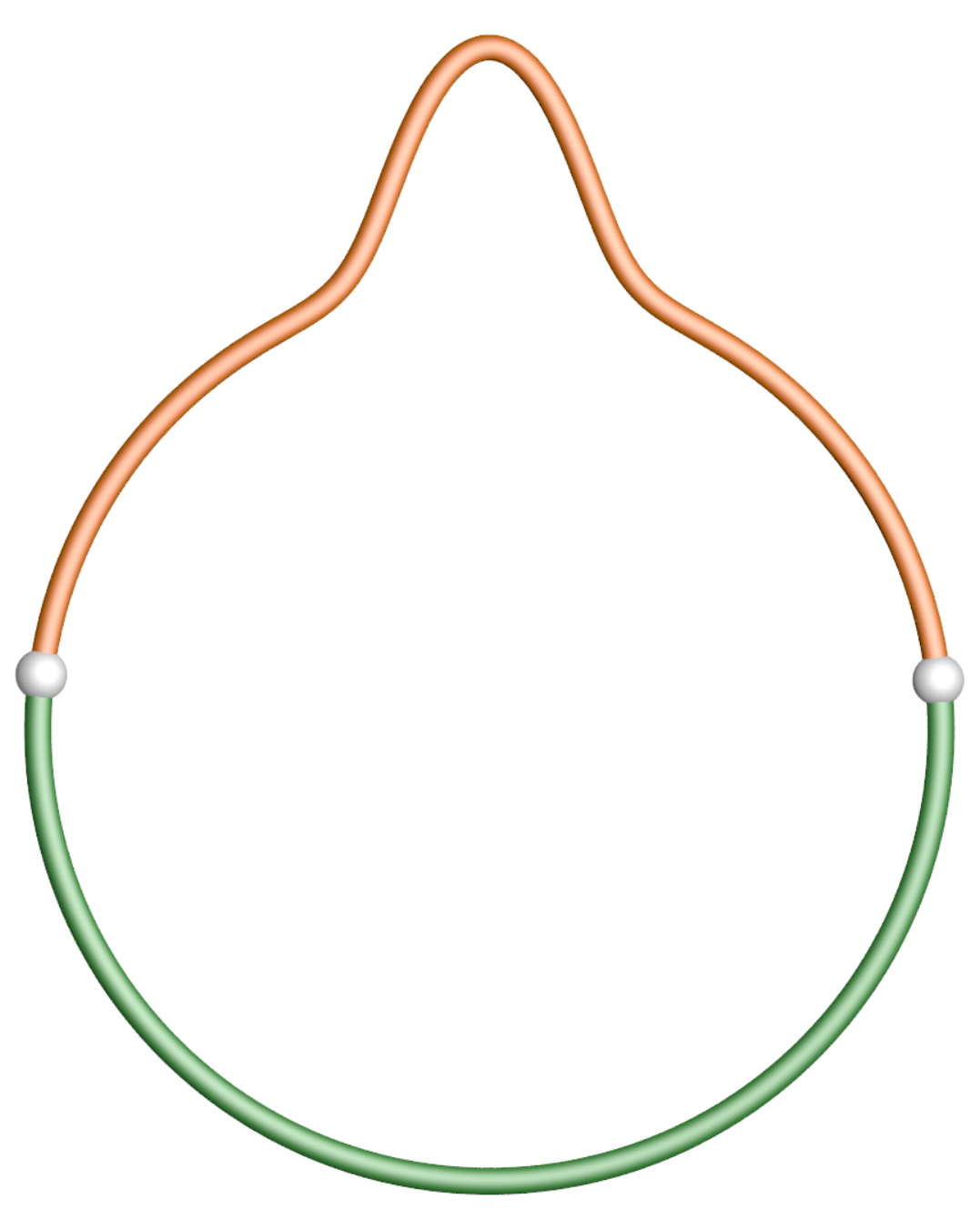}}%
    \put(0.2,0.57767442){\makebox(0,0)[cb]{\smash{$\OtherCurve(x)$}}}%
    \put(0.8,0.50232558){\makebox(0,0)[cb]{\smash{$\OtherCurve(y)$}}}%
    \put(0.5,0.18837209){\makebox(0,0)[cb]{\smash{$\OtherCurve(I_\OtherCurve(x,y))$}}}%
  \end{picture}%
\endgroup%

%% file: BadSet_pdf.tex
\begingroup%
  \makeatletter%
  \providecommand\color[2][]{%
    \errmessage{(Inkscape) Color is used for the text in Inkscape, but the package 'color.sty' is not loaded}%
    \renewcommand\color[2][]{}%
  }%
  \providecommand\transparent[1]{%
    \errmessage{(Inkscape) Transparency is used (non-zero) for the text in Inkscape, but the package 'transparent.sty' is not loaded}%
    \renewcommand\transparent[1]{}%
  }%
  \providecommand\rotatebox[2]{#2}%
  \ifx\svgwidth\undefined%
    \setlength{\unitlength}{360bp}%
    \ifx\svgscale\undefined%
      \relax%
    \else%
      \setlength{\unitlength}{\unitlength * \real{\svgscale}}%
    \fi%
  \else%
    \setlength{\unitlength}{\svgwidth}%
  \fi%
  \global\let\svgwidth\undefined%
  \global\let\svgscale\undefined%
  \makeatother%
  \begin{picture}(1,1.)%
    \put(0,0){\includegraphics[width=\unitlength]{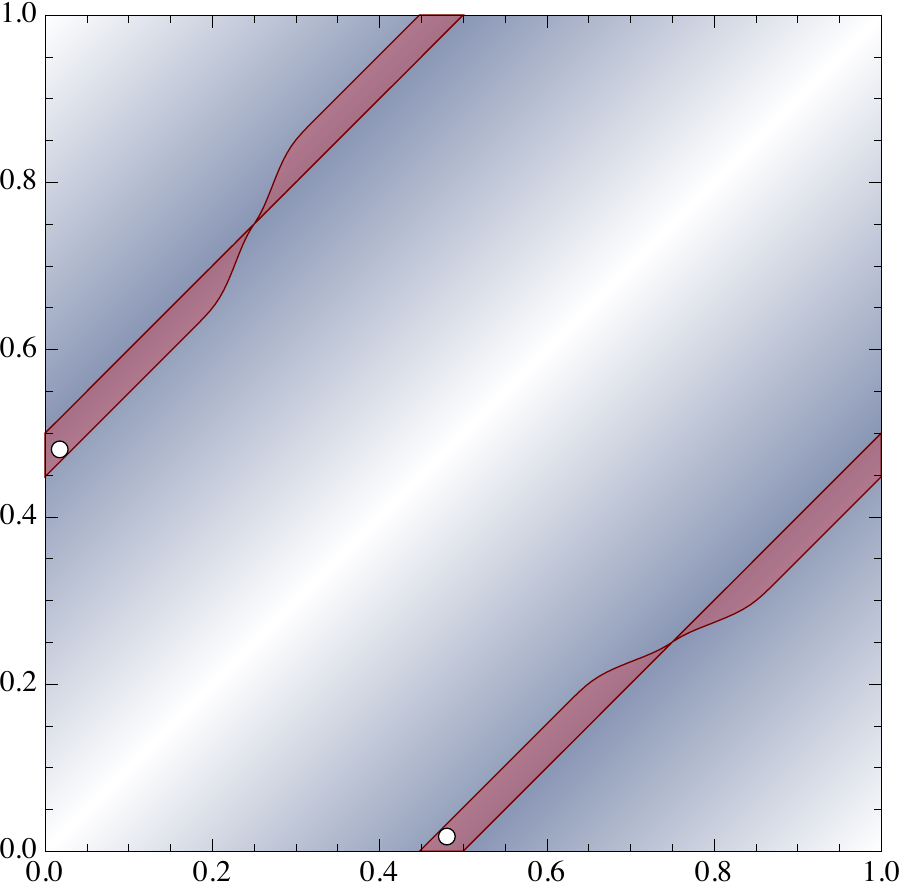}}%
    \put(0.18759259,0.41055556){\makebox(0,0)[cb]{\smash{$(y,x)$}}}%
    \put(0.40055556,0.11759259){\makebox(0,0)[cb]{\smash{$(x,y)$}}}%
  \end{picture}%
\endgroup%

%% file: Metric.tex
\section{Metrics and Riesz isomorphisms}\label{sec:Metric}

Next to the differential of the energy, the second ingredient that one requires for defining a gradient is a Riemannian metric on the configuration space $\ConfSpace$
that has been introduced in~\eqref{eq:confspace}.
Below, we pick a suitable inner product on $\HCg$ which is essentially a geometric version of the $\Sobo[\hilberts,2]$-Gagliardo inner product $G$ from the introduction (see \eqref{eq:G}).
Throughout, we represent this inner product at the point $\Curve \in \ConfSpace$ only by its Riesz isomorphism $\RI[\ConfSpace]\at_\Curve \colon \HCg \to \HCgd$.

If $\RI[\ConfSpace]\at_\Curve$ identified the tangent space $T_\Curve \ConfSpace$ of $\ConfSpace$ at $\Curve$
with the cotangent space, i.e., its dual  $T_\Curve \dual \ConfSpace$, we could define the gradient of $\Energy$ by 
\begin{align*}
	\grad(\Energy) \at_\Curve  \ceq (\RI[\ConfSpace]\at_\Curve)^{-1} \, D\Energy(\Curve).
\end{align*}
Alas, $T_\Curve \ConfSpace = \XCg$ is \emph{not} a Hilbert space for its Hilbert space completion (with respect to $\RI[\ConfSpace]$) is $\HCg \neq \XCg$, so that there cannot be any linear isomorphism $T_\Curve \ConfSpace \to T_\Curve \dual \ConfSpace = \XCgd$ induced by a positive-definite bilinear form.
However, we have already seen in \autoref{theo:DEnergy} that $D \Energy(\Curve)$ can be interpreted as an element $\RY[\Energy][\Curve]$ in the smaller space $\YCgd \subset \XCgd$, and that $\Curve \mapsto \RY[\Energy][\Curve]$ is locally Lipschitz continuous. 
Thus, \autoref{theo:Gradient} is proven as soon as we show that $\RI[\ConfSpace]\at_\Curve$ induces an isomorphism $\RJ[\ConfSpace]\at_\Curve \colon \XCg \to \YCgd$ which depends locally Lipschitz continuously on $\Curve$. This is our goal in this section.

Depending on context, we use $\ninnerprod{\cdot, \cdot}$ for the dual pairing of a Banach space with its dual, e.g., $\ninnerprod{\cdot, \cdot} \colon \HCgd \times \HCg \to \R$ and
$\ninnerprod{\cdot, \cdot} \colon \YCgd \times \YCg \to \R$,
as well as the Euclidean inner product on~$\AmbSpace$.
Note that the canonical embeddings
$\Ri[\ConfSpace] \colon \XCg \hookrightarrow \HCg$ and
$\Rj[\ConfSpace] \colon \HCg \hookrightarrow \YCg$
introduced in~\eqref{eq:canonical-embeddings} give rise to
dual maps
$\Ri[\ConfSpace]\dual \colon \HCgd \hookrightarrow \XCgd$ and
$\Rj[\ConfSpace]\dual \colon \YCgd \hookrightarrow \HCgd$.

\begin{proposition}\label{prop:MetricDefinition}
For each $\Curve \in \ConfSpace$, with $\sigma \ceq s - 1=\tfrac{1}{2}$ and the notation from \eqref{eq:Measures}, \eqref{eq:Operators}, and~\eqref{eq:Parameters},
we define
$\RI[\ConfSpace]\at_\Curve \colon \HCg \to \HCgd$
and
$\RJ[\ConfSpace]\at_\Curve \colon \XCg \to \YCgd$
as follows:
\begin{align*}
\ninnerprod{\RI[\ConfSpace]\at_\Curve \, v_1,v_2}
	&\ceq 
	\textstyle	
	\int_\Torus
		\ninnerprod{
			\diff{\sigma}{\Curve} \cD_\Curve v_1
		, 
			\diff{\sigma}{\Curve} \cD_\Curve v_2
		}
	\, \singularmeasure_\Curve
	\\ &\qquad
	\textstyle	
	+
	\int_\Torus
		\ninnerprod{
			\diff{\sigma}{\Curve} v_1
		, 
			\diff{\sigma}{\Curve} v_2
		}
		\,
		E(\Curve)		
		\, \singularmeasure_\Curve
	+
	{\textstyle
		\ninnerprod{ \int_\Circle v_1 \, \LineElement_\Curve,\int_\Circle v_2 \, \LineElement_\Curve}
	},
	\\
	\ninnerprod{\RJ[\ConfSpace]\at_\Curve \, u,w}
	&\ceq 
	\textstyle	
	\int_\Torus
		\ninnerprod{
			\diff{\strongss}{\Curve} \cD_\Curve u
		, 
			\diff{\weakss}{\Curve} \cD_\Curve w
		}
		\, \singularmeasure_\Curve
	\\ &\qquad
	\textstyle	
	+
	\int_\Torus		
		\ninnerprod{
			\diff{\strongss}{\Curve} u
		, 
			\diff{\weakss}{\Curve} w
		}
		\,
		E(\Curve)			
		\, \singularmeasure_\Curve
	+
	{\textstyle
		\ninnerprod{ \int_\Circle u \, \LineElement_\Curve,\int_\Circle w \, \LineElement_\Curve}
	}
\end{align*}
for $v_1$, $v_2 \in \HCg$, $u \in \XCg$, and $w \in \YCg$.
These operators are well-defined, continuously invertible, and they make the following diagram commutative:
\begin{equation}
\begin{tikzcd}[]
	\XCg
		\ar[d,hook, two heads,"{\Ri[\ConfSpace]}"]
		\ar[rr, "{\RJ[\ConfSpace]\at_\Curve}"]
	&&\YCgd
		\ar[d,hook, two heads,"{\Rj[\ConfSpace]\dual}"]	
	\\
	\HCg
		\ar[rr, "{\RI[\ConfSpace]\at_\Curve}"]	
	&&\HCgd
	\nospaceperiod
\end{tikzcd}	
	\label{eq:j'J=Ii}
\end{equation}
Moreover, the mappings $\RI[\ConfSpace] \colon \ConfSpace \to L(\HCg; 	\HCgd)$,
$\Curve\mapsto\RI[\ConfSpace]\at_\Curve$,
and
$\RJ[\ConfSpace] \colon \ConfSpace \to L(\XCg; \YCgd)$,
$\Curve \mapsto \RJ[\ConfSpace]\at_\Curve$,
are of class $\Holder[1]$ and hence locally Lipschitz continuous.
\end{proposition}

Here and in the following, a doubly headed arrow in the diagram~\eqref{eq:j'J=Ii} above indicates a linear operator with dense image.

\begin{proof}
\emph{Well-definedness:}
It follows  from the bi-Lipschitz continuity of $\Curve$ 
and from Hölder's inequality that the integrals 
\begin{align*}
		\textstyle
	\int_\Torus
		\ninnerprod{
			\diff{\sigma}{\Curve} \cD_\Curve v_1
			, 
			\diff{\sigma}{\Curve} \cD_\Curve v_2
		}	
		\, \singularmeasure_\Curve
	\qand
	\int_\Torus
		\ninnerprod{
			\diff{\strongss}{\Curve} \cD_\Curve u
			, 
			\diff{\weakss}{\Curve} \cD_\Curve w
		}	
		\, \singularmeasure_\Curve
\end{align*}
are well-defined and finite.
Indeed, we have by Hölder's inequality that
\begin{align*}
	\nabs{
	\textstyle
		\int_\Torus
		\ninnerprod{
			\diff{\strongss}{\Curve} \cD_\Curve u
			, 
			\diff{\weakss}{\Curve} \cD_\Curve w
		}	
		\, \singularmeasure_\Curve
	}
	&\leq
	\nnorm{\ninnerprod{
			\diff{\strongss}{\Curve} \cD_\Curve u
			, 
			\diff{\weakss}{\Curve} \cD_\Curve w
		}}_{\Lebesgue[1][\singularmeasure_\Curve]}
\\
	&\leq 
	\nnorm{\diff{\strongss}{\Curve} \cD_\Curve u}_{L_{\singularmeasure_\Curve}^{p}}
  	\nnorm{\diff{\weakss}{\Curve} \cD_\Curve w}_{L_{\singularmeasure_\Curve}^{q}}
	= 
	\seminorm{u}_{\Xnorm{\Curve}}
	\seminorm{w}_{\Ynorm{\Curve}}
\end{align*}
and, analogously,
$
	\nabs{
	\textstyle
		\int_\Torus
		\ninnerprod{
			\diff{\sigma}{\Curve} \cD_\Curve v_1
			, 
			\diff{\sigma}{\Curve} \cD_\Curve v_2
		}	
		\, \singularmeasure_\Curve
	}
	\leq
	\seminorm{v_1}_{\Hnorm{\Curve}}
	\seminorm{v_2}_{\Hnorm{\Curve}}$.
Moreover, the existence of the integrals
\begin{align*}
		\textstyle
		\int_\Torus
		\ninnerprod{
			\diff{\sigma}{\Curve} v_1
			, 
			\diff{\sigma}{\Curve} v_2
		}
		\,
		E(\Curve)
		\, \singularmeasure_\Curve
\qand
	\int_\Torus
		\ninnerprod{
			\diff{\strongss}{\Curve} u
			, 
			\diff{\weakss}{\Curve} w
		}
		\,
		E(\Curve)
		\, \singularmeasure_\Curve
\end{align*}
follows from \autoref{lem:IntegrationbyParts}
with $k=1$, $\alpha = 2$, $\beta=0$,
and $L_1 = K = \triangle/\varrho_\Curve$.
(Here we use $\sigma=\tfrac12$.)
The commutativity of \eqref{eq:j'J=Ii} follows from the pointwise identity
\begin{align*}
	\ninnerprod{\diff{\strongss}{\Curve} \varphi, \diff{\weakss}{\Curve} \psi}(x,y)
	=
	\tfrac{\ninnerprod{\varphi(x) - \varphi(y),\psi(x) - \psi(y)}}{\nabs{\Curve(x)-\Curve(y)}^{2 \sigma}}
	=
	\ninnerprod{\diff{\sigma}{\Curve} \varphi, \diff{\sigma}{\Curve} \psi}(x,y)	
\end{align*}
which holds for arbitrary functions $\varphi$, $\psi \colon \Circle \to \AmbSpace$.

\emph{Invertibility:} 
We show this only for $\RJ[\ConfSpace]\at_\Curve$, as the argument for $\RI[\ConfSpace]\at_\Curve$ is analogous.
First we observe that $\RJ[\ConfSpace]\at_\Curve$ is injective.
Indeed, let $u \in \ker(\RJ[\ConfSpace]\at_\Curve )$. Then
\begin{align*}
	0
	=
	\textstyle
	\ninnerprod{\RJ[\ConfSpace]\at_\Curve \, u ,  \Rj[\ConfSpace] \, \Ri[\ConfSpace] \, u}
	\geq 
	\nnorm{\diff{\sigma}{\Curve} \cD_\Curve u}_{\Lebesgue[2][\singularmeasure_\Curve]}^{2}
	+ 
	\nabs{\int_\Circle u \, \LineElement_\Curve}^{2}
\end{align*}
implies that $\cD_{\Curve} u$ must be constant
and that the mean value of $u$ vanishes.
But the first condition can only hold if $u$ is already constant
and the second one forces this constant to be zero.
So it suffices to show that $\RJ[\ConfSpace]\at_\Curve$ is a Fredholm operator of index $0$.
To this end, we define the operator $A_\Curve \colon \XCg \to \YCgd$ by
\begin{align}
	\textstyle
	\ninnerprod{A_\Curve \, u, w}
	\ceq 
	\int_\Torus \biginnerprod{
		\Diff{\strongss}{\Curve} \cD_\Curve u,
		\Diff{\weakss}{\Curve} \cD_\Curve u
	} \, \dd \mu_\Curve
	,
	\quad
	\text{where}
	\quad
	\Diff{\alpha}{\Curve} \varphi \ceq \frac{\triangle \varphi}{\varrho_\Curve^\alpha}.
	\label{eq:DefinitionofACurve}
\end{align}
Now we observe that
$\ninnerprod{A_\Curve \, u, w} =
	\int_\Torus \biginnerprod{
		\Diff{\strongss}{\Curve} \cD_\Curve u,
		\Diff{\weakss}{\Curve} \cD_\Curve u
	} \, \dd \mu_\Curve$
and that
\begin{align*}
	\textstyle
	\ninnerprod{\bigparen{\RJ[\ConfSpace]\at_\Curve - A_\Curve} \, u ,w}
	=
	\sum_{i=1}^2
	\sum_{j=1}^2
	(-1)^{i+j}
	\int_\Torus \ninnerprod{\cD_\Curve u \circ \pi_i,\cD_\Curve w \circ \pi_j}
	\,
	E(\Curve)
	\, \varOmega_\Curve
	+ \lot
\end{align*}
From \autoref{lem:IntegrationbyParts} we may conclude that $A_\Curve$ is a compact perturbation of $\RJ[\ConfSpace]\at_\Curve$.
So it suffices to show that $A_\Curve$ is a Fredholm operator of index $0$.
As this is a bit more involved, we defer this to \autoref{lem:FredholmProperty} below.

\emph{Fréchet differentiability:}
This can be shown by utilizing basically the same technique as in the proof of the Fréchet differentiability of the energy: first order Taylor expansion  of the integrand around the point $\Curve$ and bounding the integral of the second order remainder term (see Claim 3 in \autoref{theo:DEnergy}). In fact, the analysis here is bit easier for the geodesic distance $\varrho_\Curve$ does not appear in the definitions of $\RI[\ConfSpace]$ and $\RJ[\ConfSpace]$. For the sake of brevity, we omit the details.
\end{proof}

\subsection*{Details}

The remainder of this section is devoted to proving the following lemma. The employed techniques are fairly standard in the area of pseudo-differential operators;
only the fact that the differential operators involved here are nonlocal introduces a couple of further technicalities.
Throughout, we will suppose that $\Curve \in \ConfSpace$. Moreover, we will denote by $K \colon \cX \to \cZ$ a generic compact operator that---pretty much like the ever-expanding ``constant'' $C$---may change from line to line.

\begin{lemma}\label{lem:FredholmProperty}
The operator $A_\Curve \colon \XCg \to \YCgd$ defined in  \eqref{eq:DefinitionofACurve} is a Fredholm operator of index $0$.
\end{lemma}

\begin{proof}
As the Fredholm property and the index are invariant under compact perturbations, it suffices to show that  $A_\Curve + \mass_\Curve$ is continuously invertible,
where we define the compact operator $\mass_\Curve \colon \cX \to \cY\dual$ by
$
	\ninnerprod{ \mass_\Curve \,u , w}_{\cY\dual,\cY}
	\ceq
	\textstyle
	\int_\Circle \ninnerprod{u , w} \, \LineElement[\Curve].
$
The operator $A_\Curve + \mass_\Curve$ is injective because of
$
	\ninnerprod{(A_\Curve + \mass_\Curve) \,u , \Rj[\ConfSpace] \, \Ri[\ConfSpace] \, u}_{\cY\dual,\cY}
	\geq \nnorm{u}_{\Lebesgue[2][\Curve]}^2
$.
Since the operator $\Rj[\ConfSpace] \, \Ri[\ConfSpace]  \colon \cX \hookrightarrow \cY$ is injective and has dense image, this also implies that $A_\Curve + \mass_\Curve$ has dense image.
By virtue of the Schauder lemma (see \autoref{lem:SchauderLemma}), it suffices to establish an \emph{elliptic estimate} of the form
$	\nnorm{u}_{\XC}
	\leq 
	C\, \bigparen{
		\nnorm{(A_\Curve + M_\Curve)\, u}_{\YC\dual}
		+
		\nnorm{K \, u}_{\cZ} 
		}$
for all $u \in \cX$ with suitable norms $\nnorm{\cdot}_{\XC}$ and $\nnorm{\cdot}_{\YC}$ that will be defined soon.
Indeed, we will show in \autoref{lem:GlobalEllipticEstimate} that
\begin{align}
	\nnorm{u}_{\XC}
	\leq 
	C\, \bigparen{
		\nnorm{A_\Curve \, u}_{\YCd}
		+
		\nnorm{K \, u}_{\cZ}
	}
	\quad
	\text{for all $u \in \XCg$.}
	\label{eq:GlobalEllipticEstimate}
\end{align}
Since $M_\Curve$ is compact, this implies the previous inequality (for different $C$, $K$ and $\cZ$).
\end{proof}

In order to prove \eqref{eq:GlobalEllipticEstimate}, it is convenient to introduce some additional notation.
The operator $A_\Curve$ is defined in terms of $\Diff{\strongss}{\Curve} = \frac{\triangle}{\varrho_\Curve^\strongss}$ and $\Diff{\weakss}{\Curve} = \frac{\triangle}{\varrho_\Curve^\weakss}$, see \eqref{eq:DefinitionofACurve}, so it is natural to consider the following local semi-norms:
For an open set $U \subset \Circle$, $0<\alpha<1$, and $1 \leq r < \infty$ and a measurable $v \colon \Circle \to \AmbSpace$, we define
\begin{align*}
	\nseminorm{v}_{\Sobo[1+\alpha,r][\varrho,\Curve][U]}
	\ceq
	\paren{
		\textstyle
		\int_{U \times U}
			\nabs{ \Diff{\alpha}{\Curve} \cD_\Curve v}^r
		\,\dd \mu_\Curve
	}^{1/r}
	\qand
	\nnorm{v}_{\Sobo[1+\alpha,r][\varrho,\Curve][U]}
	\ceq
	\nnorm{v}_{\Lebesgue[r][\Curve][U]}
	+
	\nseminorm{v}_{\Sobo[1+\alpha,r][\varrho,\Curve][U]}.
\end{align*}
For convenience, we fix $\Curve$ and define
\begin{align*}
	\nnorm{\cdot}_{\XC} 
	\ceq 
	\nnorm{\cdot}_{\Sobo[\strongs,\strongp][\varrho,\Curve][\Circle]}
	\qand
	\nnorm{\cdot}_{\YC} 
	\ceq 
	\nnorm{\cdot}_{\Sobo[\weaks,\weakp][\varrho,\Curve][\Circle]}.
\end{align*}
We point out that $\Curve \in \ConfSpace$  has always finite energy, thus $\Curve$ is bi-Lipschitz continuous, see~\cite[Thm.~2.3]{MR1195506}); 
hence the norms 
$\nnorm{\cdot}_{\XC}$, $\nnorm{\cdot}_{\YC}$
are equivalent to
$\nnorm{\cdot}_{\Xnorm{\Curve}}$, $\nnorm{\cdot}_{\Ynorm{\Curve}}$, respectively, cf.~\eqref{eq:XHYNorms}.
For a measurable function $\tilde v \colon \R \to \AmbSpace$ and an open set $\tilde U \subset \AmbSpace$, we define
\begin{align*}
	\Diff{\alpha}{} \tilde v(x,y) \ceq \frac{\tilde v(x) - \tilde  v(y)}{\nabs{x-y}^\alpha}
	\qand
	\cD \tilde v (x) = \tilde v'(x),	
\end{align*}
as well as the following semi-norms:
\begin{align*}
		\nseminorm{\tilde v}_{\Sobo[1+\alpha,r][][\tilde U]}
	\ceq
	\paren{
		\textstyle
		\int_{\tilde U \times \tilde U}
			\nabs{ \Diff{\alpha}{} \cD \tilde v}^r
		\, \frac{\dd x \, \dd y}{\nabs{x-y}}
	}^{1/r}
	\qand
	\nnorm{\tilde v}_{\Sobo[1+\alpha,r][][\tilde U]}
	\ceq
	\nnorm{\tilde v}_{\Lebesgue[r][\Curve][\tilde U]}
	+
	\nseminorm{\tilde v}_{\Sobo[1+\alpha,r][][\tilde U]}.
\end{align*}
Moreover, we use following abbreviations for the global Sobolev spaces and their norms:
\begin{align*}
	\tilde \cX \ceq \Sobo[\strongs,\strongp][][\R][\AmbSpace],
	\;\;\;
	\nnorm{\cdot}_{\tilde \cX} 
	\ceq
	\nnorm{\cdot}_{\Sobo[\strongs,\strongp][][\R]},
	\;\;\;
	\tilde \cY \ceq \Sobo[\weaks,\weakp][][\R][\AmbSpace],
	\;\;\;
	\nnorm{\cdot}_{\tilde \cY} 
	\ceq
	\nnorm{\cdot}_{\Sobo[\weaks,\weakp][][\R]}.	
\end{align*}
Via localization techniques, we will compare $A_\Curve$ to
$L \colon \tilde \cX \to \tilde \cY\dual$ given by
\begin{align}
	\ninnerprod{ L \, \tilde u, \tilde v}
	\ceq
	\textstyle
	\int_{\R \times \R}
		\ninnerprod{
			\Diff{\strongs}{} \cD \tilde u,
			\Diff{\weaks}{} \cD \tilde w,			
		}
	\, \dd \mu
	\quad
	\text{with}
	\quad
	\dd \mu (x,y) \ceq	\frac{\dd x \, \dd y}{\nabs{x-y}}.
	\label{eq:DefinitionofL}		
\end{align}
Compare this to the weak formulation of the fractional Laplacian
\begin{align*}
	(-\Delta)^{\sigma}  \colon \Sobo[\strongss,\strongp][][\R][\AmbSpace] \to \Sobo[-\strongss,\strongp][][\R][\AmbSpace] \cong (\Sobo[\weakss,\weakp][][\R][\AmbSpace])\dual
\end{align*}
which is given by
\begin{align*}
	\textstyle
	\ninnerprod{(-\Delta)^{\sigma}  \varphi, \psi}
	=
	C_\sigma
	\int_{\R \times \R}
		\ninnerprod{
			\Diff{\strongs}{} \varphi,
			\Diff{\weaks}{} \psi,			
		}
	\, \dd \mu	
	=
	C_{\sigma}
	\int_{\R}\int_{\R}
		\Biginnerprod{ 
			\frac{\varphi(x) - \varphi(y)}{\nabs{x-y}^{\sigma}}, 
			\frac{\psi(x) - \psi(y)}{\nabs{x-y}^{\sigma}}
	}
	\, \frac{\dd x\, \dd y}{\nabs{x-y}}
\end{align*}
for some $C_{\sigma} > 0$
(see e.g., \cite[Theorem 1.1]{Kwasnicki:2017:TED}).
So up to a constant, we have $L = \cD\dual (-\Delta)^\sigma \cD$,
hence $L$ is a pseudo-differential operator of order $2 \, \sigma+2=2s=3$ and its principal symbol
$P(L) \colon T\dual\, \R \cong \R \times \R \to \End(\AmbSpace)$ is (up to a constant) given by $P(L)(x,\xi) = \nabs{\xi}^{2s}\, \id_{\R^{m}} = P((\id-\Delta)^{s})(x,\xi)$.
By \cite[2.3.8]{MR781540}, the operator $L$ is continuously invertible.%
\footnote{In the source it is shown that the operator $(\id-\Delta)^{\sigma} \colon B^{t}_{p,q}\to B^{t-2\sigma}_{p,q}$
is continuously invertible
for general Besov spaces $B^{t}_{p,q}$.
Note that $W^{t,p}=B^{t}_{p,p}$ for $t\in\R_{+}\setminus\Z$ and $p\in[1,\infty)$.
Moreover, $(-\Delta)^{\sigma}$ a compact perturbation  of $(\id-\Delta)^{\sigma}$, so $(-\Delta)^{\sigma}$ is a Fredholm operator of index zero; being also positive-definite, it must be continuously invertible.}

\bigskip

The following two lemmas will help us in localizing Sobolev norms:

\begin{lemma}[Norm localization]\label{lem:LocalizationNormsCircle}
Let $0<\alpha <1$, $1 \leq r < \infty$ and let $U \subset\subset V \subset\subset \Circle$ be relatively compact, open sets.
Then there is a $C = C(\Curve,\alpha,r,U,V)>0$ such that
\begin{align*}
	\nnorm{v}_{ \Sobo[1+\alpha,r][\varrho,\Curve][V]}
	\leq
	\nnorm{v}_{ \Sobo[1+\alpha,r][\varrho,\Curve][\Circle]}
	\leq
	C \, \nnorm{v}_{ \Sobo[1+\alpha,r][\varrho,\Curve][V]}	
	\;\;
	\text{for all $v \in \Sobo[1+\alpha,r][][\Circle][\AmbSpace]$ with $\supp(v) \subset V$.	}
\end{align*}
\end{lemma}

\begin{lemma}[Norm localization]\label{lem:LocalizationNormsReals}
Let $0<\alpha <1$, $1 \leq r < \infty$ and let $\tilde V \subset\subset \tilde U \subset\subset \R$ be relatively compact, open sets.
Then there is a $C = C(\alpha,r,\tilde U, \tilde V)>0$ such that
\begin{align*}
	\nnorm{\tilde v}_{ \Sobo[1+\alpha,r][][\tilde U]}
	\leq
	\nnorm{\tilde v}_{ \Sobo[1+\alpha,r][][\R]}
	\leq
	C \, \nnorm{\tilde v}_{ \Sobo[1+\alpha,r][][\tilde U]}	
	\;\;
	\text{for all $\tilde v \in \Sobo[1+\alpha,r][][\R][\AmbSpace]$ with $\supp(\tilde v) \subset \tilde V$.	}
\end{align*}
\end{lemma}

Their proofs are quite similar, so we show only the proof of the latter for it contains an additional difficulty.

\begin{proof} (of \autoref{lem:LocalizationNormsReals})
Since $\supp(\tilde v) \subset \tilde V \subset \tilde U$, we obviously have
\begin{align*}
	\nnorm{\tilde v}_{\Lebesgue[r][][\tilde U]}
	=
	\Bigparen{	
		\textstyle \int_{\tilde U} \nabs{ \tilde v(x)}^r \, \dd x
	}^{1/r}
	&=
	\Bigparen{	
		\textstyle \int_\R \nabs{ \tilde v(x)}^r \, \dd x
	}^{1/r}	
	= \nnorm{\tilde v}_{\Lebesgue[r][][\R]}
	\quad \text{and}
	\\
	\nseminorm{\tilde v}_{\Sobo[1+\alpha,r][][\tilde U]}
	=
	\Bigparen{
		\textstyle \int_{\tilde U \times \tilde U} \nabs{ \Diff{\alpha}{} \cD \, \tilde v }^r \, \dd \mu
	}^{1/r}
	&\leq
	\Bigparen{	
		\textstyle \int_{\R \times \R} \nabs{ \Diff{\alpha}{} \cD \, \tilde v }^r \, \dd \mu
	}^{1/r}	
	=
	\nseminorm{\tilde v}_{\Sobo[1+\alpha,r][][\R]}
	.
\end{align*}
The following identity is caused by the nonlocality of $\Diff{\alpha}{}$ and it is easily overlooked:
\begin{align*}
	\nseminorm{\tilde v}_{\Sobo[1+\alpha,r][][\R]}^r
	&=
	\textstyle
	\int_{\R \times \R}
	\nabs{ 
		\Diff{\alpha}{} \cD \,  \tilde v
	}^r
	\, \dd \mu
	=
	\textstyle
	\int_{\tilde U \times \tilde U}
	\nabs{ 
		\Diff{\alpha}{} \cD \,  \tilde v
	}^r
	\, \dd \mu
	+
	2 \, 
	\int_{\tilde V} 
	\nabs{ 
		\cD \, \tilde v(y)
	}^r
	\,
	\Bigparen{	
		\int_{\R \setminus  \tilde U}
		\, \frac	{\dd x}{\nabs{x-y}^{1+\alpha \, r}}
	}
	\, \dd y.
\end{align*}
Notice that we used here that $\supp(\cD \, \tilde v) \subset \tilde V$. 
We assumed that $\tilde V \subset\subset \tilde U$ is relatively compact, so there is an $R>0$ such that $\nabs{x-y} \geq R$ for all $x \in \R \setminus \tilde U$ and $y \in  \tilde V$. 
Because of $1+\alpha \, r>1$, we have
$
	\textstyle
	\sup_{y \in \tilde V}
	\int_{\R \setminus  \tilde U} \, \frac	{\dd x}{\nabs{x-y}^{1+\alpha \, r}}	
	\leq
	\int_R^\infty \frac{\dd t}{t^{1+\alpha \, r}}
	< \infty
$,
leading us to
\begin{align*}
	\nseminorm{\tilde v}_{\Sobo[1+\alpha,r][][\R]}^r
	\leq
	\textstyle
	\nseminorm{\tilde v}_{\Sobo[1+\alpha,r][][\tilde U]}^r
	+
	2 \, C \,
	\nnorm{\cD \, \tilde v}_{\Lebesgue[r][][\tilde V]}^r
	\leq 
	C\, \nnorm{\tilde v}_{\Sobo[1+\alpha,r][][\tilde U]}^r,
\end{align*}
which concludes the proof.
\end{proof}
	
Now we can start to show \eqref{eq:GlobalEllipticEstimate}, at least for functions $u$ with small support.

\begin{lemma}[Local elliptic estimate]\label{lem:LocalEllipticEstimate}
For each point $a \in \Circle$, there are open neighborhoods $W \subset\subset U \subset \Circle$
and a compact operator $K \colon \XC \to \cZ$ into some Banach space $\cZ$
such that
\begin{align}
	\nnorm{u}_{\XC}
	\leq 
	C \, \bigparen{ 
		\nnorm{A_\Curve \, u}_{\YC\dual}
		+
		\nnorm{K \, u}_{\cZ}		
	}
	\quad
	\text{holds for each $u \in \XCg$ with $\supp(u) \subset W$.}
	\label{eq:LocalSchauder}
\end{align}
\end{lemma}
\begin{proof}
We start by picking an isometric coordinate system around the point $a\in \Circle$.
With $L \ceq \int_\Circle \LineElementC$ denoting the curve length, we choose
$U \ceq B_{L/4}(a)$,
$V \ceq B_{L/8}(a)$, and
$W \ceq B_{L/16}(a)$,
where these balls are meant with respect to the intrinsic distance $\varrho_\Curve$.
We point out that $U$ is geodesically convex, i.e., each shortest arc between two points in $U$ is contained in $U$.
Likewise we define the following open balls  
$\tilde U \ceq B_{L/4}(0)$,
$\tilde V \ceq B_{L/8}(0)$,
and
$\tilde W \ceq B_{L/16}(0)$ 
in $\R$,
this time with respect to the standard distance on $\R$.
Now we define the isometric chart
$
	\isometry \colon U \to \tilde U
$ by
$\isometry(x) = \pm \varrho_\Curve(a,x)$, where the sign depends on whether the shortest curve from $a$ to $x$ is oriented positively ($+$) or negatively ($-$) with respect to the standard orientation of~$\Circle$.
Let  $u \in \XCg$ be a function with $\supp(u) \subset W$.
Then we can find a unique function $\tilde u \in \tilde \cX$ with $\supp(\tilde u) \subset \tilde W$ and $u = \tilde u \circ f$.\footnote{{We point out that $\isometry$ is of class $\Sobo[\strongs,\strongp] \subset \Holder[1]$. So indeed, both pullback and pushforward along it preserve the regularity of functions in the classes $\Sobo[\strongs,\strongp]$ and $\Sobo[\weaks,\weakp]$, a consequence of the chain rule \autoref{lem:ChainRule}.}}
Because $f$ is an isometry, we have $\nnorm{u}_{\Sobo[\strongs,\strongp][\varrho,\Curve][U]} = \nnorm{\tilde u}_{\Sobo[\strongs,\strongp][][\tilde U]}$.
From \autoref{lem:LocalizationNormsCircle} and from the continuous invertibility of $L \colon \tilde \cX \to \tilde \cY$, we deduce that there is a $C \geq 0$ (depending only on $\Curve$, $\strongs$, $\strongp$, $U$, and $W$) such that
\begin{align}
	\nnorm{u}_{\XC}
	=
	\nnorm{u}_{\Sobo[\strongs,\strongp][\varrho,\Curve][\Circle]} 
	\leq
	C \, \nnorm{u}_{\Sobo[\strongs,\strongp][\varrho,\Curve][U]} 
	= 
	C \, \nnorm{\tilde u}_{\Sobo[\strongs,\strongp][][\tilde U]}	
	\leq
	C \,\nnorm{\tilde u}_{\tilde \cX}
	\leq 
	C \, \nnorm{L^{-1}} \, \nnorm{L \, \tilde u}_{\tilde \cY\dual}.
	\label{eq:EllipticEstimateL}
\end{align}
Our next goal is to control $\nnorm{L \, \tilde u}_{\tilde \cY\dual}$ by 
$\nnorm{A_\Curve \, u}_{\YC}$ modulo a compact operator.
To this end, we choose a bump function $\eta \in \Sobo[\strongs,\strongp][][\Circle][\R]$ with values in $\intervalcc{0,1}$ that satisfies
$\eta(x) = 1$ for all $x \in \bar W$ and $\supp(\eta) \subset V$.
We denote by $\tilde \eta  \in \Sobo[\strongs,\strongp][][\R][\R]$
the unique function $\tilde \eta \circ \isometry = \eta$ and $\supp(\tilde \eta) \subset \tilde V$.
These functions induce the following multiplication operators:
\begin{align*}
	\multop \colon \XCg &\to \XCg, 
	&  \multop \, u &\ceq \eta \, u;
	&&
	&
	\tilde \multop \colon \tilde \cX &\to \tilde \cX, 
	&
	\tilde \multop \, \varphi &\ceq \tilde \eta \, \varphi;
	\\
	\multop \colon \YCg &\to \YCg, 
	&  \multop \, w &\ceq \eta \, w;
	&&
	&
	\tilde \multop \colon \tilde \cY &\to \tilde \cY,
	& 
	\tilde \multop \, \psi &\ceq \tilde \eta \, \psi.	
\end{align*}
Because of 
$\multop \, u = u$ and
$\tilde \multop \, \tilde u = \tilde u$, we may split $A_\Curve \, u $ and $L \, \tilde u$ into
\begin{align*}
	A_\Curve \, u
	=
	\multop\dual \, A_\Curve \, u
	+
	(A_\Curve \, \multop - 	\multop\dual \, A_\Curve) \, u	
	\qand
	L \, \tilde u
	=
	\tilde \multop\dual \, L \, \tilde u
	+
	(L \, \tilde \multop - 	\tilde \multop\dual \, L) \, \tilde u.
\end{align*}

\textbf{Claim:}
\emph{The operators $Q \ceq A_\Curve \, \multop - \multop\dual \, A_\Curve$ and $\tilde Q \ceq L \, \tilde \multop - 	\tilde \multop\dual \, L$ are compact.}
\newline
For $0<\alpha <1$ and $v \colon \Circle \to \AmbSpace$, the Leibniz rule implies
\begin{align*}
	\Diff{\alpha}{\Curve} \cD_{\Curve} (\eta \, v)
	&=	
	(\eta \circ \prx) \, (\Diff{\alpha}{\Curve} \cD_{\Curve} v)
	+
	(\Diff{\alpha}{\Curve} \eta) \, (\cD_{\Curve} v \circ \pry)
	\notag
	\\
	&\qquad
	+
	(\Diff{\alpha}{\Curve} \cD_{\Curve} \eta) \, (v \circ \prx)
	+
	(\cD_{\Curve} \eta \circ \pry) \,  (\Diff{\alpha}{\Curve} v),
\end{align*}
where $\prx$, $\pry \colon \Torus \to \Circle$  are the projections given by
$\prx(x,y) = x$ and $\pry(x,y) = y$.
This allows us to write
\begin{align*}
	\Diff{\strongss}{\Curve} \cD_\Curve (\eta \, u) 
	&= (\eta \circ \prx) \cdot \Diff{s+\diffoffset}{\Curve} \cD_\Curve u +K_1 \, u
	\quad \text{and}
	\\
	\Diff{\weakss}{\Curve} \cD_\Curve (\eta \, w) 
	&= (\eta \circ \prx) \cdot \Diff{s-\diffoffset}{\Curve} \cD_\Curve w +K_2 \, w,
\end{align*}
with compact operators $K_1 \colon \XCg \to \Lebesgue[\weakp][\singularmeasure_\Curve][\Circle^2][\AmbSpace])\dual$ and $K_2 \colon \YCg \to \Lebesgue[\weakp][\singularmeasure_\Curve][\Circle^2][\AmbSpace]$.
Now we have
\begin{align*}
	\MoveEqLeft
	\ninnerprod{(A_\Curve \, \multop - \multop\dual \, A_\Curve) \, u, w}_{\YCgd,\YCg}
	\\
	&=
	\textstyle
	\int_{\Circle^2}
	\bigparen{
	\ninnerprod{
		\Diff{\strongss}{\Curve} \cD_\Curve (\eta \, u),
		\Diff{\weakss}{\Curve} \cD_\Curve w
	}
	-
	\ninnerprod{
		\Diff{\strongss}{\Curve} \cD_\Curve u,
		\Diff{\weakss}{\Curve} \cD_\Curve (\eta \, w)
	}
	}
	\singularmeasure_\Curve
	\\
	&=
	\ninnerprod{
		K_1 u,
		\Diff{\weakss}{\Curve} \cD_\Curve w
	}_{(\Lebesgue[\weakp][\singularmeasure_\Curve])\dual,\Lebesgue[\weakp][\singularmeasure_\Curve]}
	-
	\ninnerprod{
		\Diff{\strongss}{\Curve} \cD_\Curve u,
		K_2 \, w
	}_{(\Lebesgue[\weakp][\singularmeasure_\Curve])\dual,\Lebesgue[\weakp][\singularmeasure_\Curve]}
	,
\end{align*}
thus $A_\Curve \, \multop - \multop\dual \, A_\Curve = (\Diff{\weakss}{\Curve} \cD_\Curve)\dual \, K_1 - K_2\dual \, \Diff{\strongss}{\Curve} \cD_\Curve$ is compact.
This shows the first statement. The second statement is proven analogously.

\bigskip

From this claim it follows that the ``leading term'' of $L \tilde u$ is $\tilde \multop\dual \, L \, \tilde u$. Next we are going to bound $\nnorm{\tilde \multop\dual \, L \, \tilde u}_{\tilde \cY\dual}$.
For this it suffices to test
$\tilde \multop\dual \, L \, \tilde u$ against only
such $\tilde w \in \tilde Y$ with $\supp(\tilde w) \subset V$. More precisely, 
we have
\begin{align*}
	\nnorm{\tilde \multop\dual \, L \, \tilde u}_{\tilde \cY\dual}
	=
	\sup_{\tilde w \in \tilde \cY, \; \supp(\tilde w) \subset  \tilde V}
	\frac{
		\ninnerprod{ L \, \tilde u, \tilde \eta \, \tilde w}_{\tilde\cY\dual,\tilde\cY}
	}{
		\nnorm{\tilde w}_{\tilde \cY}
	}.
\end{align*}
For every such $\tilde w$ 
there is a $w \in \YCg$ with $\supp(w) \subset V$ such that $w = \tilde w \circ \isometry$.
Since also $\tilde u$ and $\tilde \eta$ are constructed in this way from $u$ and $\eta$, we have
\begin{align*}
	\ninnerprod{L \, \tilde u, \tilde \eta \, \tilde w}_{\tilde \cY\dual,\tilde \cY}
	&=
	\textstyle
	\int_{\tilde U} \int_{\tilde U}
	\ninnerprod{
		D^\strongs \, \cD \, \tilde u ,
		D^\weaks \, \cD \, (\tilde \eta_a \, \tilde w)
	}
	\, \dd \mu
	\\
	&=
	\textstyle
	\int_{U} \int_{U}
	\ninnerprod{
		D_\Curve^\strongs \, \cD_\Curve  \, u ,
		D_\Curve^\weaks \, \cD_\Curve \, (\eta \, w)
	}
	\, \dd \mu_\Curve
	\\
	&=	
	\ninnerprod{A_\Curve \, u, \eta \, w}_{\YCgd,\YCg}
	=
	\ninnerprod{\multop\dual \, A_\Curve \, u, w}_{\YCgd,\YCg}
	=
	\ninnerprod{A_\Curve \, u, w}_{\YCgd,\YCg}	
	+
	\ninnerprod{ Q \, u, w}_{\YCgd,\YCg}	
	\\
	&\leq
	\bigparen{ 
		\nnorm{A_\Curve \, u}_{\YCd}
		+ 
		\nnorm{Q \, u}_{\YCd}
	} 
	\, 		
	\nnorm{w}_{\YC}
	.
\end{align*}
By \autoref{lem:LocalizationNormsCircle} and \autoref{lem:LocalizationNormsReals}, there is a $C \geq 0$ (depending only on $\Curve$, $\weaks$, $\weakp$, $U$, $W$, $\tilde U$, and $\tilde W$ such that
$
	\nnorm{w}_{\YC} 
	\leq 
	C \, \nnorm{w}_{\Sobo[\weaks,\weakp][\varrho,\Curve][U]}
	=
	C \, \nnorm{\tilde w}_{\Sobo[\weaks,\weakp][][\tilde U]}	
	\leq 
	C^2 \, \nnorm{\tilde w}_{\tilde \cY}
	.
$
Thus we obtain
\begin{align*}
	\nnorm{\tilde \multop\dual \, L \, \tilde u}_{\tilde \cY\dual}
	\leq C^2 \, 	\bigparen{ 
		\nnorm{A_\Curve \, u}_{\YCd}
		+ 
		\nnorm{Q \, u}_{\YCd}
	}.
\end{align*}
Combined with \eqref{eq:EllipticEstimateL}, this leads to
\begin{align*}
	\nnorm{u}_{\XC}
	&\leq	
	C \, \bigparen{
		\nnorm{A_\Curve \, u}_{\YCd}
		+ 
		\nnorm{Q \, u}_{\YCd}
		+
		\nnorm{\tilde Q \, \tilde u}_{\tilde \cY\dual}		
	}
	.
\end{align*}
Again by \autoref{lem:LocalizationNormsCircle} and \autoref{lem:LocalizationNormsReals}, the mapping $u \mapsto u = H u \mapsto \tilde u$ is continuous. 
Since $\tilde Q$ is compact, the mapping $u \mapsto \tilde Q \, \tilde u$ is compact as well.
This concludes the proof.
\end{proof}

Finally, we pieces together the local estimates from above.

\begin{lemma}[Global elliptic estimate]\label{lem:GlobalEllipticEstimate}
Let $\Curve \in \ConfSpace$. Then there are $C >0$ and a compact, linear operator $K \colon \XC \to \cZ$ into some Banach space $\cZ$ such that
\begin{align*}
	\nnorm{u}_{\XC}
	\leq 
	C\, \bigparen{
		\nnorm{A_\Curve \, u}_{\YCd}
		+
		\nnorm{K \, u}_{\cZ}
	}
	\quad
	\text{for all $u \in \cX$.}
\end{align*}
\end{lemma}
\begin{proof}
We start by covering $\Circle$ by finitely many open sets $W_i \subset\subset U_i \subset \Circle$, $i = 1, \dotsc,k$, as in \autoref{lem:LocalEllipticEstimate} such that there are local elliptic estimates of the form
\begin{align*}
	\nnorm{u}_{\XC}
	\leq 
	C_i \, \bigparen{ 
		\nnorm{A_\Curve \, u}_{\YCd}
		+
		\nnorm{K_i \, u}_{\cZ_i}		
	}
	\quad
	\text{for all $u \in \XC$ with $\supp(u) \subset W_i$},
\end{align*}
with suitable constants $C_i \geq 0$ and suitable compact operators $K_i \colon \XC \to \cZ_i$ into Banach spaces $\cZ_i$.
Now we pick a smooth partition of unity $\set{\varphi_1,\dotsc,\varphi_k} \subset \Holder[\infty][][\Circle][\intervalcc{0,1}]$ subordinate to $\set{W_1,\dotsc,W_k}$.
We denote the corresponding multiplication operators by
$\varPhi_i \colon \XCg \to \XCg$, $\varPhi_i \, u \ceq \varphi_i \, u$
and
$\varPhi_i \colon \YCg \to \YCg$, $\varPhi_i \, w \ceq \varphi_i \, w$.
Now we observe
\begin{align*}
	\nnorm{A_\Curve \, \varPhi_i\, u}_{\cY\dual}
	&\leq 
	\nnorm{\varPhi_i\dual \, A_\Curve \,  u}_{\cY\dual}
	+
	\nnorm{(A_\Curve \, \varPhi_i - \varPhi_i\dual \, A_\Curve) \, u}_{\cY\dual}	
	\\
	&\leq 
	\nnorm{\varPhi_i\dual}  \,\nnorm{A_\Curve \,  u}_{\cY\dual}	
	+
	\nnorm{(A_\Curve \, \varPhi_i - \varPhi_i\dual \, A_\Curve) \, u}_{\cY\dual}.	
\end{align*}
With the triangle inequality, we obtain 
\begin{align*}
	\nnorm{u}_{\cX}
	=
	\textstyle
	\nnorm{\sum_{i=1}^k \varphi_i \, u}_{\cX}	
	&\leq 
	\textstyle
	\sum_{i=1}^k \nnorm{\varphi_i \, u}_{\cX}
	\leq 
	\textstyle		
	\sum_{i=1}^k 	
	C_i \, 	
	\Bigparen{
		\nnorm{A_\Curve \,  \varPhi_i \, u }_{\cY\dual}
		+
		\nnorm{K_i \, \varPhi_i \, u}_{\cZ_i}
	}
	\\
	&\leq 
	\textstyle		
	\sum_{i=1}^k 	
	C_i \, 	
	\Bigparen{
		\nnorm{ \varPhi_i} \, \nnorm{A_\Curve \, u }_{\cY\dual}
		+
		\nnorm{(A_\Curve \, \varPhi_i - \varPhi_i\dual \, A_\Curve) \, u}_{\cY\dual}
		+
		\nnorm{K_i \, \varPhi_i \, u}_{\cZ_i}
	}	
	.
\end{align*}
The separate claim in the proof of \autoref{lem:LocalEllipticEstimate} shows that $A_\Curve \, \varPhi_i - \varPhi_i\dual \, A_\Curve$ is a compact operator.
So setting
$C \ceq \sum_{i=1}^k	C_i \, \nnorm{\varPhi_i} + \max(C_1, \dotsc, C_k)$,
$\cZ \ceq (\YCgd \oplus \cZ_1) \oplus \dotsm \oplus (\YCgd \oplus \cZ_k)$,
$\nnorm{(\eta_1,z_1,\dotsc,\eta_k,z_k)}_{\cZ} \ceq \nnorm{\eta_1}_{\YCgd} + \nnorm{z_1}_{\cZ_1} + \dotsm + \nnorm{\eta_k}_{\YCgd} + \nnorm{z_k}_{\cZ_k}$,
and
\begin{align*}
	K \, u 
	\ceq 
	\Bigparen{
		(A_\Curve \, \varPhi_1 - \varPhi_1\dual \, A_\Curve) \, u,
		K_1 \, u
		,
		\dotsc,
				(A_\Curve \, \varPhi_k - \varPhi_k\dual \, A_\Curve) \, u,
		K_k \, u
	}
\end{align*}
this concludes the proof.
\end{proof}

%% file: Constraint.tex
\section{Constraints}\label{sec:Constraints}
Our aim in this section is to set up constraints on the barycenter and on the parametrization of curves
and to show \autoref{theo:ProjectedGradients}, i.e., well-de\-fined\-ness of the associated projected gradient and its flow.

Here, as before, we abbreviated $\hilbertss \ceq \hilberts -1=\tfrac{1}{2}$.
By the choice of $\diffoffset$ and $\strongp$ (see \eqref{eq:Parameters}), we have $\ConfSpace \subset \SoboC[1+\strongss,\strongp] \subset \SoboC[1,\infty]$. Hence for each $\Curve \in \ConfSpace$, the functions $x \mapsto \nabs{\Curve'(x)}$ and $x \mapsto \nabs{\Curve'(x)}^{-1}$ are both members of $\Sobo[\strongss,\strongp][][\Circle][\R] \hookrightarrow \Lebesgue[\infty][][\Circle][\R]$.
By the chain rule \autoref{lem:ChainRule}, the following mapping is well-defined:
\begin{align}\label{eq:constraint}
	\ConstraintMap \colon \ConfSpace \to \Sobo[\strongss,\strongp][][\Circle][\R] \oplus \AmbSpace,
	\quad
	\ConstraintMap(\Curve) 
	\ceq 
	\Bigparen{
		\log(\nabs{\Curve'}) - \log(L),
		\textstyle
		\int_{\,\Circle} \Curve \, \LineElementC
	}
	.
\end{align}
A curve $\Curve \in \ConfSpace$ is parametrized by constant speed $L$ and has $0$ as barycenter if and only if $\ConstraintMap(\Curve) = (0,0)$. 
Our main task is to prove \autoref{prop:SubmanifoldTheorem} below; it states that the feasible set
\begin{align*}
		\ConstraintMfld \ceq \set{\Curve \in \ConfSpace | \ConstraintMap(\Curve) = (0,0)},
\end{align*}
equipped with a generalized Riesz isomorphism inherited from $\RJ[\ConfSpace]$ is almost a Riemannian manifold, at least in view of the \emph{projected} or \emph{intrinsic} gradients. \autoref{theo:ProjectedGradients} will follow from this immediately.

To this end (and in analogy to the space triple $\XC$, $\HC$, and $\YC$), we introduce the Banach space triple
\begin{align*}
	\XN \ceq \Sobo[\strongss,\strongp][][\Circle][\R] \oplus \AmbSpace,
	\;
	\HN \ceq \Sobo[\hilbertss,2][][\Circle][\R]	 \oplus \AmbSpace,
	\;
	\YN \ceq \Sobo[\weakss,\weakp][][\Circle][\R] \oplus \AmbSpace
\end{align*}
and the continuous dense injections
$
	\Ri[\TargetSpace] \colon \XN \hookrightarrow \HN
$ and
$
	\Rj[\TargetSpace] \colon \HN \hookrightarrow \YN
$.
A straight-forward computation shows that $\ConstraintMap$ is differentiable and that its derivative $D\ConstraintMap(\Curve) \colon \XCg \to \XN$ is given by
\begin{align}\label{eq:Phi-diff}
	D\ConstraintMap(\Curve) \, u = 
	\textstyle	
	\bigparen{
		\;
		\ninnerprod{\cD_\Curve \Curve, \cD_\Curve u}
		\;
		,
		\;
		\int_\Circle \nparen{
			u +\Curve \, \ninnerprod{\cD_\Curve \Curve, \cD_\Curve u}
		}\, \LineElementC
		\;
	}
	\quad
	\text{for $u \in \XCg$.}
\end{align}
\autoref{lem:MultiplicationLemma} implies that
$u \mapsto \ninnerprod{\cD_\Curve \Curve ,\cD_\Curve u}$
induces well-defined and continuous linear operators
$\XCg \to \XN$,
$\HCg \to \HN$,
and
$\YCg \to \YN$,
provided that $\diffoffset > 0$, and $\strongp \geq 2$.
With the usual convention
$\RX[\ConstraintMap][\Curve] \ceq D\ConstraintMap(\Curve)$
etc.\ we generate a triple
$(\RX[\ConstraintMap][\Curve],\RH[\ConstraintMap][\Curve],\RY[\ConstraintMap][\Curve])$ of continuous, linear operators that makes the following diagram commutative:
\begin{equation}
	\begin{tikzcd}[]
	\XCg
		\ar[d,hook, two heads,"{\Ri[\ConfSpace]}"]
		\ar[rr, "{\RX[\ConstraintMap][\Curve]}"]
	&&\XN
		\ar[d,hook, two heads,"{\Ri[\TargetSpace]}"]	
	\\
	\HCg
		\ar[d,hook, two heads,"{\Rj[\ConfSpace]}"]	
		\ar[rr, "{\RH[\ConstraintMap][\Curve]}"]	
	&&\HN
		\ar[d,hook, two heads,"{\Rj[\TargetSpace]}"]		
	\\
	\YCg
		\ar[rr, "{\RY[\ConstraintMap][\Curve]}"]	
	&&\YN
	\end{tikzcd}
	\label{eq:PhiDiagramm}
\end{equation}
By \autoref{lem:RightInverse}, the mapping $\ConstraintMap$ is a submersion. Thus the implicit function theorem implies that the set $\ConstraintMfld$ is a Banach submanifold of $\ConfSpace$.
For $\Curve \in \ConstraintMfld$ and $\RZ \in \set{\RX,\RH,\RY}$, define
$
	\RZ[\ConstraintMfld][\Curve] \ceq \ker( \RZ[\ConstraintMap][\Curve])
$.
The set $\RZ[\ConstraintMfld] \ceq \coprod_{\Curve \in \ConstraintMfld} \set{\Curve} \times \RZ[\ConstraintMfld][\Curve]$ together with the footpoint map $\pi_{\RZ[\ConstraintMfld]} \colon \RZ[\ConstraintMfld] \to \ConstraintMfld$ constitutes a smooth Banach vector bundle over $\ConstraintMfld$ and we have $T\ConstraintMfld = \RX[\ConstraintMfld]$. The Banach spaces $\RH[\ConstraintMfld][\Curve]$ and $\RY[\ConstraintMfld][\Curve]$ are the completions of $T_\Curve \ConstraintMfld = \ker( D \ConstraintMap(\Curve))$ with respect to the topologies of $\HCg$ and $\YCg$, respectively.
Via Galerkin subspace projection, we may define linear operators
$\RI[\ConstraintMfld]\at_\Curve \colon \RH[\ConstraintMfld][\Curve] \to \RH[\ConstraintMfld][\Curve][\dual]$
and
$\RJ[\ConstraintMfld]\at_\Curve \colon \RX[\ConstraintMfld][\Curve] \to \RY[\ConstraintMfld][\Curve][\dual]$
by
\begin{align*}
	\ninnerprod{\RI[\ConstraintMfld]\at_\Curve \, v_1, v_2}
	\ceq 
	\ninnerprod{\RI[\ConfSpace]\at_\Curve \, v_1, v_2}
	\qand
	\ninnerprod{\RJ[\ConstraintMfld]\at_\Curve \, u, w}
	\ceq \ninnerprod{\RJ[\ConfSpace]\at_\Curve \, u, w}
\end{align*}
for $v_1$, $v_2 \in \RH[\ConstraintMfld][\Curve]$, $u \in \RX[\ConstraintMfld][\Curve]$, and $w \in \RY[\ConstraintMfld][\Curve]$.
The mappings $\Ri[\ConfSpace]$ and $\Rj[\ConfSpace]$ induce continuous injections
$\Ri[\ConstraintMfld]\at_\Curve \colon \RX[\ConstraintMfld][\Curve] \hookrightarrow \RH[\ConstraintMfld][\Curve]$ 
and
$\Rj[\ConstraintMfld]\at_\Curve \colon \RH[\ConstraintMfld][\Curve] \hookrightarrow \RY[\ConstraintMfld][\Curve]$.
By \eqref{eq:j'J=Ii}, we have $\Rj[\ConstraintMfld]\dual \, \RJ[\ConstraintMfld] = \RI[\ConstraintMfld] \, \Ri[\ConstraintMfld]$.

We define the \emph{intrinsic gradient}
$\grad_\ConstraintMfld(\Energy|_\ConstraintMfld)\at_\Curve$ by 
\begin{align*}
	\ninnerprod{\RJ[\ConstraintMfld]\at_\Curve \grad_\ConstraintMfld(\Energy|_\ConstraintMfld)\at_\Curve, w}
	=
	D(\Energy|_\ConstraintMfld)(\Curve) \,w 
	\quad
	\text{for all $w \in \RY[\ConstraintMfld][\Curve]$}
\end{align*}
or simply by $\grad_\ConstraintMfld(\Energy|_\ConstraintMfld)\at_\Curve  \ceq (\RJ[\ConstraintMfld]\at_\Curve)^{-1} D(\Energy|_\ConstraintMfld)(\Curve)$.
Its well-definedness is established by the following theorem which states that $\ConstraintMfld$ has a ``nearly Riemannian structure''. Note that $\ConstraintMfld$ cannot support a Riemannian structure because the tangent space $T_\Curve \ConstraintMfld$ is not \emph{Hilbertable} in the sense that there is no positive-definite bilinear form whose norm topologizes $T_\Curve \ConstraintMfld$.

\begin{theorem}\label{prop:SubmanifoldTheorem}
The operators $\RI[\ConstraintMfld]\at_\Curve$ and
$\RJ[\ConstraintMfld]\at_\Curve$ define a family of
continuous and continuously invertible operators.
\end{theorem}

\begin{proof}
Denote by $\Morphism \colon \ConstraintMfld \hookrightarrow \ConfSpace$ and by 
\begin{align*}
	\RX[\Morphism][\Curve]  \colon \RX[\ConstraintMfld][\Curve] \hookrightarrow \XCg,
	\quad
	\RH[\Morphism][\Curve]  \colon \RH[\ConstraintMfld][\Curve] \hookrightarrow \HCg,
	\qand
	\RY[\Morphism][\Curve]  \colon \RY[\ConstraintMfld][\Curve] \hookrightarrow \YCg
\end{align*}
the canonical injections
which give rise to dual maps
\begin{align*}
	\RX[\Morphism][\Curve][\dual]  \colon \XCgd \hookrightarrow \RX[\ConstraintMfld][\Curve][\dual],
	\quad
	\RH[\Morphism][\Curve][\dual]  \colon \HCgd \hookrightarrow \RH[\ConstraintMfld][\Curve][\dual],
	\qand
	\RY[\Morphism][\Curve][\dual]  \colon \YCgd \hookrightarrow \RY[\ConstraintMfld][\Curve][\dual].
\end{align*}
Observe that
$\RI[\ConstraintMfld] \at_\Curve \ceq \RH[\Morphism][\Curve][\dual] \; \RI[\ConfSpace]\at_\Curve \; \RH[\Morphism][\Curve]$
and
$\RJ[\ConstraintMfld] \at_\Curve \ceq \RY[\Morphism][\Curve][\dual] \; \RJ[\ConfSpace]\at_\Curve \; \RX[\Morphism][\Curve]$.

The Galerkin projection of a Hilbert space's Riesz isomorphism onto a closed subspace equals the Riesz isomorphism of the restricted scalar product. So the invertibility of $\RI[\ConstraintMfld]$ is straight-forward.
The nontrivial part here is to show that $\RJ[\ConstraintMfld]$ is continuously invertible. 
By the open mapping theorem, it suffices to show that $\RJ[\ConstraintMfld]$ is both injective and surjective.
Injectivity can be deduced from the injectivity of $\RI[\ConfSpace]$, $\RH[\Morphism][\Curve]$, and $ \Ri[\ConstraintMfld]$ as follows:
Let $u \in \ker(\RJ[\ConstraintMfld] \at_\Curve)$ and put 
$v \ceq \RH[\Morphism][\Curve] \; \Ri[\ConstraintMfld] \at_\Curve \; u$. Now the following shows that  $u=0$:
\begin{align*}
	\ninnerprod{ \RI[\ConfSpace]\at_\Curve \; v, v}
	&=
	\ninnerprod{ 
		\RI[\ConfSpace]\at_\Curve \; \RH[\Morphism][\Curve] \; \Ri[\ConstraintMfld] \at_\Curve \; u,
		\RH[\Morphism][\Curve] \; \Ri[\ConstraintMfld] \at_\Curve \; u
	}
	=
	\ninnerprod{ 
		(\RH[\Morphism][\Curve][\dual] \; \RI[\ConfSpace]\at_\Curve \; \RH[\Morphism][\Curve]) \; \Ri[\ConstraintMfld] \at_\Curve \; u,
		\Ri[\ConstraintMfld] \at_\Curve \; u
	}
	\\
	&=
	\ninnerprod{ 
		\RI[\ConstraintMfld]\at_\Curve \; \Ri[\ConstraintMfld] \at_\Curve \; u,
		\Ri[\ConstraintMfld] \at_\Curve \; u
	}
	=
	\ninnerprod{ 
		\RJ[\ConstraintMfld]\at_\Curve \; u,
		\Rj[\ConstraintMfld]\at_\Curve  \; \Ri[\ConstraintMfld] \at_\Curve \; u
	}
	= 0.
\end{align*}

In order to establish surjectivity of $\RJ[\ConstraintMfld]$, we fix an arbitrary $\eta \in \RY[\ConstraintMfld][\Curve][\dual]$.
By the Hahn-Banach theorem, the mapping $\RY[\Morphism][\Curve][\dual] \colon \YCgd \to \RY[\ConstraintMfld][\Curve][\dual]$ is surjective. 
Thus there is an  $\eta_0 \in \YCgd$ with 
$\RY[\Morphism][\Curve][\dual] \, \eta_0 = \eta$.
By \autoref{lem:Saddlepointmatrix} below, the saddle point problem
\begin{align}
	\begin{pmatrix}
		\RJ[\ConfSpace]\at_\Curve & \RY[\ConstraintMap][\Curve][\dual]\\
		\RX[\ConstraintMap][\Curve] &0
	\end{pmatrix}
	\begin{pmatrix}
		u_0 \\
		\lambda_0
	\end{pmatrix}
	=
	\begin{pmatrix}
		\eta_0\\
		0
	\end{pmatrix}
	\label{eq:SaddlePointEquation}
\end{align}
has a unique solution with $u_0 \in \XCg$ and $\lambda_0 \in \YNgd$.
In particular, we have $u_0 \in \ker(\RX[\ConstraintMap][\Curve])$, hence we may write $u_0 = \RX[\Morphism][\Curve] \, u$ with $u \in \RX[\ConstraintMfld][\Curve]$.
Thus, we have
\begin{align*}
	(\RJ[\ConstraintMfld]\at_\Curve) \, u
	= \RY[\Morphism][\Curve][\dual] \, (\RJ[\ConfSpace]\at_\Curve) \, \RX[\Morphism][\Curve] \, u
	= \RY[\Morphism][\Curve][\dual] \, (\RJ[\ConfSpace]\at_\Curve) \, u_0
	= \RY[\Morphism][\Curve][\dual] \, (\eta_0 - \RY[\ConstraintMap][\Curve][\dual] \,\lambda_0)
	= \eta
\end{align*}
where we used the fact that
$\RY[\Morphism][\Curve][\dual] \, \RY[\ConstraintMap][\Curve][\dual] = \RY[(\ConstraintMap\circ\Morphism)][\Curve][\dual]=0$
in the last step which follows by
$\ConstraintMap\circ\Morphism=0$ due to
$\ConstraintMfld=\ConstraintMap^{-1}(0)$.
\end{proof}

This leads us immediately to the proof of our main result \autoref{theo:ProjectedGradients}:
\begin{proof}
The proof of \autoref{prop:SubmanifoldTheorem} shows that we the gradient
$u_0 \ceq \grad_\ConstraintMfld(\Energy|_\ConstraintMfld)\at_\Curve$ can be computed
by solving \eqref{eq:SaddlePointEquation} with $\eta_0 = D\Energy(\Curve)$.
\autoref{cor:PseudoinverseOrthoprojector} below shows that 
it is also the \emph{projected gradient}, i.e.,
$\grad_\ConstraintMfld(\Energy|_\ConstraintMfld)\at_\Curve$ coincides with the
$\RI[\ConstraintMfld]\at_\Curve$-orthogonal projection of $\grad(\Energy)\at_\Curve$ onto $T_\Curve \ConstraintMfld = \ker(D \ConstraintMap(\Curve))$.
Because $D\Energy(\Curve)$ and the saddle point matrix from \eqref{eq:SaddlePointEquation} depend locally Lipschitz continuously on $\Curve$, the gradient $\grad_\ConstraintMfld(\Energy|_\ConstraintMfld)$ is a locally Lipschitz continuous vector field on $\ConstraintMfld$.
\end{proof}

\subsection*{Details}

Now, statements and proofs of the auxiliary results are in order.

\begin{lemma}[Right inverse]\label{lem:RightInverse}
The triple $(\RX[\ConstraintMap][\Curve] ,\RH[\ConstraintMap][\Curve], \RY[\ConstraintMap][\Curve])$ induced by the derivative $D\ConstraintMap(\Curve)$ allows a triple $(\RX[B][\Curve] ,\RH[B][\Curve], \RY[B][\Curve])$ of continuous right inverses such that the following diagram commutes:
\begin{equation}
\begin{tikzcd}[]
	\XNg
		\ar[d,hook, two heads,"{\Ri[\TargetSpace]}"]
		\ar[rr, "{\RX[B][\Curve]}"]
	&&\XCg
		\ar[d,hook, two heads,"{\Ri[\ConfSpace]}"]	
	\\
	\HNg
		\ar[d,hook, two heads,"{\Rj[\TargetSpace]}"]	
		\ar[rr, "{\RH[B][\Curve]}"]	
	&&\HCg
		\ar[d,hook, two heads,"{\Rj[\ConfSpace]}"]		
	\\
	\YNg
		\ar[rr, "{\RY[B][\Curve]}"]	
	&&\YCg\nospaceperiod
\end{tikzcd}
	\label{eq:RightInverseDiagram}
\end{equation}
Moreover $(\RX[B][\Curve] ,\RH[B][\Curve], \RY[B][\Curve])$ depend smoothly on $\Curve$ and in particular, they are locally Lipschitz continuous.
\end{lemma}
\begin{proof}
Denote by $\pr_\Curve^\perp(x) \ceq \id_{\AmbSpace} - \Tangent_\Curve(x) \otimes \ninnerprod{\Tangent_\Curve(x),\cdot}$ the orthogonal projector onto the orthogonal complement of $\Tangent_\Curve(x)$.
Fix a $\RZ \in \set{\RX,\RH,\RY}$ and a $(\xi,U) \in \ZNg$. 
Denote the length of $\Curve$ by $L = \int_\Circle \LineElement_\Curve$. 
Fix a given point $y_0 \in \Circle$ and define
\begin{align*}
	v(y)
	&\ceq 
	\textstyle
	\int_{y_0}^y \bigparen{
		\Tangent_\Curve(x) \, \xi(x) 
		+ \pr_\Curve^\perp(x) \, \tilde U
	}\, \LineElement_\Curve(x)
	\quad \text{and}
	\\
	u(y)
	&\ceq
	\textstyle 
	v(y)
	+\frac{1}{L} \bigparen{ 
		U 
		- 
		\int_\Circle (v(x) + \Curve(x) \, \xi(x)) \, \LineElement_\Curve(x)
	}
\end{align*}
with a vector $\tilde U \in \AmbSpace$ to be determined later.
This way,
the components of $D\ConstraintMap(\Curve) \, u$ stated in~\eqref{eq:Phi-diff} amount to
\begin{align}
	\ninnerprod{\cD_\Curve \Curve, \cD_\Curve u}
	=
	\ninnerprod{\Tangent_\Curve, \cD_\Curve u}
	= \xi,
	\qand
	\textstyle
	\int_\Circle \bigparen{ u +\Curve \, \ninnerprod{\cD_\Curve \Curve, \cD_\Curve u} }\, \LineElementC
	= U	
	.
	\label{eq:Constraint}	
\end{align}
Using the product rule \autoref{lem:MultiplicationLemma} (for which $p\ge2$ is crucial here), we see that
$u$ is a member of $\ZCg$, provided that we can find a $\tilde U \in \AmbSpace$ such that $u$ becomes continuous at $y = y_0$. For this, it is necessary and sufficient that
$
	\textstyle
	\int_\Circle \bigparen{
		\Tangent_\Curve \, \xi
		+ \pr_\Curve^\perp \, \tilde U
	}\, \LineElement_\Curve
	= 0.
$
Define the vector $b_\Curve(\xi) \in \AmbSpace$
and the symmetric matrix $\varTheta_\Curve \in \End(\AmbSpace)$ by
$
	b_\Curve(\xi) \ceq \textstyle \int_\Circle \Tangent_\Curve \, \xi \, \LineElement_\Curve
$
and
$
	\varTheta_\Curve \ceq \textstyle \int_\Circle  \pr_\Curve^\perp \, \LineElement_\Curve.
$
\emph{Assume} $\varTheta_\Curve$ is not invertible. Then there is a unit vector $V \in \AmbSpace$ in its kernel and we have
$
	0 = \ninnerprod{V,\varTheta_\Curve \, V}
	= 	\textstyle \int_\Circle \bigparen{ \nabs{V}^2 - \ninnerprod{\Tangent_\Curve,V}^2} \, \LineElementC
$.
But that means that $\Tangent_\Curve(x) = \pm V$ has to hold for almost every $x\in\Circle$. Since $\Tangent_\Curve$ is continuous, this implies that $\Curve$ is a straight line, which is impossible due to $\Curve$ being closed. This \emph{contradicts} our assumption and thus $\varTheta_\Curve$ must be invertible.
So we may choose $\tilde U \ceq - \varTheta_\Curve^{-1} \, b_\Curve(\xi)$ and put
$
	\RZ[B][\Curve] \, (\xi, U) \ceq u.
$
By \eqref{eq:Constraint}, $\RZ[B][\Curve]$ is indeed a right inverse of $\RZ[\ConstraintMap][\Curve]$. Finally, it is only a matter of some elementary calculus to show that $\RZ[\ConstraintMap][\Curve]$ and $\RZ[B][\Curve]$ depend smoothly on $\Curve$.
\end{proof}

In analogy to \autoref{prop:MetricDefinition}, we may equip the target space $\TargetSpace \ceq \RX[\TargetSpace]$ with the following Riesz isomorphisms. This will help us to generalize the concept of adjoint operators between Hilbert spaces.

\begin{proposition}\label{prop:MetricDefinition2}
Analogously to \autoref{prop:MetricDefinition}, we define the $\Curve$-dependent, linear operators $\RI[\TargetSpace]\at_{\Curve}  \colon \HNg \to \HNgd$ and $\RJ[\TargetSpace]\at_{\Curve}  \colon \XNg \to \YNgd$ as follows:
	\begin{align*}
		\ninnerprod{\RI[\TargetSpace] \at_{\Curve} \,(\eta_1,V_1),(\eta_2,V_2)}
		&\ceq 
		\textstyle
		\int_\Torus
			(\diff{\sigma}{\Curve} \eta_1)
			\,
			(\diff{\sigma}{\Curve} \eta_2)
		\, \singularmeasure_\Curve
		+
		\int_{\Circle}	
			\eta_1
			\,
			\eta_2
		\, 
		\LineElement_\Curve
		+ \ninnerprod{V_1,V_2},
	\\
	\ninnerprod{\RJ[\TargetSpace]\at_{\Curve}\,(\xi,U),(\psi,W)}
		&\ceq
		\textstyle		
		\int_{\Torus}
			(\diff{\strongss}{\Curve} \xi)
			\,
			(\diff{\weakss}{\Curve} \psi)
		\, \singularmeasure_\Curve
		+
		\int_{\Circle}
			\xi
			\, 
			\psi
		\, \LineElement_\Curve
		+ \ninnerprod{U,W}	
	,
\end{align*}
for $(\eta_1,V_1)$, $(\eta_2,V_2) \in \HNg$, $(\xi,U) \in \XNg$, and $(\psi,W) \in \YNg$.
These operators are well-defined, continuous, and continuously invertible, and they 
satisfy $\Rj[\TargetSpace]\dual \, \RJ[\TargetSpace] = \RI[\TargetSpace] \, \Ri[\TargetSpace]$.
Moreover, $\RI[\TargetSpace] \colon \ConfSpace \to L(\HN;\HNd)$ and $\RJ[\TargetSpace] \colon \ConfSpace \to L(\XN;\YNd)$ are of class $\Holder[1]$.
\end{proposition}
The proof is entirely along the lines of the proof of \autoref{prop:MetricDefinition}.

\begin{lemma}[Saddle point matrix]\label{lem:Saddlepointmatrix}
For each $\Curve \in \ConfSpace$, the saddle point matrix
\begin{align*}
	\SaddlePointMatrix\at_\Curve \ceq
	\begin{pmatrix}
		\RJ[\ConfSpace]\at_\Curve 
		& 
		\RY[\ConstraintMap][\Curve][\dual]
		\\
		\RX[\ConstraintMap][\Curve] &0
	\end{pmatrix}
	\colon
		\XCg \oplus \YNgd
	\longrightarrow
		\YCgd \oplus \XNg
\end{align*}
is continuously invertible.
\end{lemma}
\begin{proof}
Let $B$ be as in \autoref{lem:RightInverse} above.
As $\RJ[\ConfSpace]$ is invertible, the saddle point matrix $\SaddlePointMatrix\at_\Curve$ is invertible if and only if its Schur complement $S \ceq - \RX[\ConstraintMap][\Curve] \; (\RJ[\ConfSpace]\at_\Curve)^{-1} \, \RY[\ConstraintMap][\Curve][\dual]$ is invertible.
In analogy to the adjoint operators
$\RH[\ConstraintMap][\Curve][\adj] = (\RI[\ConfSpace]\at_\Curve)^{-1} \, \RH[\ConstraintMap][\Curve][\dual] \, (\RI[\TargetSpace]\at_\Curve) $
and
$\RH[B][\Curve][\adj] = (\RI[\TargetSpace]\at_\Curve)^{-1} \, \RH[B][\Curve][\dual]\; (\RI[\ConfSpace]\at_\Curve)$, we introduce the generalized adjoint operators 
\begin{align*}
	\RX[\ConstraintMap][\Curve][\adj] 
	&\ceq (\RJ[\ConfSpace]\at_\Curve)^{-1} \, \RY[\ConstraintMap][\Curve][\dual] \, (\RJ[\TargetSpace]\at_\Curve) 
	,
	&
	\RX[B][\Curve][\adj] 
	&\ceq (\RJ[\TargetSpace]\at_\Curve)^{-1} \, \RY[B][\Curve][\dual]\; (\RJ[\ConfSpace]\at_\Curve)
	.
\end{align*}
Observe that we may express the Schur complement as
$S = - \RX[\ConstraintMap][\Curve] \, \RX[\ConstraintMap][\Curve][\adj]\, (\RJ[\TargetSpace]\at_\Curve)^{-1}$, hence it suffices to show that $\RX[\ConstraintMap][\Curve] \, \RX[\ConstraintMap][\Curve][\adj]$ is continuously invertible.
Since $\RH[\ConstraintMap][\Curve]$ is surjective, $\RH[\ConstraintMap][\Curve]\, \RH[\ConstraintMap][\Curve][\adj]$ is invertible.
Utilizing the identities
$\Rj[\ConfSpace]\dual \, \RJ[\ConfSpace] = \RI[\ConfSpace] \, \Ri[\ConfSpace]$ and
$\Rj[\TargetSpace]\dual \, \RJ[\TargetSpace] = \RI[\TargetSpace] \, \Ri[\TargetSpace]$
as well as the diagram \eqref{eq:PhiDiagramm},
one verifies that
$\RH[\ConstraintMap][\Curve]\, \RH[\ConstraintMap][\Curve][\adj] \, \Ri[\TargetSpace] = \Ri[\TargetSpace] \, \RX[\ConstraintMap][\Curve]\, \RX[\ConstraintMap][\Curve][\adj]$.
This shows that $\RX[\ConstraintMap][\Curve]\, \RX[\ConstraintMap][\Curve][\adj]$ is injective.
By \autoref{lem:BstarBinv} below; the operator $\RX[B][\Curve][\adj] \, \RX[B][\Curve]$ is invertible. 
This allows us to define the projector $Q \ceq \RX[B][\Curve]\,(\RX[B][\Curve][\adj] \, \RX[B][\Curve])^{-1} \RX[B][\Curve][\adj]$.
With 
$\ima(\RX[\ConstraintMap][\Curve][\adj]) = \ima(Q\,\RX[\ConstraintMap][\Curve][\adj]) \oplus \ima((1-Q)\,\RX[\ConstraintMap][\Curve][\adj])$, 
$\RX[\ConstraintMap][\Curve] \, \RX[B][\Curve] = \id_{\RX[\TargetSpace]}$, 
and
$\RX[B][\Curve][\adj] \, \RX[\ConstraintMap][\Curve][\adj] = \id_{\RX[\TargetSpace]}$,
we can verify that $\RX[\ConstraintMap][\Curve] \, \RX[\ConstraintMap][\Curve][\adj]$ is surjective:
\begin{align*}
	\ima(\RX[\ConstraintMap][\Curve] \, \RX[\ConstraintMap][\Curve][\adj]) 
	&\supset \RX[\ConstraintMap][\Curve] \,(\ima(Q\,\RX[\ConstraintMap][\Curve][\adj])) 
	= \ima \nparen{\RX[\ConstraintMap][\Curve] \, Q \, \RX[\ConstraintMap][\Curve][\adj] }
	\\
	&= \ima\bigparen{\RX[\ConstraintMap][\Curve] \, \RX[B][\Curve]\,(\RX[B][\Curve][\adj] \, \RX[B][\Curve])^{-1} \RX[B][\Curve][\adj] \, \RX[\ConstraintMap][\Curve][\adj] }
	= \ima \bigparen{(\RX[B][\Curve][\adj] \,  \RX[B][\Curve])^{-1}} 
	= \XNg.
\end{align*}
Finally, the open mapping theorem implies that 
$\RX[\ConstraintMap][\Curve]\, \RX[\ConstraintMap][\Curve][\adj]$
is continuously invertible.
\end{proof}

\begin{lemma}[Invertibility of $B^*B$]\label{lem:BstarBinv}
For each $\Curve \in \ConfSpace$, the linear operator 
$\RX[B][\Curve][\adj]\;\RX[B][\Curve]$ is continuously invertible.
\end{lemma}
\begin{proof}
We have
$\RX[B][\Curve][\adj]\;\RX[B][\Curve]=
(\RJ[\TargetSpace]\at_\Curve)^{-1} \, T$ with $T \ceq \RY[B][\Curve][\dual] \; (\RJ[\ConfSpace]\at_\Curve) \; \RX[B][\Curve]$.
Thus it suffices to show that $T$ is invertible.
Let $\bar \xi \ceq (\xi,U) \in \XNg$ and $\bar \psi \ceq (\psi,W) \in \YNg$.
Put $u \ceq \RX[B][\Curve]\;\bar \xi$ and $w \ceq \RY[B][\Curve]\; \bar \psi$.
By construction, we have
\begin{align*}
	\ninnerprod{T \, \bar \xi,\bar \psi}
	&= 
	\ninnerprod{\RJ[\ConfSpace]\at_\Curve \; \RX[B][\Curve]\; \bar \xi, \RY[B][\Curve]\;\bar \psi}
	= 
	\textstyle
	\int_\Torus\!
		\ninnerprod{
		\diff{\strongss}{\Curve} \cD_\Curve u
		, 
		\diff{\weakss}{\Curve} \cD_\Curve w
		}
		\, \mu_\Curve
	+ \lot
\end{align*}
With the notation from the proof of \autoref{lem:RightInverse}, we put
$\tilde U \ceq - \varTheta_\Curve^{-1} \, b_\Curve (\xi)$ and
$\tilde W \ceq - \varTheta_\Curve^{-1} \, b_\Curve (\psi)$.
Now we observe that
\begin{align*}
	\diff{\strongss}{\Curve} \cD_\Curve u 
	&= 
	 (\Tangent_\Curve \circ \prx )\; (\diff{\strongss}{\Curve} \xi)
	 + (\diff{\strongss}{\Curve} \Tangent_\Curve) \; (\xi \circ \pry)
	 + \diff{\strongss}{\Curve} \pr_\Curve^\perp \; \tilde U
	\qand
	\\
	\diff{\weakss}{\Curve} \cD_\Curve w
	&= 
	 (\Tangent_\Curve \circ \prx)\; (\diff{\weakss}{\Curve}  \psi)
	 + (\diff{\weakss}{\Curve} \Tangent_\Curve)  \; (\psi \circ  \pry)
	 + \diff{\weakss}{\Curve}  \pr_\Curve^\perp \; \tilde W.
\end{align*}
Writing only the terms of highest order in $\xi$ and $\psi$, we obtain
\begin{align*}
	\textstyle
	\int_\Torus
	\ninnerprod{
		\diff{\strongss}{\Curve}  \cD_{\Curve} u
		, 
		\diff{\weakss}{\Curve}  \cD_{\Curve} w
	} \, \singularmeasure_\Curve
	=
	\int_\Torus
		(\diff{\strongss}{\Curve} \, \xi)
		\,
		(\diff{\weakss}{\Curve} \,\psi)
	\, \singularmeasure_\Curve
	+ \lot
\end{align*}
The latter pairing
is identical to  
$\ninnerprod{\RJ[\TargetSpace]|_{\Curve} \, \bar \xi, \bar \psi}$
up to the term  $\int_{\Circle} \xi \, \psi \, \omega_\Curve + \ninnerprod{U,W}$,
which is a combination of lower order and finite rank, thus represents a compact operator $\XNg \to \YNgd$.
This means that $T$ is a compact perturbation of $\RJ[\TargetSpace]|_{\Curve}$ and thus a Fredholm operator of index $0$.
Hence it suffices to show that $T$ is injective.
Let $\bar \xi \in \ker(T)$ and put $\bar \eta \ceq \Ri[\TargetSpace] \; \bar \xi$. 
A diagram chase in \eqref{eq:RightInverseDiagram} and \eqref{eq:j'J=Ii} yields
\begin{align*}
	\ninnerprod{
	 \RI[\ConfSpace]\at_\Curve \; \RH[B][\Curve] \; \bar \eta,
	 \RH[B][\Curve] \; \bar \eta
	}
	&=
	\ninnerprod{
	 \RI[\ConfSpace]\at_\Curve \; \RH[B][\Curve] \; \Ri[\TargetSpace] \; \bar \xi,
	 \RH[B][\Curve] \; \bar \eta
	}
	=
	\ninnerprod{
	 \RI[\ConfSpace]\at_\Curve \; \Ri[\ConfSpace] \; \RX[B][\Curve] \; \bar \xi,
	 \RH[B][\Curve] \; \bar \eta
	}
	\\
	&=
	\ninnerprod{
	 \Rj[\ConfSpace]\dual \;  \RJ[\ConfSpace]\at_\Curve \;\RX[B][\Curve] \; \bar \xi,
	 \RH[B][\Curve] \; \bar \eta
	}
	=
	\ninnerprod{
	 \RJ[\ConfSpace]\at_\Curve \; \RX[B][\Curve] \; \bar \xi,
	 \Rj[\ConfSpace] \; \RH[B][\Curve] \; \bar \eta
	}
	\\
	&=
	\ninnerprod{
	 \RJ[\ConfSpace]\at_\Curve \; \RX[B][\Curve] \; \bar \xi,
	 \RY[B][\Curve] \; \Rj[\TargetSpace] \; \bar \eta
	}
	=
	\ninnerprod{
		T \, \bar \xi,		
		\Rj[\TargetSpace] \; \bar \eta
	}
	= 0.
\end{align*}
Since $\RH[B][\Curve]$ is injective and since 
$\ninnerprod{\RI[\ConfSpace]\at_\Curve \, \cdot, \cdot }$ is a scalar product on $\HCg$, this implies $\Ri[\TargetSpace] \; \bar \xi =  \bar \eta = 0$.
The injectivity of $\Ri[\TargetSpace]$ yields $\bar \xi = 0$ and we see that $T$ is injective. So as an injective Fredholm operator with index zero, $T$ must also be surjective, hence continuously invertible by the open mapping theorem.
\end{proof}

The invertibility of the saddle point matrix leads to the following generalizations of 
(i) the Moore-Penrose pseudoinverse of a surjective operator between Hilbert spaces
and
(ii) the orthoprojector onto the orthogonal complement of the operator's null space. Being able to reduce the action of these operators to solving a linear saddle point system will be crucial for applications (see Section~\ref{sec:ComputationalTreatment}).

\begin{corollary}\label{cor:PseudoinverseOrthoprojector}
The Moore-Penrose pseudoinverse 
$\RH[\ConstraintMap][\Curve][\pinv] \ceq \RH[\ConstraintMap][\Curve][\adj]\;(\RH[\ConstraintMap][\Curve]\;\RH[\ConstraintMap][\Curve])^{-1}$ of $\RH[\ConstraintMap][\Curve]$
and
the orthoprojector
$\RH[P][\Curve] \ceq \RH[\ConstraintMap][\Curve][\pinv] \; \RH[\ConstraintMap][\Curve]$ with kernel $\ker(\RH[\ConstraintMap][\Curve])$
can be completed to a continuous right inverse 
$\RX[\ConstraintMap][\Curve][\pinv] \ceq \RX[\ConstraintMap][\Curve][\adj]\;(\RX[\ConstraintMap][\Curve]\;\RX[\ConstraintMap][\Curve][\adj])^{-1}$
of
$\RX[\ConstraintMap][\Curve]$
and a continuous projector
$\RX[P][\Curve] = \RX[\ConstraintMap][\Curve][\pinv] \; \RX[\ConstraintMap][\Curve]$.
For $\xi \in \XNg$ and $\tilde u \in \XCg$,
the operators can be evaluated by solving the following saddle point systems
\begin{align*}
	\begin{pmatrix}
		\RJ[\ConfSpace]\at_\Curve & \RY[\ConstraintMap][\Curve][\dual]\\
		\RX[\ConstraintMap][\Curve] &0
	\end{pmatrix}
	\begin{pmatrix}
		\RX[\ConstraintMap][\Curve][\pinv] \, \xi \\
		\lambda
	\end{pmatrix}
	=
	\begin{pmatrix}
		0\\
		\xi
	\end{pmatrix}
	\qand
	\begin{pmatrix}
		\RJ[\ConfSpace]\at_\Curve & \RY[\ConstraintMap][\Curve][\dual]\\
		\RX[\ConstraintMap][\Curve] &0
	\end{pmatrix}
	\begin{pmatrix}
		\RX[P][\Curve] \, \tilde u \\
		\mu
	\end{pmatrix}
	=
	\begin{pmatrix}
		\RJ[\ConfSpace]\at_\Curve \, \tilde u\\
		0
	\end{pmatrix}	
	,
\end{align*}
where $\lambda$, $\mu \in \YNgd$ act as Lagrange multipliers.
\end{corollary}

%% file: NumericalExperiments.tex
\section{Computational treatment}\label{sec:ComputationalTreatment}

For the ease of use, we discretize curves by polygonal lines and approximate the Möbius energy and the Riesz isomorphisms from \autoref{prop:MetricDefinition} by simple quadrature rules. 
In the language of finite element analysis, we employ a nonconforming Ritz--Galerkin scheme because the discrete ansatz space is not a subset of the smooth configuration space.
We try to outline a discrete setting that can be applied also to more general self-avoiding energies; therefore, we do not care about Möbius-invariance of the energy, although Möbius-invariant discretizations have already been proposed (see e.g., \cite{MR1702037} and \cite{2018arXiv180907984B,2019arXiv190406818B}).

\subsection{Spatial discretization}\label{sec:Discretization}

Let $\Triangulation$ denote a partition of $\Circle$ with vertex set $\Vertices(\Triangulation) \subset \Circle$
and edge set $\Edges(\Triangulation) \subset \Vertices(\Triangulation) \times \Vertices(\Triangulation)$. 
Denote the number of edges by $N$.
If the partition is sufficiently fine, i.e., $\MaxRadius(\Triangulation) \ceq \max_{\Edge \in \Edges(\Triangulation)} \nabs{\Edge}$ is sufficiently small,
then we may identify each edge with the closed, oriented interval connecting its end vertices. 
For an edge $\Edge \in \Edges(\Triangulation)$, we denote by $\Edge^\downarrow \in \Vertices(\Triangulation)$ and $\Edge^\uparrow \in \Vertices(\Triangulation)$ its backward and forward boundary vertex, respectively.

Let $\Polygon \colon \Vertices(\Triangulation) \to \AmbSpace$ be an embedded polygon in $\AmbSpace$, i.e.,
there is a piecewise linear embedding $\Curve \colon \Circle \to \AmbSpace$ such that $\Curve|_{\Vertices(\Triangulation)} = \Polygon$ and such that $\Curve$ maps $\Edge$ affinely onto the line segment connecting $\Polygon(\Edge^\downarrow)$ to $\Polygon(\Edge^\uparrow)$.
We denote by $\ConfSpace_\subT$ the set of such embedded polygons
which is an open set in the space of all closed polygons
with~$N$ edges. Since the latter is finite dimensional and isomorphic to~$(\AmbSpace)^N$,
we have $\XCg_\subT = \HCg_\subT = \YCg_\subT \cong (\AmbSpace)^N$. Likewise, we discretize the target spaces by $\XNg_\subT = \HNg_\subT = \YNg_\subT = \set{\lambda \colon \Edges(\Triangulation) \to \R} \times \AmbSpace \cong \R^N \times \AmbSpace$.
By $\EdgeLengths_\Polygon(I) \ceq \nabs{P(\Edge^\downarrow) - P(\Edge^\uparrow)}$, we denote the edge length of edge $I$.

\subsubsection*{Discrete energy}\label{sec:DiscreteEnergy}

There are several possibilities to discretize the Möbius energy $\Energy$.
A~very general approach employs simple quadrature rules and works for reparametrization-invariant energies $\cF$ of the form
$
	\textstyle
	\cF(\Curve)
	=
	\int_\Torus 
	F(\Curve)
	\, \varOmega_\Curve
$
with some energy density $F(\Curve) \colon \Torus \to \R$.
If, for a sufficiently smooth curve $\Curve$,
the integrand $F(\Curve)$ is not too singular around the diagonal of the integration domain $\Torus$, we have
\begin{align}
	\textstyle
	\cF(\Curve)
	\approx
	\sum_{\bar \Edge \cap \bar \OtherEdge = \emptyset}
	\int_{\Edge} \! \int_{\OtherEdge}  F(\Curve) \, \varOmega_\Curve.
	\label{eq:DiscreteEnergy1}	
\end{align}
Typically, the right hand side makes sense also if $\Curve$ is a polygonal line.
Indeed, cutting out the diagonal is somewhat necessary:
An elegant scaling argument in \cite[Figure 2.2]{Strzelecki:756025}) shows that the Möbius energy of a polygonal line with at least one nontrivial turning angle is infinite.

We may exploit parametrization invariance and pull back $F(\Curve)$ along the 
the \emph{local} pa\-ra\-me\-ter\-i\-za\-tion
$\Curve_\Edge \colon \intervalcc{0,1} \to \AmbSpace$, $\Curve_\Edge (s) \ceq P(\Edge^\downarrow) \, (1 - s) + s\, P(\Edge^\uparrow)$ 
and
$\Curve_\OtherEdge \colon \intervalcc{0,1} \to \AmbSpace$, $\Curve_\OtherEdge (t) \ceq P(\OtherEdge^\downarrow) \, (1 - t) + t \, P(\OtherEdge^\uparrow)$
to the unit square. Denoting the pullback by $F_{IJ}(P) \colon \intervalcc{0,1}^2 \to \R$, we have
\begin{align*}
	\textstyle
	\int_{\Edge} \! \int_{\OtherEdge}  F(\Curve) \, \varOmega_\Curve
	=
	\EdgeLengths_\Polygon(\Edge)\,
	\EdgeLengths_\Polygon(\OtherEdge)	
	\int_0^1 \! \int_0^1   F_{IJ}(\Polygon)(s,t) \, \dd s \, \dd t.
\end{align*}
So with a $k$-point quadrature rule $t_1, \dotsc, t_k \in \intervalcc{0,1}$, $\omega_1,\dotsc,\omega_k \in \R$, we 
may discretize $\cF$ by 
$
	\cF_\subT(\Polygon)
	\ceq
	\textstyle
	\sum_{\bar \Edge \cap \bar \OtherEdge = \emptyset}
	W_{IJ}(\Polygon)
$
with the \emph{local contributions}
\begin{align}
	W_{IJ}(\Polygon)
	\ceq
	\textstyle	
	\EdgeLengths_\Polygon(\Edge)\,
	\EdgeLengths_\Polygon(\OtherEdge)	
	\sum_{i=1}^k
	\sum_{j=1}^k	
	F_{IJ}(\Polygon)(t_i,t_j)\, \omega_i \, \omega_j.
	\label{eq:DiscreteEnergyW}
\end{align}
Applying this with $k =1$ to $F=E$ from \eqref{eq:EnergyDensity}, 
one is naturally lead 
to the \emph{vertex energy} ($t_1 = 0$, $\omega_1 = 1$) and
to the \emph{edge energy} ($t_1 = 1/2$, $\omega_1 = 1$) as proposed by Kusner and Sullivan in \cite{MR1470748}.
Scholtes proved in \cite{MR3268981} that the vertex energy for equilateral polygons 
$\Gamma$-converges towards $\Energy$ under refinement of partitions, i.e., for $\MaxRadius(\Triangulation) \to 0$, with respect to the $W^{k,q}$-topology,
$k\in\set{0,1}$, $q\in[1,\infty]$.
Roughly speaking, \hbox{$\Gamma$-con\-ver\-gence} implies that cluster points of minimizers of the discrete energies are minimizers of $\Energy$.
This result justifies the quite harsh variational crimes that one commits by choosing polygonal lines as discrete configurations. Although it is restricted to equilateral polygons (which was one of the reasons for us to include the edge length constraint), we deem it likely that it can be extended to non-equilateral polygons with a uniform bound on $\left.\max\EdgeLengths_\Polygon\middle/\min\EdgeLengths_\Polygon\right.$
as $h\to0$. At least our experiments indicate that the precise distributions of edge lengths does not matter.

We require also the derivative of the discrete energy. Similarly as in Section~\ref{sec:Energy}, the explicit dependence of $E$ on the geodesic distance $\varrho_\Curve$ causes problems: 
Without taking further measures, this would lead to the very high complexity of $\Omega(N^3)$ to assemble the derivative $D\Energy_{\subT}(\Polygon)$ for the vertex energy and edge energy.\footnote{For an optimization method that requires only the projected gradients and that enforces the edge length constraints in each iteration, the contribution of $D \varrho(\Curve)$ to $D \Energy(\Curve)$ can be ignored.}
This can be circumvented by utilizing the identity 
$\Energy(\Curve)
	=
	\textstyle
	4 \!+\!
	\int_\Torus 
	F(\Curve)
	\, \varOmega_\Curve$
with the integrand
\begin{align*}
	F(\Curve)
	\ceq
	\frac{\nabs{\triangle \Tangent_\Curve}^2}{2 \, \nabs{\triangle \Curve}^2}
	+
 	2
	\frac{
		\ninnerprod{\Tangent_\Curve \circ \prx,\Tangent_\Curve \circ \pry}
	}{
		\nabs{\triangle \Curve}^2
	}
	-
	2 
	\frac{
		\ninnerprod{\triangle \Curve,\Tangent_\Curve \circ \prx}\,\ninnerprod{\triangle \Curve,\Tangent_\Curve \circ \pry}
	}{
		\nabs{\triangle \Curve}^4
	}
	,
\end{align*}	
which was derived by Ishizeki and Nagasawa in \cite{MR3273894}.
For the sake of efficiency, we discretize with the midpoint rule, i.e., with $k =1$, $t_1 = 1/2$, and $\omega_1 = 1$.
For this $F$, the local contributions $W_{IJ}(\Polygon)$ depend only on the coordinates of the four points $\Polygon(I^\downarrow)$, $\Polygon(I^\uparrow)$, $\Polygon(J^\downarrow)$, and $\Polygon(J^\uparrow)$.
So the expression of the first and second derivative of $W_{IJ}$ with respect to these four points can once be computed symbolically and compiled into runtime-efficient libraries.
The first and second derivative of $\cF_\subT$ can then be assembled from $DW_{IJ}(P)$ and $D^2W_{IJ}(P)$ as a vector and a matrix of size $\AmbDim\, N$ and $(\AmbDim \, N) \times (\AmbDim \, N)$, respectively. Due to the nonlocal nature of the energy, the matrix $D^2\cF(\Polygon)$ is dense.

\subsubsection*{Discrete inner product}

Next we discretize the inner product $\RI[\ConfSpace]$ from \autoref{prop:MetricDefinition}. Let $U \colon \Vertices(\Triangulation) \to \AmbSpace$ and denote by $u \colon \Circle \to \AmbSpace$ piecewise linear interpolation.
For the computation of the local contribution of the edge pair $(\Edge,\OtherEdge)$ to the Gram matrix, we put
\begin{gather*}
	u_\Edge (s) \ceq U(\Edge^\downarrow) \, (1 - s) + s \, U(\Edge^\uparrow),
	\qand
	u_\OtherEdge (t) \ceq U(\OtherEdge^\downarrow) \, (1 - t) + t \, U(\OtherEdge^\uparrow).
\end{gather*}
The first two terms of $\ninnerprod{\RI[\ConfSpace] \, u, u}$  can now be discretized as follows:
\begin{gather*}
	\textstyle
	\sum_{\substack{
		\Edge \cap \OtherEdge = \emptyset
	}}
	\EdgeLengths_\Polygon(\Edge)\,\EdgeLengths_\Polygon(\OtherEdge)
	\bigabs{
	\frac{u_\Edge(\Edge^\uparrow)-u_\Edge(\Edge^\downarrow)}{\EdgeLengths_\Polygon(\Edge)}
	-
	\frac{u_\OtherEdge(\OtherEdge^\uparrow)-u_\OtherEdge(\OtherEdge^\downarrow)}{\EdgeLengths_\Polygon(\OtherEdge)}
	}^2	
	\sum_{i=1}^k
	\sum_{j=1}^k
	\frac{
		\omega_{i} \, \omega_{j}
	}{
		\nabs{\Curve_\Edge(t_i) - \Curve_\OtherEdge(t_j)}^2
	} 
	\;\;\text{and}
	\\
	\textstyle	
	\sum_{\substack{
		\Edge \cap \OtherEdge = \emptyset
	}}
	\EdgeLengths_\Polygon(\Edge)\,\EdgeLengths_\Polygon(\OtherEdge)
	\sum_{i=1}^k
	\sum_{j=1}^k
	\frac{	\abs{u_\Edge(t_i) - u_\Edge(t_j)}^2
	}{
		\nabs{\Curve_\Edge(t_i) - \Curve_\OtherEdge(t_j)}^2
	}
	\,
	E_{IJ}(\Polygon)(t_i,t_j)
	\, \omega_{i} \, \omega_{j}
	,	
\end{gather*}
where we employ the same quadrature rule as for the discrete Möbius energy.
In the presence of a barycenter constraint, we may simply omit the term $\ninnerprod{ \int_\Circle u \, \LineElement_\Curve,\int_\Circle u \, \LineElement_\Curve}$ without loosing definiteness of the inner product on $\ker(D \ConstraintMap_\subT(\Polygon))$.
By virtue of the polarization formula, this defines the Gram matrix uniquely,
leading to discrete bilinear forms $G_\Polygon = \RI[\ConfSpace_\subT] \at_\Polygon = \RJ[\ConfSpace_\subT] \at_\Polygon$.
The local  matrices are of size $(4 \,\AmbDim) \times (4 \,\AmbDim)$ ($\AmbDim$ coordinates for each of the four vertices belonging to the edge pair $(\Edge,\OtherEdge)$). They can be computed in parallel and added into the global matrix afterwards. 
The resulting global Gram matrix is a dense matrix of size $(\AmbDim \, N) \times (\AmbDim \, N)$.\footnote{In fact, the assembly can be sped up by 
first assembling  the $N \times N$-matrix  $\RI[\ConfSpace_\subT]$ for the case $\AmbDim =1$. This way, the local matrices have only size $4 \times 4$.
Afterwards, the $(\AmbDim \, N) \times (\AmbDim \, N)$ matrix can be obtained as block-diagonal matrix with 
$\AmbDim$ identical blocks of size $N \times N$.
}

\subsubsection*{Discrete constraints}
As for the constraints, we discretize $\ConstraintMap$ by
\begin{align*}
	\ConstraintMap_\subT(\Polygon)
	\ceq
	\Bigparen{
		\;
		\bigparen{ 
			\log ( \EdgeLengths_\Polygon(I))- \log(\ell_0(I))
		}_{\Edge \in \Edges(\Triangulation)}
	\;,\;
	\textstyle
	 \sum_{\Edge \in \Edges(\Triangulation)}
	 	\tfrac{1}{2} \, \EdgeLengths_\Polygon(\Edge) 
	 	\, 
	 	\nparen{ \Polygon(\Edge^\uparrow) + \Polygon(\Edge^\downarrow)}
		\;	 	
	},
\end{align*}
where $\ell_0 \colon \Edges(\Triangulation) \to \intervaloo{0,\infty}$ is a prescribed distribution of desired edge lengths, for example $\ell_0(\Edge)  = L \, \nabs{I}$.
Although restoring feasibility for the edge length constraint comes at a certain cost, it prevents edges from collapsing to points and from being overstretched in the course of optimization. The latter is crucial since the discrete energy is not exactly self-avoiding; it becomes singular only if \emph{quadrature points} approach each other. So overstretched edges make it more likely that the curve tries to form a self-intersection.

Some care should be given to the choice of the target edge lengths $\ell_0$.
A coarse mesh may not be sufficient to preclude self-intersections,
a very fine mesh is expensive as the
computational effort grows quadratically in the number of nodes.
As a rule of thumb, the distance between two neighboring vertices
of a polygon should be strictly smaller than the distance
between any other pairs of vertices.
In principle, it is also possible to drop the edge length constraints; instead one could introduce a global length constraint and one could handle short and long edges by adaptive edge split and edge collapse strategies. We refrained from opting for this route here for the sake of simplicity.

\subsection{Projected gradient}

Once the vector $\RY[\Energy_\subT][\Polygon] = D\Energy_\subT(\Polygon)$,  
and the matrices 
$\RI[\ConfSpace_\subT] \at_\Polygon$
and 
$\RX[\ConstraintMap_\subT][\Polygon] = \RY[\ConstraintMap_\subT][\Polygon] = D \ConstraintMap_\subT(\Polygon)$ have been assembled,
the projected gradient $u \ceq \grad_{\ConstraintMfld_\subT}(\Energy_\subT|_{\ConstraintMfld_\subT})\at_\Polygon$ can be obtained
by solving the following discrete analogue of the linear saddle point system \eqref{eq:SaddlePointEquation}:
\begin{align}
	\begin{pmatrix}
		\RJ[\ConfSpace_\subT]\at_\Polygon 
		&D\ConstraintMap_\subT(\Polygon)\dual \;
		\\
		D\ConstraintMap_\subT(\Polygon) 
		&0
	\end{pmatrix}
	\,
	\begin{pmatrix}
		u \\
		\lambda
	\end{pmatrix}
	=
	\begin{pmatrix}
		\eta\\
		0
	\end{pmatrix}
	\quad
	\text{with}
	\quad
	\eta =  D\Energy_{\subT}(\Polygon).
	\label{eq:DiscreteSaddlePointSystem}
\end{align}
We assemble the saddle point matrix as a dense, symmetric matrix with $(N \,\AmbDim + N  + \AmbDim)$ rows, and solve it via a dense $LU$-factorization.
Hence it costs roughly $\LandO(N^2 \AmbDim^2)$ for the assembly and a further $\LandO(N^3 \AmbDim^3)$ for the factorization. It is not surprising that this  is the most expensive part in the overall optimization process. 
We would like to point out that this can be sped up considerably by more sophisticated methods:
The assembly of the saddle point matrix can be avoided by assembling $D\ConstraintMap_\subT(\Polygon)$ as a sparse matrix and by compressing $\RJ[\ConfSpace_\subT]\at_\Polygon$ in a hierarchical matrix data structure that is efficient for fast matrix-vector multiplication.
Similar techniques can be employed to approximate $\Energy_\Triangulation(\Polygon)$ and $D\Energy_\Triangulation(\Polygon)$ in subquadratic time, but all this is beyond the scope of the present work.

\subsection{Restoring feasibility and time step size rules}
Suppose that $\ConstraintMap_\subT(\Polygon) = 0$ and that $u$ is a feasible search direction, i.e., $D\ConstraintMap_\subT(\Polygon) \,u  = 0$.
The constraint mapping $\ConstraintMap_\subT$ is Lipschitz continuously differentiable.
Hence provided that 
the step size $\tau > 0$ is sufficiently small, the modified Newton method
\begin{align}
	\OtherPolygon_0 = \Polygon + \tau \, u,
	\qquad
	\OtherPolygon_{i+1} = \OtherPolygon_i - D\ConstraintMap_\subT(\Polygon)\pinv \; \ConstraintMap_\subT(\OtherPolygon_i)
	\quad \text{for $i \in \N$}
	\label{eq:ModifiedNewtonBackProjection}
\end{align}
converges quickly to a point $\OtherPolygon_\infty$ that satisfies $\ConstraintMap_\subT(\OtherPolygon_\infty) = 0$. Here $D\ConstraintMap_\subT(\Polygon)\pinv$ denotes the Moore-Penrose pseudoinverse with respect to the inner product $G_\Polygon$ and we utilize \autoref{cor:PseudoinverseOrthoprojector} to evaluate it.\footnote{We employ the modified Newton method (instead of Newton's method) because the saddle point matrix from \eqref{eq:DiscreteSaddlePointSystem} is already factorized, 
so that evaluating $D\ConstraintMap_\subT(\Polygon)\pinv \, \tilde u$ on a given vector $\tilde u$ can be performed quite inexpensively with \autoref{cor:PseudoinverseOrthoprojector}.
Alternatively, also every other scheme for solving $\ConstraintMap_\subT(\OtherPolygon) = 0$ can be employed.}
For a given descending direction $u$, we may apply backtracking line search to find a suitable step size $\tau > 0$:
If the residual $\ConstraintMap_\subT(\OtherPolygon_i)$ is smaller than a prescribed tolerance after a small, prescribed number of iterations, then the point $\OtherPolygon_i$ may serve as the next iterate of the optimization method.
Otherwise we shrink $\tau$ and restart the modified Newton method. 
By shrinking $\tau$ even further, if necessary, we can also achieve that $\OtherPolygon_i$ 
satisfies the Armijo condition
$\Energy_\subT(\OtherPolygon_i) \leq \Energy_\subT(\Polygon) + (\tau/2) \, D\Energy_\subT(\Polygon) \, u$. 
An initial guess for $\tau$ can be obtained, e.g., by collision detection (see, e.g., \cite{doi:10.1111/1467-8659.t01-1-00587}): 
One determines the smallest step size $\tau_*$ such that $\Polygon + \tau_* \, u$ has a self-intersection and starts the backtracking procedure with, e.g., $\tau = \tfrac{2}{3} \tau_*$.
By utilizing suitable space partitioning data structures, this collision detection can be performed in subquadratic time. 
However, we simply cycled over all $O(N^2)$ edge pairs because its runtime is proportional to the runtime of $D\Energy_\subT(\Polygon)$.

\subsection{Optimization methods employed in \protect\autoref{fig:Benchmark}}\label{sec:OptimizationMethods}

\paragraph*{Feasible methods}
Projected $L^2$-, $W^{1,2}$-, $W^{3/2,2}$-, and $W^{2,2}$-flows were simulated both with explicit and implicit time integration schemes. 
We followed the approach above, only replacing $\RJ[\ConfSpace_\subT]\at_\Polygon$ by the Riesz operator corresponding to the particular choice of metric. Armijo backtracking line search automatically determines a stable step size.
For the implicit integration of the $L^2$-gradient flow, we employ the backward Euler method. Since it is not unconditionally stable, Armijo backtracking has to be employed also here. Because backtracking requires the implicit equations to be solved again, this is particularly expensive.

The employed trust region method is a blend of the method from \cite{MR909071} with the
two-dimen\-sional subspace method from \cite{MR772882} (without computing the lowest eigenvalues):
The next iterate is found by minimizing a quadratic model in a trust region within a low-dimensional subspace spanned by the current projected gradient, the projection of the previous gradient onto the current tangent space, and the Newton search direction -- provided the current projected gradient is shorter than a given threshold.
This means that the optimization is mostly driven by gradient and momentum; and the Hessian is utilized only in the end phase of optimization.
Shrinkage and expansion of the trust region is handled as usual, but the radius is of course to be interpreted with respect to the employed inner product.

\paragraph*{Infeasible methods}
In order to compare also to unconstrained optimization methods, 
we applied them to an analogous discretization of the penalized energy
\begin{align*}
	\textstyle
	\Energy_\alpha(\Curve)
	\ceq
	\Energy(\Curve) 
	+ 
	\alpha \,
	\nnorm{\ConstraintMap(\Curve)}_{L^2}^2
	=
	\Energy(\Curve) 
	+ 
	\alpha  \int_\Circle 
	\, \log( \nabs{\Curve'(t)}/L )^2 \, L \, \dd t,
\end{align*}
whose penalty can be interpreted as Hencky's stretch energy.
The optimization methods
were made aware of this penalty by using the metric	
$\RJ[\ConfSpace]\at_\Curve + \alpha \, D\ConstraintMap(\Curve)\dual \, \RI[\Lebesgue[2]]  D\ConstraintMap(\Curve)$ to compute gradients, where $\cI_{\Lebesgue[2]}$ denotes the Riesz operator of $\Lebesgue[2][][\Circle][\R]$.\footnote{This had a negative effect on methods based on the $L^2$-metric, so we omitted this extension in that case.}
As nonlinear conjugate gradient method, we employed the Polak-Ribière method ``with automatic reset'' (method $\mathrm{PR}_+$ in \cite[Section~5.2]{MR2244940}).
L-BFGS was implemented with history length $30$ and as described in \cite[Section~7.2]{MR2244940}. The only difference is that we replace the initial guess for the inverse Hessian by the inverse of the \emph{current} metric (because using a single initial guess turned out to be less efficient).\footnote{We are well-aware that this ad-hoc modification might not superlinearly convergent.}
As for Nesterov's accelerated gradient method (acc. grad.), we followed \cite{MR701288}, but added collision detection to truncate the step sizes (in both steps of the method). Moreover, as suggested in \cite{MR3348171}, we reset the momentum to $0$ whenever an increase of the objective was observed.\footnote{We are also aware that Nesterov's method was designed for convex optimization problems; as a heavy ball method it still serves its purpose to push the optimization through shallow regions of the energy landscape.}  
All these methods were complemented with a line search that tries to find a weak Wolfe-Powell step size.

%% file: Appendix.tex
\section{Auxiliaries}

We require some technical results for \SoboSlobo\ spaces on the circle $\Circle$.
Typically, such statements are formulated on $\R^n$ or for sufficiently smooth domains $\varOmega \subset \R^n$,
but standard techniques allow one to port them also to smooth manifolds such as $\Circle$.
Proofs for the following two results can be found, e.g., in \cite[Theorem~5.3.6/1~(ii)]{MR1419319}
and
\cite[2.8.2, Eq.~(19)]{MR503903}.

\begin{lemma}[Chain rule]\label{lem:chain}\label{lem:ChainRule}
Let $\varOmega\subset\R^{n}$ be a bounded $C^{\infty}$-domain,
$\sigma \in \intervaloc{0,1}$, $p \in \intervalcc{1,\infty}$.
If $\psi:\R\to\R$ is Lipschitz continuous
with Lipschitz constant $\varLambda \geq 0$
and $f \in \Sobo[\sigma,p](\varOmega;\R)$
then $\psi \circ f \in \Sobo[\sigma,p](\varOmega;\R)$ and we have
$
	\nseminorm{\psi \circ f}_{\Sobo[\sigma,p]} 
	\leq 
	\varLambda \, \nseminorm{f}_{\Sobo[\sigma,p]}
	.
$
\end{lemma}

\begin{lemma}[Sobolev embedding]\label{lemma:Wembedding}
Let $\varOmega\subset\R^{n}$ be a bounded $C^{\infty}$-domain.
If $s_{0},s_{1}\in\R$, $p_{0},p_{1}\in \intervaloo{1,\infty}$ satisfy
$s_{0} \geq s_{1}$ and $s_{0}-\frac n{p_0} \ge s_{1}-\frac n{p_1}$
then the embedding
$
	W^{s_{0},p_{0}}(\varOmega ;\R) \hookrightarrow  W^{s_{1},p_{1}}(\varOmega ;\R)
$ is well-defined and continuous.
\end{lemma}

\begin{lemma}[Vector-valued Sobolev embedding]\label{lem:vectorWembedding}
Let $\varOmega\subset\R^{n}$ be a bounded $C^{\infty}$-domain.
If $s \in \intervaloc{0,1}$, $p,q \in \intervaloo{1,\infty}$ with $s-\frac{n}{p} = - \frac{n}{q} <0$
then,
for any Banach space $X$, there is a continuous Sobolev embedding
$
	\Sobo[s,p][][\varOmega][X] \hookrightarrow \Lebesgue[q][][\varOmega][X].
$
\end{lemma}

\begin{proof}
We follow the argumentation in Theorem 5.1 from \cite{MR3731026}:
The norm $\psi \ceq \nnorm{\cdot}_{X}$ is Lipschitz with Lipschitz constant $1$. 
Utilizing the Sobolev embedding for \emph{scalar} functions  (\autoref{lemma:Wembedding}) and for $s \in \intervaloc{0,1}$ along with the chain rule (\autoref{lem:ChainRule}), we obtain 
$
	\nnorm{u}_{\Lebesgue[q][][\varOmega][X]}
	=
	\nnorm{\psi \circ u }_{\Lebesgue[q][][\varOmega][\R]}
	\leq
	C \, \nnorm{\psi \circ u }_{\Sobo[s,p][][\varOmega][\R]}
	\leq
	C \, \nnorm{u }_{\Sobo[s,p][][\varOmega][\R]}
$
where $C>0$ is the Sobolev constant of the embedding for scalar functions.
\end{proof}

 The following is essentially a fractional Leibniz rule.
 The first proof seems to be due to Zolesio~\cite{MR0500121}
 who even considers the more general concept of Besov spaces.
 For \SoboSlobo spaces stronger requirements apply
 compared to the case of Triebel--Lizorkin spaces, see
 Runst and Sickel~\cite[Theorem 4.3.1/1~(i),
 Equation~(11)]{MR1419319}.
 We refer to the survey of Behzadan and Holst~\cite{2015arXiv151207379B}
 for further information.

\begin{lemma}[Product rule]\label{lem:MultiplicationLemma}
Let $\varOmega\subset\R^{n}$ be a bounded $C^{\infty}$-domain,
$\sigma_i \in \intervaloo{0,1}$, $p_i \in \intervalcc{1,\infty}$,
for $i \in \set{1,2}$.
Let $b \colon \R^{m_{1}} \times \R^{m_{2}} \to \R^{m}$ be a bounded bilinear mapping.
Then the bilinear mapping
$
	B \colon \Sobo[\sigma_1,p_1][][\varOmega][\R^{m_{1}}] \times \Sobo[\sigma_2,p_2][][\varOmega][\R^{m_{2}}]
	\to 
	\Sobo[\sigma_2,p_2][][\varOmega][\R^{m}]
$, $B(u_1,u_2)(x) = b(u_1(x),u_2(x))$
is well-defined and continuous if  at least one of the following two conditions is satisfied:
\begin{enumerate}
	\item[\textup{(i)}] $\conf{1}>0$, $\conf{2}>0$, $\sigma_1 \geq \sigma_2$, and $\conf{1} \geq \conf{2}$,
	\item[\textup{(ii)}] $\conf{1}>0$, $\sigma_1 > \sigma_2$, and $\conf{1} > \conf{2}$.
\end{enumerate}
\end{lemma}

The following ``Schauder lemma'', communicated to us by Thorsten Hohage, helps us in \autoref{sec:Metric} to show that the Riesz isomorphism of the metric is invertible.
It does so by allowing us to play invertibility back to an ``elliptic estimate''. See, e.g., \cite[Appendix A, Proposition 6.7]{MR2744150} for a proof.

\begin{lemma}[Schauder lemma]\label{lem:SchauderLemma}
Let $X$ be a Banach space, let $Y$ and $Z$ be normed spaces, and let
$A \colon X \to Y$ be a continuous, injective, linear operator.
Suppose that there exists a $\tilde C \geq 0$ and a compact, linear operator $K \colon X \to Z$ into a further Banach space $Z$ such that
$\nnorm{u}_X \leq \tilde  C \,\paren{  \nnorm{A \, u}_Y + \nnorm{K \, u}_Z}$
holds for all $u \in X$. 
Then $A$ has closed image and there is a further constant $C \geq 0$ such that
\begin{align}
	\nnorm{u}_X \leq C \, \nnorm{A \, u}_Y
	\quad
	\text{holds for all $u \in X$.}
	\label{eq:SchauderEstimate}
\end{align}
\end{lemma}

%% file: Acknowledgments.tex
\section*{Acknowledgments}

\input{Funds}
\input{FurtherAcknowledgments}

%% file: Funds.tex
This work was partially funded 
by  a postdoc fellowship of the German Academic Exchange Service (H.\,S.),
by DFG-Grant RE \hbox{3930/1--1} (Ph.\,R.),
and
by  DFG-Project 282535003: \emph{Geometric curvature functionals: energy landscape and discrete methods} (both authors).

%% file: FurtherAcknowledgments.tex
Both authors wish to thank Armin Schikorra and Thorsten Hohage for fruitful discussions.